\documentclass[12pt,reqno]{amsart}
\usepackage{amsmath} 

\usepackage{amsthm}
\usepackage{amssymb}
\usepackage{graphics}
\usepackage{latexsym}

\numberwithin{equation}{section}
\newtheorem{thm}{Theorem}[section]
\newtheorem{prop}[thm]{Proposition}
\newtheorem{lem}[thm]{Lemma}
\newtheorem{cor}[thm]{Corollary}

\theoremstyle{remark}
\newtheorem{rem}[thm]{Remark}
\newtheorem{defn}{Definition}

\newcommand{\BBB}{\mathbb}
\newcommand{\R}{{\BBB R}}
\newcommand{\Z}{{\BBB Z}}
\newcommand{\T}{{\BBB T}}

\newcommand{\N}{{\BBB N}}
\newcommand{\C}{{\BBB C}}

\newcommand{\lec}{{\ \lesssim \ }}
\newcommand{\gec}{{\ \gtrsim \ }}

\newcommand{\al}{\alpha}
\newcommand{\be}{\beta}
\newcommand{\ga}{\gamma}

\newcommand{\vp}{\varphi}

\newcommand{\e}{\varepsilon}

\newcommand{\p}{\partial}
\newcommand{\la}{\lambda}
\newcommand{\La}{\Lambda}

\newcommand{\de}{\delta}

\newcommand{\supp}{\operatorname{supp}}

\newcommand{\dis}{\displaystyle}
\newcommand{\mT}{\mathcal{T}}

\newcommand{\EQ}[1]{\begin{equation} \begin{split} #1
 \end{split} \end{equation}}
\newcommand{\EQS}[1]{\begin{align} #1 \end{align}}

\newcommand{\EQQ}[1]{\begin{equation*} \begin{split} #1
 \end{split} \end{equation*}}

\newcommand{\ti}{\widetilde}
\newcommand{\ha}{\widehat}

\title[LWP of fifth order mKdV type equations on the torus]
{Cancellation properties and unconditional well-posedness for fifth order modified KdV type equations with periodic boundary conditions
}
\author[T. K.  Kato]{Takamori Kato}
\author[K. Tsugawa]{Kotaro Tsugawa}
\address[T.K. Kato]{Mathematical Science Course, Faculty of Science and Engineering, Saga University, Saga, 840-8502, 
Japan}
\address[K. Tsugawa]{Department of Mathematics, Faculty of Science and Engineering, Chuo University,
Bunkyo-ku, Tokyo, 112-8551, Japan}
\email[T.K. Kato]{tkkato@cc.saga-u.ac.jp}
\email[K. Tsugawa]{tsugawa@math.chuo-u.ac.jp}
\keywords{fifth order modified KdV, normal form, well-posedness, Cauchy problem, low regularity, unconditional}

\begin{document}

\begin{abstract}
We prove the unconditional well-posedness result for fifth order modified KdV type equations in $H^s(\T)$ when $s \ge 3/2$, which includes non-integrable cases.
By the conservation laws, we also obtain the global well-posedness result when $s = 2$, which also includes non-integrable cases.
The main idea is to employ the normal form reduction and a kind of cancellation properties to deal with the derivative losses. 
\end{abstract}

\maketitle
\setcounter{page}{001}

%%%%%%%%%%%%%%%%%%%%%%%%%%%%%%%%%%%%%%%%%%%%%%%%%%%%%%%%%%%%%%%%%%%%%%%%%%%%%%%
%%%%%%%%%%%%%%%%%%%%%%%%%%%%%%%%%%%%%%%%%%%%%%%%%%%%%%%%%%%%%%%%%%%%%%%%%%%%%%%
%%%%%%%%%%%%%%%%%%%%%%%%%%%%%%%%%%%%%%%%%%%%%%%%%%%%%%%%%%%%%%%%%%%%%%%%%%%%%%
\section{Introduction}

We consider the Cauchy problem of the fifth order modified KdV type equations:
\EQS{
&\p_t u +\p_x^5 u +\al(\p_x u)^3+\be\p_x(u(\p_xu)^2)+\ga\p_x(u\p_x^2( u^2))+6\de \p_x (u^5)=0,\label{5mKdV1}\\
&u(0,x)=\vp(x)\in H^s(\T)\label{initial}
}
where $\al, \be, \ga, \de \in\R$, $u(t,x): \R\times \T\to \R$ and $\vp(x): \T \to \R$.
%Note that \eqref{1eq1} is invariant for the scaling transformation $u(t,x) \mapsto \la u(\la^5 t, \la x)$ with $\la>0$ if we consider the case $x\in\R$. 
%The scaling critical exponent of the Sobolev norm is $-1/2$, which is equal to the critical exponent for the modified KdV equation.
There are many results on the well-posedness of the Cauchy problem for the KdV equations and the modified KdV equations.
However, relatively few results are known for the fifth-order KdV and modified KdV type equations.
This is because the presence of strong singularities in the nonlinear terms makes the problem difficult. 
See \cite{BKV}, \cite{GKK}, \cite{KM}, \cite{TKK}, \cite{KTT}, \cite{KP}, \cite{Kwa}, \cite{Kwo}, \cite{Mc} and \cite{Ponce} for the fifth order KdV type equations.
Here, we state the known results for the fifth order modified KdV type equations.
For the case of $x\in \R$, in \cite{Kwo2}, Kwon proved the local well-posedness in $H^s$ for $s\ge 3/4$ by the Fourier restriction norm method, which was introduced by Bourgain in \cite{Bo}.
When $x\in \T$, the linear part does not have any smoothing effects and that makes the problem more difficult.
For the case of $x\in \mathbb{T}$, in \cite{Kwa2}, Kwak proved the local well-posedness in $H^s$ for $s>2$ by combining the modified energy method and the short time Fourier restriction norm method when the initial data are on a level set and the conditions $A_0$, $A_1$ and $A_2$ below hold, which is the integrable case.
It was refined by Kwak and Lee in \cite{Kwa3} to $s \ge 2$, which include the global well-posedness in $H^s$ for $s = 2$.
There are few studies of \eqref{5mKdV1} without any conditions on $\al, \be, \ga$ and $\de$.
The second author studied general fifth order dispersive equations on $\T$ in \cite{T} and
proved the necessary and sufficient conditions on the nonlinear terms to assure the local well-posedness for sufficiently large $s$, which include the local well-posedness for \eqref{5mKdV1} without any condition on $\al, \be, \ga$ and $\de$.
In the present paper, we are especially interested in low regularity solutions.
Our main results are as follows.
%%%%%%%%%%%%%%%%%%%%%%%%%%%%%%%%%%%%%%%%%%%%%%%%%%%%%%%%%%%%%%%%%%%%%%%%%%%%%%%%%%%%%%%%%%%%%%
\begin{thm} \label{thm_LWP}
Let $s \geq 3/2$ and $\alpha= \beta$ or $\gamma=0$. Then, for any $\varphi \in H^s (\mathbb{T})$, 
there exist $T=T(\| \varphi \|_{H^{s}}) >0$ and a unique solution $u \in C([-T,T]: H^{s} (\mathbb{T}))$ to 
\eqref{5mKdV1}--\eqref{initial}. 
Moreover, the solution map $H^s(\mathbb{T}) \ni \varphi \mapsto u \in C([-T,T]: H^s (\mathbb{T}))$ is continuous.  
\end{thm}
%%%%%%%%%%%%%%%%%%%%%%%%%%%%%%%%%%%%%%%%%%%%%%%%%%%%%%%%%%%%%%%%%%%%%%%%%%%%%%%%%%%%%
\begin{rem}
\eqref{5mKdV1} makes sense as a distribution for $u \in C([-T,T]: H^{s} (\mathbb{T}))$ and $s \ge 7/6$ since $(\p_x u)^3 \in L^1(\mathbb{T})$
by the Sobolev inequality.
Moreover, if $\alpha=0$ then it makes sense as a distribution for $s \ge 1$ since $\p_x(u\p_x^2( u^2))=\p_x^2(u\p_x( u^2))-\p_x(\p_x u\p_x( u^2)) \in H^{-2}(\mathbb{T})$.
The condition $s \geq 3/2$ in Theorem \ref{thm_LWP} arises from the so-called high-high-high interaction for the cubic terms.
A similar issue arises for the modified KdV equation when $s<1/2$.
The difficulty was overcome by Takaoka and Tsutsumi in \cite{TaTs}
and they proved the local well-posedness for $s>3/8$.
Nakanishi, Takaoka and Tsutsumi in \cite{NTT}
improved it to have the well-posedness for $s>1/3$ and the existence of solutions for $s>1/4$.
Moreover, Molinet, Pliod and Vento in \cite{MPV} proved the unconditional well-posedness for $s \ge 1/3$.
Therefore, we conjecture that the condition $s \ge 3/2$ in Theorem \ref{thm_LWP} could be improved if we combine their method and our idea below.
\end{rem}
Here, we observe some conservation quantities. Define $\int_{\T} f \, dx:= \frac{1}{2 \pi} \int_0^{2\pi} f(x) \, dx$ for 
$2 \pi$-periodic function $f$. Put
\EQQ{
&E_0(f):=\int_{\mathbb{T}} f \, dx, \qquad E_1(f) :=\int_{\mathbb{T}} f^2 \, dx, \qquad
E_2(f):=\int_{\mathbb{T}} \frac{1}{2} (\p_x f)^2-\frac{\be+3\ga}{30} f^4 \, dx, \\
&E_3(f):= \int_{\mathbb{T}} \frac{1}{2} (\p^2_x f)^2- \ga f^2 (\p_x f )^2 + \de f^6 \, dx, \\
&A_0: \al=0,\qquad A_1: \al=\be,\qquad A_2: (\al-3\be+6\ga)(\be+3\ga)=450\de.
}
Let $u$ be a solution of \eqref{5mKdV1}--\eqref{initial}.
Then, for each $j=0,1,2$, we have the conservation law
$E_j(\vp)=[E_j(u)](t)$ in a formal sense when assumption $A_j$ holds.
We have $E_3(\vp)=[E_3(u)] (t)$ when both $A_0$ and $A_1$ hold.
In this case, \eqref{5mKdV1} can be regarded as the Hamiltonian PDE
$\p_t u =-\p_x E_3' (u)$. 
When all assumptions $A_0, A_1, A_2$ hold, it follows that
$\al=\be=0, \ga^2=25\de$.
In this case, \eqref{5mKdV1} is called the fifth order modified KdV equation,
which is a completely integrable system and have infinitely many conserved quantities including $E_0, E_1,E_2$ and $E_3$.
Note that by taking $u(t,x) \mapsto u(\la^{-5/2} t, \la^{-1/2}x)$ with $\la=|\ga|/5$, 
the coefficients of this case are normalized as $(\al,\be,\ga,\de)=(0,0,5,1)$ or $(0,0,-5,1)$. 
When $A_0$ and $A_1$ hold, by the conserved quantities $E_1$ and $E_3$, we have the following global result as a corollary of Theorem \ref{thm_LWP}.
\begin{cor}
Let $s = 2$, $\vp\in H^s(\T)$ and $\al=\be=0$. Then, the solution obtained in Theorem \ref{thm_LWP} can be extended to the solution on $t\in (-\infty,\infty)$.
\end{cor}

Our main tool to prove Theorem \ref{thm_LWP} is the normal form reduction.
It was introduced by Shatah in \cite{Sh} to study the existence of the global solution of quadratic nonlinear Klein-Gordon equations 
for small initial data.
It was used to recover derivative losses by Takaoka-Tsutusmi in \cite{TaTs} in the study of the well-posedness of the modified KdV equation 
and 
by Babin-Ilyin-Titi in \cite{BIT} in the study of unconditional uniqueness of the KdV equation. 
For the studies of unconditional uniqueness by the normal form reduction, 
see \cite{GKO}, \cite{Kn1}, \cite{Kn2}, \cite{Kn3}, \cite{Kn4}, \cite{Kn5}, \cite{KO}, \cite{KOY}, \cite{MoP} and \cite{MoY}.
In our proof, main difficulty comes from the loss of derivatives in the nonlinear terms.
We can extract a kind of smoothing effect for the non-resonant parts in the nonlinear terms by the normal form reduction.
Therefore, we need to eliminate the resonant parts with derivative losses in the nonlinear terms.
When $\al=\be$ or $\ga=0$, it follows that $\ga E_1(\vp)= \ga [E_1(u)](t)$ for $u\in C([-T,T]:H^{3/2}(\T))$ in the rigorous sense (see Lemma~\ref{lem_E1}). 
Thus, \eqref{5mKdV1} can be rewritten into
\EQ{
&\p_t u +\p_x^5 u +2\ga E_1(\vp)\p_x^3 u +\left\{(3 \alpha + \beta - 2 \gamma ) \int_{\mathbb{T}} (\p_x u)^2 \, dx
 + 30 \delta  \int_{\mathbb{T}} u^4 \, dx \right\}\p_x u\\
&= J_1(u) + J_2 (u)  +J_3 (u)\label{5mKdV21}
}
where 
 \begin{align*}
J_1(u)= & -30 \delta \Big( u^4- \int_{\mathbb{T}} u^4 \, dx \Big) \p_x u, \\ 
J_2(u)=& - \alpha (\p_x u)^3 - \beta \p_x (u (\p_x u)^2 ) + (3 \alpha + \beta) \int_{\mathbb{T}} (\p_x u)^2 \, dx \, \p_x u, \\
J_3(u)= & - \gamma \p_x (u \p_x^2 (u^2)  ) + 2 \ga  \int_{\mathbb{T}} u^2 \, dx \, \p_x^3 u 
-2 \ga  \int_{\mathbb{T}} (\p_x u)^2 \, dx \, \p_x u.
\end{align*}
Note that the resonant parts with one derivative loss in $J_1(u)$ have been removed (see \eqref{EE5}).
Roughly speaking Lemmas \ref{Le1}, \ref{Le2}, \ref{Le3} and \ref{Le4} mean
that only one derivative loss can be recovered by applying the normal form reduction once
because $\Phi_f^{(N)} \gec k_{\max}^4$ holds for non-resonant parts, the nonlinear terms include the third derivatives and $4-3=1$.
Since $J_1(u)$ has only one derivative loss, the loss can be recovered by applying the normal form reduction once.
$J_2(u)+J_3(u)$ still have resonant parts of high-high-high interaction (see the second term of \eqref{EE4}).
However, they are harmless since $H^s$ norms of them are estimated by $\|u\|_{H^s}^3$ when $s \ge 3/2$.
The non-resonant parts of $J_2(u)+J_3(u)$ (see the first term of \eqref{EE4}) have three derivative losses. 
Therefore, we need to apply the normal form reduction three times.
After we apply it once, quintic terms appear.
Most of the resonant parts with derivative losses of quintic terms can be cancelled by symmetry (see Proposition \ref{prop_res3}).
However, the resonant part represented by the symbol $M_{9,\vp}^{(5)}$ has one derivative loss and it cannot be cancelled by symmetry.
To overcome this difficulty, we add $J_4(u)$ as below.
Then, the derivative loss of $M_{9,\vp}^{(5)}$ can be cancelled by $M_{1,\vp}^{(5)}$, which comes form $J_4(u)$ (see Proposition \ref{prop_res1}) .
After applying the normal form reduction for the non-resonant parts of the quintic terms, septic terms with one derivative loss appear.
The resonant part of them, which represent $M_{6,\vp}^{(7)}$ is cancelled by symmetry (see Proposition \ref{prop_res2}).
For the non-resonant parts of septuple terms, we apply the normal form reduction.
Then, all derivative losses are recovered.
Therefore, all terms can be estimated by variants of the Sobolev inequalities
and we use only the Sobolev norm and do not use the Strichartz estimates or the Fourier restriction norm.
For that purpose, we rewrite \eqref{5mKdV21} into
\EQ{
\p_t u +\p_x^5 u +2\ga E_1(\vp)\p_x^3 u +K(u)\p_x u =J_1(u)+ J_2 (u)  +J_3 (u)+ J_4(u)\label{5mKdV2}
}
where 
 \begin{align*}
J_4(u)  = & -\frac{4}{5} \gamma^2 \Bigl\{ \Big(  \int_{\mathbb{T}} u^4 \, dx \Big) 
- \Big( \int_{\mathbb{T}} u^2 \, dx  \Big)^2 \Bigr\} \p_x u,\\
K(u)  = &(3 \alpha + \beta - 2 \gamma ) \int_{\mathbb{T}} (\p_x u)^2 \, dx
 + \frac{4}{5} \gamma^2 E_1^2(u) + \Big( 30 \delta- \frac{4}{5} \gamma^2  \Big)  \int_{\mathbb{T}} u^4 \, dx.
\end{align*}
% and $\lambda = \frac{1}{15} \bigl\{(3 \alpha + \beta - 2 \gamma) (\beta+ 3 \gamma) + 450 \delta  \bigr\}$.
Note that $K(u)$ does not depend on $x$.
Thus, by the changing variable 
\EQQ{
 u(t,x) \mapsto u\Big(t, x+ \int_0^t [K(u)](t')~dt'\Big),
}
\eqref{5mKdV2} is rewritten into 
\begin{align} \label{5mKdV3}
\p_t u + \p_x^5 u + 2 \gamma E_1(\varphi) \p_x^3 u = J_1(u)+ J_2(u) +J_3(u)+J_4(u).  
\end{align}
This is the main idea in this paper.
The same idea was used for the fifth order KdV type equations in \cite{KTT}.
For a more detailed explanation, see the introduction of \cite{KTT}.

Note that the case $\al \neq \be$ and $\ga \neq 0$ is excluded in Theorem \ref{thm_LWP}.
In this case, 
the phase function $\Phi_{u}^{(N)}$ defined in Section 2 depends on $t$
since we do not have the conservation low $E_1(\vp)=[E_1(u)](t)$.
Therefore, the formulas by the normal form reduction, that is \eqref{NF21} in Proposition \ref{prop_NF2} and \eqref{eq20} in Proposition \ref{prop_req1}
become more complicated and we need more regularity of the solution with respect to $t$ if we apply our method to this case.
This means that the assumption on $s$ must be higher.
Therefore, we do not discuss the case in the present paper.

\if0
Now, any resonant parts with no derivative loss do not exist in $J_1(u)$, $J_2(u)$ and $J_3(u)$.
Note that the standard iteration argument by the $X^{s,b}$ space which is introduced by Bourgain in \cite{Bo} does not work even for \eqref{5mKdV3}, since $J_3(u)$ have the third derivatives and the trilinear estimate for it fail by the so called high-low-low interaction.
We will use the normal form reduction to recover the derivative loss.
Since $J_1(u)$ has only one derivative loss, we only need to apply the normal form reduction once.
On the other hand, $J_2(u)$ and $J_3(u)$ have the third derivatives. Therefore, we need to apply it three times.
After we apply the normal form reduction to $J_3(u)$ the first time, some quintic terms with second derivatives appear.
For the resonant parts of them, which correspond to the symbol defined in Proposition, we use a kind of cancellation
\fi

Finally, we introduce some notations.
We write $k_{i,i+1,\ldots,j}$ to mean $k_i+k_{i+1}+\cdots+k_j$ for integers $i$ and $j$ satisfying $i < j$. 
$k_{\max}$ is denoted by $k_{\max}:= \max_{1 \le j \le N} \{ |k_j| \}$ and 
$\text{sec}_{1 \le j \le N} \{  |k_j|  \}$ is denoted by the second largest number among $|k_1|, \dots, |k_N|$. 
We will use $A\lesssim B$ to denote an estimate of the form $A \le CB$ for some positive constant $C$ and write $A \sim B$ 
to mean $A \lesssim B$ and $B \lesssim A$. 
If $a \in \R$, $a+$ will denote a number slightly greater than $a$.  
$\| \cdot \|_{L_T^{\infty} X }$ is denoted by $\| \cdot \|_{L_T^{\infty} X}:= \sup_{t \in [-T, T]} \| \cdot \|_{X}$ for a Banach space $X$. 
For $s \in \R$, $l_s^2$ is denoted by $l_s^2 := \{ f: \Z \to \C: \|  f \|_{l_s^2}:= \| \langle \cdot  \rangle^s f \|_{l^2} < \infty \}$.

%%%%%%%%%%%%%%%%%%%%%%%%%%%%%%%%%%%%%%%%
%%%%%% Acknowledgement %%%%%%%%%%%%%%%%%%
%%%%%%%%%%%%%%%%%%%%%%%%%%%%%%%%%%%%%%%%
\section*{Acknowledgement}
We would like to thank Professor Nobu Kishimoto for his valuable comments.
% The first author was supported by JSPS KAKAEHI Grant Number JP18K13442.
%and the second author was supported by JSPS KAKENHI Grant Number ****. %% JP17K05316. 

%%%%%%%%%%%%%%%%%%%%%%%%%%%%%%%%%%%%
%%%%%%%%%%%%%%%%%%%%%%%%%%%%%%%%%%%%
%%%%%%%%%%%%%%%% Section 2 %%%%%%%%%%%%%%
%%%%%%%%%%%%%%%%%%%%%%%%%%%%%%%%%%%%
%%%%%%%%%%%%%%%%%%%%%%%%%%%%%%%%%%%%

\section{notations and preliminary lemmas}

%%%%%%%%%%%%%%%%%%%%%%%%%%%%%%%%%%%%%
In this section, we prepare some lemmas to prove main theorem. 
First, we give some notations. 
For a $2 \pi$-periodic function $f$ and a function $g$ on $\Z$, we define the Fourier transform and the inverse Fourier transform by 
\begin{align*}
(\mathcal{F}_x  f)(k):= \ha{f}(k):= \int_{\T} e^{-ixk} f(x) \, dx, \hspace{0.5cm}
(\mathcal{F}^{-1} g) (x):= \sum_{k \in \Z} e^{ixk} g(k). 
\end{align*} 
Then we have 
\begin{align*}
f= \mathcal{F}^{-1} (\mathcal{F}_x f), \hspace{0.5cm} 
\| f \|_{L^2}:= \Big( \int_{\T} |f(x)|^2 \, dx \Big)^{1/2}= \Big( \sum_{k \in \Z} |\ha{f} (k)|^2 \Big)^{1/2}=\|\ha{f}\|_{l^2}. 
\end{align*}
%%%%%%%%%%%%%%%%%%%%%%%%%%%%%%%%%
%%%%%%%%%%%%%%%%%%%%%%%%%%%%%%%%%
We give some estimates on the phase function $\Phi_{\vp}^{(N)}$ defined as 
\begin{align*}
&\Phi_{\vp}^{(N)}:= \Phi_{\vp}^{(N)} (k_1, \dots, k_N):= \phi_{\vp}(k_{1,\dots, N})- \sum_{j=1}^N \phi_{\vp} (k_j),\\
&\phi_{\vp} (k):= -ik^5 + 2 i \ga  E_1(\vp) k^3 = \mathcal{F}_x(-\p_x^5 - 2 \ga E_1(\vp) \p_x^3)
\end{align*}
for $N \in \N$, which plays an important role to recover some derivatives when we estimate non-resonant parts of nonlinear terms.
A simple calculation yields that $\Phi_{f}^{(1)}=0$, 
\begin{align}
& \Phi_f^{(2)}= -\frac{5}{2} i k_1 k_2 k_{1,2} \big(k_{1}^2 +k_2^2 +k_{1,2}^2 - \frac{12}{5} \ga E_1(f) \big),   \notag \\
& \Phi_{f}^{(3)}= -\frac{5}{2} i k_{1,2} k_{2,3} k_{1,3} \big( k_{1,2}^2 + k_{2,3}^2 +k_{1,3}^2- \frac{12}{5} \ga E_1(f)  \big) \label{2eq2}
\end{align}
for $f \in L^2(\T)$. We easily prove the following lemmas by the factorization formula \eqref{2eq2}.
%%%%%%%%%%%%%%%%%%%%%%%%%%%%%%%%%%%%%%%%%%%%%%%%%%%%%%%%%%%%%
%%%%%%%%%%%%%%%%%%%%%%%%%%%%%%%%%%%%%%%%%%%%%%%%%%%%%%%%%%%%%
\begin{lem} \label{Le1}
Assume that $f,g \in L^2(\T)$, $k_{\max}> 8 \max\{ | \ga| E_1(f) , |\ga| E_1(g) \}$ and $k_{1,2} k_{2,3} k_{1,3} \neq 0$. If 
\begin{align} \label{rel1}
|k_1| \sim |k_2| \sim |k_3|,
\end{align}
then it follows that 
\begin{align}
& |\Phi_f^{(3)}|  \gtrsim \min\{ |k_{1,2} | |k_{1,3}|, |k_{1,2}| |k_{2,3}|, |k_{1,3}| |k_{2,3}|  \}  \, k_{\max}^3, \label{ff2} \\
& \Big| \frac{1}{\Phi_{f}^{(3)}}- \frac{1}{\Phi_g^{(3)}} \Big|   \lesssim \frac{|\ga| |E_1(f) -E_1(g) |}{ k_{\max}^2} 
\min \left\{  \frac{1}{|\Phi_f^{(3)}|}, \frac{1}{|\Phi_g^{(3)}|} \right\}.\label{ff3}
\end{align}
If \eqref{rel1} does not hold, then it follows that
\begin{align} \label{ff4}
|\Phi_f^{(3)}|  \gtrsim   \min \{  |k_{1,2}|, |k_{1,3}|, |k_{2,3}|  \} \, k_{\max}^4
\end{align}
and \eqref{ff3}.
\end{lem}
%%%%%%%%%%%%%%%%%%%%%%%%%%%%%%%%%%%%%%%
%%%%%%%%%%%%%%%%%%%%%%%%%%%%%%%%%%%%%%%
We need estimates similar to \eqref{ff2} and \eqref{ff4} for $N\ge 4$ to recover derivative losses. 
However, no factorization formulas are known for $N \ge 4$ and the following proposition holds. 
%%%%%%%%%%%%%%%%%%%%%%%%%%%%%%%%%%%%%%%%%%%%%%%%%%%%
%%%%%%%%%%%%%%%%%%%%%%%%%%%%%%%%%%%%%%%%%%%%%%%%%%%%
\begin{prop} \label{prop_counter}
Let $N \ge 4$, $f \in L^2(\T)$ and $\ga \in \R$. 
Then, there exists a $(k_1, \dots, k_{N}) \in \Z^{N}$ such that 
\begin{align*}
& |k_{N}| \gg \max_{1 \le j \le N-1} \{ |k_j| \}, \hspace{0.5cm} k_{1,\dots, N-1} \neq 0, 
\hspace{0.5cm} |k_{N}| \gg |\ga| E_1(f), \\
& |\Phi_{f}^{(N)}| \ll |k_{1, \dots, N-1}| |k_{N}|^4. 
\end{align*}
\end{prop}
%%%%%%%%%%%%%%%%%%%%%%%%%%%%%%%%%%%%%%%%%%%%%%%%%%
\begin{proof}
A simple calculation yields that 
\begin{align} \label{fact1}
& \Phi_f^{(N)}(k_1, \dots, k_{N}) \notag \\
& = \Phi_f^{(2)} (k_{1, \dots, N-1}, k_{N})+ \Phi_{f}^{(3)} (k_{1, \dots, N-3}, k_{N-2}, k_{N-1})
+\Phi_f^{(N-3)} (k_1, \dots, k_{N-3}) \notag \\
&= \big\{ \Phi_0^{(2)} (k_{1, \dots, N-1}, k_{N})+ \Phi_0^{(3)} (k_{1, \dots, N-3}, k_{N-2}, k_{N-1})
+\Phi_0^{(N-3)} (k_1, \dots, k_{N-3})  \big\} \notag \\
& \hspace{0.5cm} + 2 i \ga E_1(f) \big\{ 3 k_{N} k_{1,\dots , {N-1} } (k_{1, \dots, N-1}+k_{N}) \notag \\
& \hspace{0.5cm} 
+ 3 (k_{N-2} + k_{N-1}) (k_{1, \dots , N-3}+ k_{N-2}) (k_{1, \dots, N-3} + k_{N-1})
+ (k_{1, \dots, N-3}^3- \sum_{j=1}^{N-3} k_j^3)  \big\} \notag\\
&=: \psi_1(k_1, \dots, k_{N})+ 2 i \ga E_1(f) \psi_2(k_1, \dots, k_{N}). 
\end{align} 
Now, we take 
\begin{equation*}
k_{N}=n^{2N-1}, \hspace{0.5cm} k_{N-1}=n^{2N-2}+n^4 , \hspace{0.5cm} k_{N-2}=-n^{2N-2}
\end{equation*}
and $(k_1, \dots, k_{N-3}) \in \Z^{N-3}$ satisfying $|k_1| \le |k_2| \le \dots \le |k_{N-3}|$,
\begin{equation*}
|k_{N-3}| \sim n^{N}, \hspace{0.5cm} k_{1,\dots , N-3}=-n^4+1
\end{equation*}
for sufficiently large number $n$. Then, it follows that $k_{N-2}+k_{N-1}=n^4$ and $k_{1, \dots, N-1}=1$, 
which leads 
\begin{align*}
& \Phi_0^{(2)} (k_{1, \dots, N-1}, k_{N})= -5 i \, k_{1,\dots, N-1} k_{N}^4  +R_1=-5i \, n^{8N-4}+ R_1, \\
& \Phi_0^{(3)} (k_{1,\dots, N-3} , k_{N-2}, k_{N-1})\\
&\hspace{0.3cm} = - \frac{5}{2} i \, (k_{N-2}+k_{N-1} ) k_{N-2}  k_{N-1} (k_{N-2}^2 + k_{N-1}^2)+R_2 \\
& \hspace{0.3cm} =5i \, n^{8N-4}+10i \, n^{6N+2}+ \frac{15}{2} i \, n^{4N+8}+ \frac{5}{2} i \, n^{2N+14}+R_2
\end{align*}
where $|R_1| \sim n^{6N-3}$ and $|R_2| \sim n^{4N+8}$. 
Thus, by $|\Phi_0^{(N-3)} (k_1, \dots, k_{N-3})| \lesssim n^{5N} $, it follows that 
\begin{equation} \label{fact2}
|\psi_1(k_1, \dots, k_{N} )| \sim n^{6N+2}.
\end{equation}
Since
\begin{align*}
& |k_{1,\dots, N-1} k_{N} (k_{1, \dots , N-1} + k_{N}) | \sim n^{4N-2}, \\
& | (k_{N-2} +k_{N-1} ) (k_{1, \dots, N-3} + k_{N-2}) (k_{1, \dots, N-3} + k_{N-1})   | \sim n^{4N}, \\
& \Big| k_{1, \dots, N-3}^3- \sum_{j=1}^{N-3} k_j^3 \Big| \lesssim n^{3N},
\end{align*}
it follows that  $|\psi_2(k_1, \dots, k_{N})| \sim n^{4N}$. 
By taking $n$ sufficiently large, we have $|\ga| E_1(f) \ll n^{2N-1}=|k_N| $ and  
\begin{equation} \label{fact3}
| 2 i \ga E_1(f) \psi_2(k_1, \dots, k_{N}) | \ll n^{6N-1}.
\end{equation}
Collecting \eqref{fact1}--\eqref{fact3}, we obtain 
\begin{equation*}
|\Phi_{f}^{(N)} | \sim n^{6N+2} \ll n^{8N-4} = |k_{1, \dots , N-1}| |k_{N}|^4. 
\end{equation*}
\end{proof}
%%%%%%%%%%%%%%%%%%%%%%%%%%%%%%%%%%%%%%%%%%%%%%%%%%%%%%%%%%%%%%%%%%%
For instance, when $N=4$, we take 
\begin{equation*} 
k_4= n^{7}, \hspace{0.3cm} k_3=n^{6}+n^4, \hspace{0.3cm} k_2=-n^{6}, \hspace{0.3cm} k_1=-n^4+1
\end{equation*}
for sufficiently large number $n$. Then, it  follows that 
\begin{equation*}
|\Phi_f^{(4)}| \sim n^{26} \ll n^{28}=|k_{1,2,3}| |k_5|^4.
\end{equation*}
Moreover, when $N=5$, we take 
\begin{equation*}
k_1=n^{9}, \hspace{0.3cm} k_2=n^8 +n^4, \hspace{0.3cm} k_3=-n^8, \hspace{0.3cm} k_2=-n^5-n^4+1, \hspace{0.3cm} k_1=n^5
\end{equation*}
for sufficiently large number $n$. Then, it follows that 
\begin{equation*}
|\Phi_f^{(5)}| \sim n^{32} \ll n^{36}= |k_{1,2,3,4}| |k_5|^4. 
\end{equation*}
%%%%%%%%%%%%%%%%%%%%%%%%%%%%%%%%%%%%%%%%%%%%
%%%%%%%%%%%%%%%%%%%%%%%%%%%%%%%%%%%%%%%%%%%%
\begin{lem} \label{Le2}
Assume $f, g \in L^2(\T)$ and $k_{\max} > 16 \max\{1, |\ga| E_1(f), |\ga| E_1(g)  \}$. Then the following hold. \\
{\upshape (i)} If 
\begin{equation} 
|k_5| > 6 |k_4| > 96 \max_{1 \le j \le 3} \{  |k_j| \} \label{rel3} 
\end{equation}
holds, then it follows that 
\begin{align}
& |\Phi_f^{(5)}| > |k_4| |k_5|^4, \label{2eq5} \\
& \Big| \frac{1}{\Phi_f^{(5)}}- \frac{1}{ \Phi_{g}^{(5)}} \Big| \lesssim \frac{ |\ga| |E_1(f) -E_1(g) |  }{ k_{\max} } \,
\min \left\{ \frac{1}{|\Phi_f^{(5)}|}, \frac{1}{ |\Phi_g^{(5)}|}  \right\}. \label{2eq6}
\end{align}
% In addition, if \eqref{rel3} and $|k_5 | \le 8^3 |k_4|$, then we have $|\Phi_{f}^{(5)}| \gtrsim |k_5|^5$. \\
%%%%%%%%%%%%%%%%%%%%%%%%%%%%%%%%%%
{\upshape (ii)} If 
\begin{equation} \label{rel30}
|k_5|^{4/5} > 8^3 \max_{1 \le j \le 4} \{ |k_j| \}, \hspace{0.5cm} k_{1,2,3,4} \neq 0
\end{equation}
holds, then it follows that \eqref{2eq6} and
\begin{equation}
|\Phi_f^{(5)} | > |k_{1,2,3,4}| |k_5|^4. \label{2eq7}
\end{equation}
%%%%%%%%%%%%%%%%%%%%%%%%%%%%%%%%%%%%%%%%%%%%%%%%%
{\upshape (iii)} If 
\begin{equation} \label{rel14}
|k_5| > 8^3 \max_{1 \le j \le 4} \{|k_j| \}, \hspace{0.5cm} 16|k_{1,2}| < |k_{3,4}|
\end{equation}
or 
\begin{equation} \label{rel15}
|k_5| > 8^3 \max_{1 \le j \le 4} \{|k_j| \}, \hspace{0.5cm} 16|k_{3,4}| < |k_{1,2}|
\end{equation}
holds, then it follows that \eqref{2eq6} and 
\begin{equation} \label{2eq8}
|\Phi_f^{(5)}| \gtrsim \max \{ |k_{1,2}|,  |k_{3,4}| \} \, |k_5|^4. 
\end{equation}
\end{lem}

%%%%%%%%%%%%%%%%%%%%%%%%%%%%%%%%%%%%%%%%%%%%%
\begin{proof}
A simple calculation yields that 
\begin{align} \label{L1}
 &\Phi_f^{(5)} (k_1, k_2, k_3, k_4, k_5) = \Phi_{f}^{(3)} (k_{1,2,3}, k_4, k_5) + \Phi_{f}^{(3)} (k_1, k_2, k_3) \notag \\
& = \Phi_{0}^{(3)} (k_{1,2,3}, k_4, k_5) + \Phi_{0}^{(3)} (k_1, k_2, k_3) + \ga E_1(f) R^{(5)}(k_1, k_2, k_3, k_4, k_5)
\end{align}
where  
\begin{equation} \label{Rem1}
R^{(5)}(k_1, k_2, k_3, k_4, k_5): = 6 i ( k_{1,2,3,4} k_{1,2,3,5} k_{4,5}  + k_{1,2} k_{2,3} k_{1,3}  ).
\end{equation}
By \eqref{2eq2}, we have
\begin{align} \label{exc}
&\Phi_0^{(3)} (k_{1,2,3}, k_4, k_5) =- \frac{5}{2} i k_{1,,2,3,4} k_{1,2,3,5} k_{4,5} (k_{1,2,3,4}^2 + k_{1,2,3,5}^2 + k_{4,5}^2), \notag \\
& \Phi_0^{(3)} (k_1, k_2, k_3) = - \frac{5}{2} i k_{1,2} k_{2,3} k_{1,3} (k_{1,2}^2 + k_{2,3}^2 + k_{1,3}^2).
\end{align}
% By symmetry, we may assume that $|k_1| \le |k_2| \le |k_3|$ holds. 
(i) By \eqref{rel3}, we easily check that 
\begin{align}
& |\Phi_0^{(3)} (k_{1,2,3}, k_4, k_5)| > 2 |k_4| |k_5|^4, \hspace{0.3cm} 
|\Phi_{0}^{(3)} (k_1, k_2, k_3)| \le 240 \max_{j=1,2,3} \{ |k_j| ^5\} < 8^{-6} |k_4| |k_5|^4, \notag \\
& |R^{(5)} (k_1k_2, k_3, k_4,  k_5)| < 12 |k_4| |k_5|^2 \label{L21}
\end{align}
and by $|k_5| > 16 \max \{ 1, |\ga| E_1(f) \}$
\begin{equation*}
|\ga| E_1(f) |R^{(5)} (k_1, k_2, k_3, k_4, k_5)| < 16^{-1} |k_4| |k_5|^2.
\end{equation*}
Thus, by \eqref{L1}, we get \eqref{2eq5}. Since
\begin{equation} \label{L22}
\frac{1}{\Phi_f^{(5)}}- \frac{1}{ \Phi_g^{(5)}}= \ga (E_1(g) - E_1(f)) \, \frac{ R^{(5)} (k_1, k_2, k_3, k_4, k_5) }{ \Phi_f^{(5)} \Phi_{g}^{(5)} }, 
\end{equation}
by \eqref{2eq5} and \eqref{L21}, we get \eqref{2eq6}. \\
%% \begin{equation*} 
%% \Big| \frac{1}{\Phi_f^{(5)}}-\frac{1}{\Phi_g^{(5)}} \Big| \lesssim \frac{|\ga| |E_1(f)-E_1(g)| }{|k_5|^2} \, 
%% \min \left\{ \frac{1}{|\Phi_f^{(5)}|}, \frac{1}{|\Phi_g^{(5)}|}  \right\}
%% \end{equation*}
%% which implies \eqref{2eq6}. 
% When \eqref{rel3} and $|k_5| \le 8^3 |k_4|$, by \eqref{LR0}, we obtain 
% \begin{equation*}
% |\Phi_f^{(5)}| > |k_4||k_5|^4 \ge 8^{-3} |k_5|^5.
% \end{equation*}
%%%%%%%%%%%%%%%%%%%%%%%%%%%%%%%%%%%%%%%%%%%%%%%%%%%%%%%%%%%%%%%%%%%%%%%%%%%
(ii) By \eqref{rel30}, we can easily check that 
\begin{align}
& |\Phi_0^{(3)} (k_{1,2,3}, k_4, k_5) | > 2 |k_{1,2,3,4}| |k_5|^4, \hspace{0.3cm} 
|\Phi_0^{(3)} (k_1, k_2, k_3) | \le 8^{-12} |k_5|^4 \le 8^{-12} |k_{1,2,3,4}| |k_5|^4, \notag \\
& |R^{(5)} (k_1, k_2, k_3, k_4, k_5) | 
% <8|k_{1,2,3,4}| |k_5|^2+ 48 \max_{j=1,2,3,4} \{|k_j|^3 \}
< 8 |k_{1,2,3,4}| |k_5|^2 + 8^{-7} |k_5|^3 \label{L23}
\end{align}
and by $|k_5| > 16 |\ga| E_1(f)$ 
\begin{equation*}
|\ga| E_1(f) |R^{(5)}(k_1, k_2, k_3, k_4, k_5)| \le 8^{-3} |k_{1,2,3,4}| |k_5|^4. 
\end{equation*}
Thus, by \eqref{L1}, we have \eqref{2eq7}. 
By  \eqref{2eq7}, \eqref{L22} and \eqref{L23}, we get \eqref{2eq6}. \\
%%%%%%%%%%%%%%%%%%%%%%%%%%%%%%%%%%%%%%%%%%%%%%%%%%%%%%%%%%%%%%%%%%%%%%
(iii) Now, we prove 
\begin{equation} \label{L24}
|\Phi_f^{(5)}| > |k_{3,4}| |k_5|^4
\end{equation}
when \eqref{rel14} holds. 
By \eqref{rel14}, we can easily check that 
\begin{align}
& |\Phi_0^{(3)} (k_{1,2,3}, k_4, k_5) | > 2 |k_{3,4}| |k_5|^4, \notag \\
& |\Phi_0^{(3)} (k_1, k_2, k_3) | \le 120 |k_{1,2}| \max_{j=1,2,3,4} \{ |k_j|^4 \} \le 8^{-10} |k_{3,4}| |k_5|^4, \notag \\
& |R^{(5)} (k_1, k_2, k_3, k_4, k_5) | 
% <8|k_{3,4} | |k_5|^2+ 4|k_{1,2}| \max_{j=1,2,3,4} \{|k_j|^2 \}
< 12 |k_{3,4}| |k_5|^2 \label{L25}
\end{align}
and by $|k_5| > 16 |\ga| E_1(f)$ 
\begin{equation*}
|\ga| E_1(f) |R^{(5)}(k_1, k_2, k_3, k_4, k_5)| \le 8^{-3} |k_{3,4}| |k_5|^4. 
\end{equation*}
Thus, by \eqref{L1}, we have \eqref{L24}. 
By \eqref{L22}, \eqref{L24} and \eqref{L25}, we get \eqref{2eq6}. 
In a similar manner as above, we obtain $|\Phi_f^{(5)}| \gtrsim |k_{1,2}| |k_5|^4$ and \eqref{2eq6} 
when \eqref{rel15} holds. 
\end{proof}

%%%%%%%%%%%%%%%%%%%%%%%%%%%%%%%%%%%%%%%%%%%%%%%%%%%%%%%%%%%%%%%%%%%%%%%%%%
\begin{lem} \label{Le3}
Assume that $f, g \in L^2(\T)$ and $k_{\max} > 16 \max \{ 1, |\ga| E_1(f), |\ga| E_1(g) \}$. 
If either 
\begin{equation} \label{rel21}
|k_7| > 8^5 |k_6| > 16 \cdot 8^{5} \max_{1 \le j \le 5} \{ |k_j| \}
\end{equation}
or 
\begin{equation} \label{rel22}
|k_7|^{4/5} > 8^5 \max_{1 \le j \le  6} \{ |k_j| \} , \hspace{0.5cm} k_{1,2,3,4,5,6} \neq 0
\end{equation}
holds, then it follows that 
\begin{align}
& |\Phi_f^{(7)}| >|k_{1,2,3,4,5,6}| |k_7|^4, \label{C31} \\
& \left| \frac{1}{\Phi_f^{(7)}}-\frac{1}{\Phi_g^{(7)}}  \right| \lesssim \frac{ |\ga| |E_1(f)-E_1(g) | }{ k_{\max} } \,
 \min \left\{\frac{1}{ |\Phi_f^{(7)} | }, \, \frac{1}{ |\Phi_g^{(7)} | } \right\}. \label{C32}
\end{align}
\end{lem}

%%%%%%%%%%%%%%%%%%%%%%%%%%%%%%%%%%%%%%%%%%%%%%%%%%%%%%%%%%%%%%
\begin{proof}
We notice that it follows that 
\begin{align}
% \Phi_f^{(7)}& = \Phi_f^{(3)} (k_{1,2,3,4,5}, k_6, k_7)+ \Phi_f^{(3)}(k_{1,2,3}, k_4, k_5)+ \Phi_f^{(3)}(k_1, k_2, k_3) \notag \\
\Phi_f^{(7)} &=\Phi_0^{(3)} (k_{1,2,3,4,5}, k_6, k_7)+ \Phi_0^{(3)}(k_{1,2,3}, k_4, k_5)+ \Phi_0^{(3)}(k_1, k_2, k_3)  \notag \\
& \hspace{0.5cm} + \ga E_1(f) R^{(7)}(k_1, k_2, k_3, k_4, k_5,k_6, k_7) \label{rel23}
\end{align}
where 
\EQQ{
& R^{(7)}(k_1, k_2, k_3,k_4, k_5, k_6, k_7) \\
& \hspace{1cm} = 6i (k_{1,2,3,4,5,6}k_{1,2,3,4,5,7} k_{6,7}+ k_{1,2,3,4} k_{1,2,3,5} k_{4,5}+k_{1,2} k_{2,3} k_{1,3} ).
} 
By \eqref{2eq2}, it follows that 
\begin{align*}
& \Phi_0^{(3)} (k_{1,2,3,4,5}. k_6, k_7)=- \frac{5}{2} i k_{1,2,3,4,5,6} k_{1,2,3,4,5,7} k_{6,7} (k_{1,2,3,4,5,6}^2 +k_{1,2,3,4,5,7}^2 +k_{6,7}^2), \\
& \Phi_{0}^{(3)} (k_{1,2,3},k_4, k_5)=-\frac{5}{2} i k_{1,2,3,4} k_{1,2,3,5} k_{4,5} (k_{1,2,3,4}^2 +k_{1,2,3,5}^2 + k_{4,5}^2), \\
& \Phi_{0}^{(3)} (k_1,k_2, k_3)= -\frac{5}{2} i k_{1,2} k_{2,3} k_{1,3} (k_{1,2}^2+k_{1,3}^2+k_{2,3}^2). 
\end{align*}
First, we suppose that \eqref{rel21} holds. 
By symmetry, we assume that $|k_1| \le |k_2| \le |k_3| \le |k_4| \le |k_5|$ holds. 
Then, by \eqref{rel21}, we can easily check that 
\begin{align}
& |\Phi_0^{(3)} (k_{1,2,3,4,5}, k_6, k_7)  |> 2 |k_6| |k_7|^4, \hspace{0.5cm}
|\Phi_0^{(3)} (k_{1,2,3}, k_4, k_5) | < 8^4 |k_5|^5 < 8^{-21} |k_6| |k_7|^4, \notag \\
& |\Phi_0^{(3)} (k_1, k_2, k_3) | < 240 |k_3|^5 < 8^{-22}  |k_6| |k_7|^4, \notag \\
& |R^{(7)} (k_1k_2, k_3, k_4, k_5, k_6, k_7) | < 12 |k_6| |k_7|^2 \label{LR1}
\end{align}
and, by $|k_7|>16  \max \{ 1, |\ga| E_1(f) \}$, 
\begin{equation*}
|\ga| E_1(f) |R^{(7)} (k_1, k_2, k_3, k_4,  k_5, k_6, k_7) |  < 8^{-5} |k_6| |k_7|^4. 
\end{equation*}  
Thus, by \eqref{rel23} and $|k_6|> 16 \max_{1 \le j \le 5} \{ |k_j| \}$, we have \eqref{C31}. Since 
% \begin{equation*}
% |\Phi_f^{(7)}| > |k_6| |k_7|^4 > \frac{1}{2} |k_{1,2,3,4,5,6}| |k_7|^4
% \end{equation*}
% which means that \eqref{C31} holds. 
\begin{equation} \label{LR2}
\frac{1}{\Phi_f^{(7)} }- \frac{1}{\Phi_g^{(7)}}= \ga (E_1(g)-E_1(f)) \,
\frac{ R^{(7)} (k_1, k_2, k_3, k_4, k_5, k_6, k_7)  }{ \Phi_f^{(7)} \Phi_g^{(7)} },
\end{equation}
by \eqref{C31} and \eqref{LR1}, we get \eqref{C32}.

%%%%%%%%%%%%%%%%%%%%%%%%%%%%%%%%%%%%%%%%%%%%%%%%%%%%%%%
Next, we suppose that \eqref{rel22} holds. 
By symmetry,  we assume that $|k_1| \le |k_2| \le |k_3| \le |k_4| \le |k_5| \le |k_6|$ holds. 
Then, by \eqref{rel22}, we can easily check that 
\begin{align}
& |\Phi_0^{(3)} (k_{1,2,3,4,5}, k_6, k_7)  |> 4 |k_{1,2,3,4,5,6}| |k_7|^4, \notag \\
& |\Phi_0^{(3)} (k_{1,2,3}, k_4, k_5) | < 8^4 |k_5|^5 < 8^{-21}  |k_7|^4 \leq  8^{-21} |k_{1,2,3,4,5,6}| |k_7|^4 , \notag \\
& |\Phi_0^{(3)} (k_1, k_2, k_3) | < 240 |k_3|^5 < 8^{-22}   |k_7|^4 \le 8^{-22} |k_{1,2,3,4,5,6}| |k_7|^4, \notag \\
& |R^{(7)} (k_1k_2, k_3, k_4, k_5, k_6, k_7) | < 8 |k_{1,2,3,5,6}| |k_7|^2 + 8^{-12} |k_7|^3 \label{LR3}
\end{align}
and, by $|k_7| > 16 |\ga|  E_1(f) $, 
\begin{equation*}
|\ga| E_1(f) |R^{(7)} (k_1, k_2, k_3, k_4, k_5,  k_6, k_7) | < 8^{-5} |k_{1,2,3,4,5,6}| |k_7|^4.
\end{equation*}
Thus, by \eqref{rel23}, we have \eqref{C31}. 
By \eqref{C31}, \eqref{LR2} and \eqref{LR3}, we get \eqref{C32}. 
\end{proof}

%%%%%%%%%%%%%%%%%%%%%%%%%%%%%%%%%%%%%%%%%
\begin{lem} \label{Le4}
Assume that $f,g \in  L^2 (\T)$ and $ k_{\max}> 16 \max \{  1, |\ga| E_1(f), |\ga| E_0(g) \}$. If 
\begin{equation} \label{L12}
8^3 \max \{ |k_1|, |k_2| \} \le |k_{3,4,5}|, \hspace{0.3cm} 16 |k_3|  < \min \{ |k_4| , |k_{4,5}| \}
\hspace{0.3cm} \text{and} \hspace{0.3cm} |k_4| \le |k_5| \le 8 |k_4|
\end{equation}
then, it follows that \eqref{2eq6} and
\begin{equation}
|\Phi_f^{(5)}| \gtrsim |k_{1,2,3,4,5}| |k_5|^4. \label{C42}
\end{equation}
\end{lem}

%%%%%%%%%%%%%%%%%%%%%%%%%%%%%%%%%%%%
\begin{proof}
A direct computation yields that 
\begin{align} \label{L13}
\Phi_f^{(5)}= \Phi_0^{(3)} (k_{1,2,3}, k_4, k_5)+ \Phi_0^{(3)} (k_1, k_2,k_3) + \ga E_1 (f) R^{(5)}(k_1, k_2, k_3, k_4, k_5)
\end{align}
where $R^{(5)} (k_1, k_2, k_3, k_4, k_5) $ is defined by \eqref{Rem1}. 
We notice that \eqref{exc} holds by \eqref{2eq2}. 
By \eqref{L12}, we easily check that 
\begin{align}
& |\Phi_0^{(3)} (k_{1,2,3}, k_4, k_5) | > \frac{5}{4} |k_{4,5}| |k_4| |k_5|^3, \notag \\
& |\Phi_0^{(3)} (k_1, k_2, k_3) | \le 240 \max_{j=1,2,3} \{ |k_j|^5 \} \le 8^{-4} |k_{4,5}| |k_4| |k_5|^3, \notag \\
& |R^{(5)} (k_1, k_2, k_3, k_4, k_5) | < 8 |k_{4,5}| |k_4| |k_5| \label{LR4}
\end{align}
and, by $|k_5| > 16 \max \{1, |\ga| E_1(f) \}$,
\begin{equation*}
|\ga| E_1(f) |R^{(5)} (k_1, k_2, k_3, k_4, k_5) | < \frac{1}{32} |k_{4,5}| |k_4| |k_5|^3. 
\end{equation*}
Thus, by \eqref{L13}, we have
\begin{equation*}
|\Phi_f^{(5)}|> \frac{19}{16} |k_{4,5}| |k_4| |k_5|^3 > \frac{1}{16} |k_{1,2,3,4,5}| |k_5|^4,
\end{equation*}
which implies \eqref{C42}. 
By \eqref{L22}, \eqref{C42} and \eqref{LR4}, we get \eqref{2eq6}. 
\end{proof}
\begin{rem} \label{rem_fac1}
It follows that 
\begin{equation*}
\Phi_0^{(5)}(k_1,, k_2, k_3, k_4, k_5)= -5 i k_{1,2,3,4} k_5^4+R_0(k_1, k_2, k_3, k_4, k_5)
\end{equation*}
where $R_0$ is a polynomial of degree $5$ and satisfies 
\begin{equation} \label{fac11}
|R_0(k_1, k_2, k_3, k_4, k_5)| \lesssim |k_{1,2,3,4}|^2 |k_5|^3
\end{equation}
when 
\begin{equation} \label{fac12}
|k_5|^{3/5} > 8^4 \max_{j=1,2,3,4} \{ |k_j| \}, \hspace{1cm} k_{1,2,3,4} \neq 0.
\end{equation}
In fact, by symmetry, we may assume $|k_1| \le |k_2| \le |k_3| \le |k_4|$. 
By direct computation, it follows that 
\begin{equation*}
\Phi_0^{(5)}(k_1, k_2, k_3, k_4, k_5)= \Phi_{0}^{(2)} (k_{1,2,3,4}, k_5)+ \Phi_{0}^{(3)} (k_{1,2}, k_3, k_4)
+\Phi_0^{(2)} (k_1, k_2)
\end{equation*}
and
\begin{align}
\Phi_0^{(2)} (k_{1,2,3,4}, k_5)& =-5ik_{1,2,3,4} k_5^4 -10i k_{1,2,3,4}^2 k_5^3 -10 i k_{1,2,3,4}^3 k_5^2 -5i k_{1,2,3,4}^4 k_5, \notag \\
\Phi_0^{(3)} (k_{1,2}, k_3, k_4)& =- \frac{5}{2} i k_{1,2,3} k_{1,2, 4} k_{3,4} (k_{1,2,3}^2 +k_{1,2,4}^2 + k_{3,4}^2),  \label{fac13} \\
\Phi_0^{(2)} (k_1,k_2) & = -\frac{5}{2} i k_{1} k_2 k_{1,2} (k_1^2+ k_2^2 +k_{1,2}^2). \label{fac14}
\end{align}
Thus, $R_0$ is equal to 
\begin{equation*} 
-10i k_{1,2,3,4}^2 k_5^3-10i k_{1,2,3,4}^3 k_5^2 -5i k_{1,2,3,4}^4 k_5+ \Phi_0^{(3)} (k_{1,2}, k_3, k_4) + \Phi_0^{(2)} (k_1, k_2). 
\end{equation*}
If \eqref{fac12} holds, by \eqref{fac13} and \eqref{fac14}, we have 
\begin{align*}
& |\Phi_{0}^{(3)} (k_{1,2}, k_3, k_4)| \le 990 |k_4|^5 \lesssim |k_5|^3 \le  |k_{1,2,3,4}|^2 |k_5|^3, \\
& | \Phi_0^{(2)} (k_1, k_2) | \le 30 |k_2|^5 \lesssim  |k_5|^3 \le  |k_{1,2,3,4}|^2 |k_5|^3, \\
& 10 |k_{1,2,3,4}^2 k_5^3 | + 10|k_{1,2,3,4}^3 k_5^{2}|+ 5 |k_{1,2,3,4}^4 k_5| \lesssim |k_{1,2,3,4}|^2 |k_3|^3,
\end{align*}
which leads that \eqref{fac11} holds. 
\end{rem}

%%%%%%%%%%%%%%%%%%%%%%%%%%%%%%%%%%%%%%%%%%%%%%%%
%%%%%%%%%%%%%%%%%%%%%%%%%%%%%%%%%%%%%%%%%%%%%%%%
%%%%%%%%%%%%%%%%%%%%%%%%%%%%%%%%%%%%%%%%%%%%%%%%
\begin{defn}
For $f \in L^2(\T)$ and an $N$-multiplier $m^{(N)} (k_1, \dots, k_N)$, we define $N$-linear functional $\La_f^{(N)}$ by 
\begin{align*}
\La_f^{(N)} (m^{(N)}, \ha{v}_1, \dots, \ha{v}_N) (t,k)
= \sum_{k=k_{1,\dots, N}} e^{-t \Phi_{f}^{(N)} } m^{(N)} (k_1, \dots, k_N) \prod_{l=1}^N \ha{v}_l (k_l) 
\end{align*}
where $\ha{v}_1, \dots, \ha{v}_N$ are functions on $\Z$. 
$\La_f^{(N)} (m^{(N)}, \ha{v}, \dots, \ha{v} )$ may simply be written as $\La_f^{(N)} (m^{(N)}, \ha{v}) $. 
\end{defn}

%%%%%%%%%%%%%%%%%%%%%%%%%%%%%%%%%%%%%%%%%%
\begin{defn}
We say an $N$-multiplier $m^{(N)}$ is symmetric if 
\begin{equation*}
m^{(N)}(k_1, \dots, k_N)= m^{(N)} (k_{\sigma(1)}, \dots, k_{ \sigma(N) })
\end{equation*}
for all $\sigma \in S_N$, 
the group of all permutations on $N$ objects. 
The symmetrization of an $N$-multiplier $m^{(N)}$ is the multiplier  
\begin{align*}
[m^{(N)} ]_{sym}^{(N)} (k_1, \dots, k_N):= \frac{1}{N!} \sum_{\sigma \in S_N} m^{(N)} (k_{\sigma(1)}, \dots, k_{ \sigma(N) }).
\end{align*}
We also use $\ti{m}^{(N)}:=[m^{(N)} ]_{sym}^{(N)}$ for short.
\end{defn}
%%%%%%%%%%%%%%%%%%%%%%%%%%%%%%%%%%%%%%%%%%%%%%%%%%%%%%%%%%%%%%%%%%%
\begin{defn}
% We say an $N$-multiplier $m^{(N)}$ is symmetric with $(k_1, k_2, \dots, k_l)$ if
% \begin{equation} \label{defn3}
% m^{(N)}(k_1, \ldots, k_l, k_{l+1}, \ldots,  k_N)= m^{(N)} (k_{\sigma(1)}, \ldots, k_{\sigma(l)}, k_{l+1}, \ldots,  k_N)
% \end{equation}
% for any $\sigma \in S_l$. We notice that \eqref{defn3} with $l=N$ means that $m^{(N)}$ is symmetric. 
We say an $N$-multiplier $m^{(N)}$ is symmetric with $(k_i, k_{i+1}, \dots, k_{j})$ if 
\begin{align*}
& m^{(N)} (k_1, \dots, k_{i-1}, k_{i}, k_{i+1}, \dots, k_{j}, k_{j+1}, \dots, k_{N}) \\
 & = m^{(N)}(k_1, \dots, k_{i-1}, k_{\tau(i)}, k_{\tau(i+1)}, \dots, k_{\tau(j)}, k_{j+1}, \dots, k_{N})
\end{align*}
for $1 \le i < j \le N$ and $\tau \in \ti{S}_{j-i+1}$ where $\ti{S}_{j-i+1}$ is the group of all permutations on $\{ i, i+1, \dots, j \}$. 
\end{defn}
%%%%%%%%%%%%%%%%%%%%%%%%%%%%%%%%%%%%%%%%%%%%%%%%%%%%%%%%
%%%%%%%%%%%%%%%%%%%%%%%%%%%%%%%%%%%%%%%%%%%%%%%%%%%%%%%
\begin{defn}
We define $(N+j)$-extension operators of an $N$-multiplier $m^{(N)}$ with $j \in \N$ by 
\begin{align*}
[ m^{(N)}  ]_{ext1}^{(N+j)} (k_1, \dots, k_{N+j})& = m^{(N)} (k_1, \dots, k_{N-1}, k_{N, \dots, N+j}), \\
[ m^{(N)}  ]_{ext2}^{(N+j)} (k_1, \dots, k_{N+j})& = m^{(N)} (k_{j+1}, \dots, k_{j+N}). 
\end{align*}
\end{defn}
\begin{rem}\label{rem_sym}
For any multipliers $m_1, m_2$, it follows that
$$[m_1m_2]^{(N)}_{ext1}=[m_1]^{(N)}_{ext1}[m_2]^{(N)}_{ext1}, \ \ \ \ 
[m_1m_2]^{(N)}_{ext2}=[m_1]^{(N)}_{ext2}[m_2]^{(N)}_{ext2}.$$
\end{rem}

%%%%%%%%%%%%%%%%%%%%%%%%%%%%%%%%%%%%%%
%%%%%%%%%%%%%%%%%%%%%%%%%%%%%%%%%%%%%%
For $L>0$, we define some multipliers to restrict summation regions in the Fourier space as follows:
\begin{align*}
& \chi_{\le L}^{(2N+1)}:=
\begin{cases}
\dis 1, \hspace{0.2cm} \text{when} \hspace{0.2cm} \max_{1 \le j \le 2N+1}  \{  |k_j| \} \le L \\
0, \hspace{0.2cm} \text{otherwise}
\end{cases},
\hspace{0.3cm} 
\chi_{> L}^{(2N+1)}:=
\begin{cases}
\dis 1, \hspace{0.2cm} \text{when} \hspace{0.2cm} \max_{1 \le j \le 2N+1}  \{  |k_j| \} > L \\
0, \hspace{0.2cm} \text{otherwise}
\end{cases} \\ % error/ overfull
& \chi_{H1}^{(2N+1)}:=
\begin{cases}
\dis 1, \hspace{0.2cm} \text{when} \hspace{0.2cm} 8^{2N-1} \max_{1 \le l \le 2N} \{ |k_l|   \} < |k_{2N+1}| \\
0, \hspace{0.2cm} \text{otherwise}
\end{cases}, 
\end{align*}
with $N=1,2,3$ and 
\begin{align*}
& \chi_{H2,1}^{(3)}:=
\begin{cases}
\dis 1, \hspace{0.2cm} \text{when} \hspace{0.2cm} 16|k_1| < \min\{ |k_{2,3}|, |k_2|, |k_3|\}, 
\hspace{0.3cm} |k_2|/8 \le |k_3| \le 8 |k_2| \\
0, \hspace{0.2cm} \text{otherwise}
\end{cases}, \\
& \chi_{H2,2}^{(3)}:=
\begin{cases}
\dis 1, \hspace{0.2cm} \text{when} \hspace{0.2cm} |k_{2,3}| \le 16|k_1| < \min\{ |k_2|, |k_3| \}, 
\hspace{0.3cm} |k_2|/8 \le |k_3| \le 8 |k_2| \\
0, \hspace{0.2cm} \text{otherwise}
\end{cases}, 
\end{align*} % error/overfull
$\chi_{H3}^{(3)}:=1-[3 \chi_{H1}^{(3)} ]_{sym}^{(3)}- [ 3\chi_{H2,1}^{(3)} ]_{sym}^{(3)}-[3 \chi_{H2,2}^{(3)}]_{sym}^{(3)}$. 
Note that $|k_1| \sim |k_2| \sim |k_3|$ for $(k_1, k_2, k_3) \in \supp \chi_{H3}^{(3)}$. 
%%%%%%%%%%%%%%%%% non-resonant parts %%%%%%%%%%%%%%%%%%%%%%%%%%%%%%%%%%%%%%%%%%%%%%
We put
\begin{align*}
& \chi_{NR1}^{(3)}:=
\begin{cases}
\dis 1, \hspace{0.2cm} \text{when} \hspace{0.2cm} k_{1,2} k_{2,3} k_{1,3} \neq 0 \\
0, \hspace{0.2cm} \text{otherwise}
\end{cases},
\hspace{0.3cm}
\chi_{NR2}^{(5)}:=
\begin{cases}
\dis 1, \hspace{0.2cm} \text{when} \hspace{0.2cm} k_{1,2}k_{3,4} k_{1,2,3,4} \neq 0 \\
0, \hspace{0.2cm} \text{otherwise}
\end{cases},
\end{align*}
$\chi_{NR1}^{(5)}:=[\chi_{NR1}^{(3)}]_{ext1}^{(5)} [ \chi_{NR1}^{(3)} ]_{ext2}^{(5)}$ and 
%%%%%%%%%%%%%%%%% resonant parts %%%%%%%%%%%%%%%%%%%%%%%%%%%%%%%%%%%%%%
\begin{align*}
& \chi_{R1}^{(2N+1)}:=
\begin{cases}
\dis 1, \hspace{0.2cm} \text{when} \hspace{0.2cm} k_{1,2,\dots, 2N}= 0 \\
0, \hspace{0.2cm} \text{otherwise}
\end{cases}
\end{align*}
with $N=1,2,3$ and 
\begin{align*}
& \chi_{R2}^{(5)}:=
\begin{cases}
\dis 1, \hspace{0.2cm} \text{when} \hspace{0.2cm} k_{1,2} k_{3,4} =0 \\
0, \hspace{0.2cm} \text{otherwise}
\end{cases},
\hspace{0.3cm}
\chi_{R3}^{(3)}:=
\begin{cases}
\dis 1, \hspace{0.2cm} \text{when} \hspace{0.2cm} k_1=-k_2=k_3 \\
0, \hspace{0.2cm} \text{otherwise}
\end{cases}, \\
& \chi_{R4}^{(5)}:=
\begin{cases}
\dis 1, \hspace{0.2cm} \text{when} \hspace{0.2cm}  |k_5|^{4/5} \le 8^3 \max_{1 \le j \le 4} \{ |k_j| \}, \hspace{0.3cm} 
\max_{1 \le j \le 4} \{ |k_j| \} \le 16  \text{sec}_{1 \le j \le 4} \{ |k_j|  \}, \\
\hspace{2cm} |k_{1,2}| /16 \le |k_{3,4}| \le 16|k_{1,2}| \\
0, \hspace{0.2cm} \text{otherwise}
\end{cases}, \\
& \chi_{R5}^{(2N+1)}:=
\begin{cases}
\dis 1, \hspace{0.2cm} \text{when} \hspace{0.2cm}  |k_{2N+1}|^{4/5} \le 8^{2N-1}  \max_{1\le j \le 2N} \{|k_j| \}, \\
\hspace{2cm} \max_{1 \le j \le 2N} \{ |k_j| \} \leq 16 \text{sec}_{1 \le j \le 2N} \{ |k_j| \} \\
0, \hspace{0.2cm} \text{otherwise}
\end{cases}
\end{align*}
with $N=2,3$. Note that $\chi_{NR2}^{(5)}= (1-\chi_{R1}^{(5)}) (1-\chi_{R2}^{(5)})$. We define 
%%%%%%%%%%%%%%%%%%%%%%%%%%%%%%%%
\begin{align*} 
& \chi_{A1}^{(2N+1)}:=
\begin{cases}
1, \hspace{0.2cm} \text{when} \hspace{0.2cm}  16 \, \text{sec}_{1 \le j \le 2N} \{ |k_j| \} < \max_{1 \le j \le 2N} \{ |k_j| \}, \\
0, \hspace{0.2cm} \text{otherwise}
\end{cases}, 
\end{align*}
with $N=2,3$ and 
\begin{align*}
& \chi_{A2}^{(5)}:=
\begin{cases}
1, \hspace{0.2cm} \text{when} \hspace{0.2cm} 16 |k_{1,2}|< |k_{3,4}| \\
0, \hspace{0.2cm} \text{otherwise}
\end{cases}, \\
& \chi_{A3}^{(5)}:=
\begin{cases}
1, \hspace{0.2cm} \text{when} \hspace{0.2cm} 8^3 \max \{ |k_1|, |k_2| \} < |k_{3,4,5}| \\
0, \hspace{0.2cm} \text{otherwise}
\end{cases}, \\
& \chi_{A4}^{(7)}:=
\begin{cases}
1, \hspace{0.2cm} \text{when} \hspace{0.2cm}  8^5 \max_{1 \le j \le 6} \{ |k_j| \} < |k_7|^{3/5}  \\
0, \hspace{0.2cm} \text{otherwise}
\end{cases}. 
\end{align*}
Moreover, we define
\begin{align*}
& \chi_{NR(1,1)}^{(5)}:= [ \chi_{H1}^{(3)}]_{ext1}^{(5)} [ \chi_{H1}^{(3)} ]_{ext2}^{(5)}, 
\hspace{0.3cm}
\chi_{NR(1,2)}^{(5)}:= \chi_{H1}^{(3)} (k_{3,4,5}, k_2, k_1) \chi_{H1}^{(3)} (k_3,k_4, k_5), \\
& \chi_{NR(2,1)}^{(5)}:= [ \chi_{H1}^{(3)}]_{ext1}^{(5)} [ \chi_{H2,1}^{(3)} ]_{ext2}^{(5)}, 
 \hspace{0.3cm}
\chi_{NR(2,2)}^{(5)}:= \chi_{H1}^{(3)} (k_{3,4,5}, k_2, k_1) \chi_{H2,1}^{(3)} (k_3, k_4,  k_5), \\
& \chi_{NR(1,1)}^{(7)}:= [ \chi_{H1}^{(5)}]_{ext1}^{(7)} [ \chi_{H1}^{(3)} ]_{ext2}^{(7)}, 
\hspace{0.3cm}
\chi_{NR(1,2)}^{(7)}:= \chi_{H1}^{(5)} (k_{5,6,7}, k_2, k_3, k_4, k_1) \chi_{H1}^{(3)} (k_5, k_6,  k_7), \\
& \chi_{NR(2,1)}^{(7)}:=[ \chi_{H1}^{(5)}]_{ext1}^{(7)} [ \chi_{H2,1}^{(2)} ]_{ext2}^{(7)}, 
\hspace{0.3cm}
\chi_{NR(2,2)}^{(7)}:= \chi_{H1}^{(5)} (k_{5,6,7}, k_2, k_3, k_4, k_1) \chi_{H2,1}^{(3)} (k_5,k_6,  k_7).
\end{align*}

%%%%%%%%%%%%%%%%%%%%%%%%%%%%%%%%%%%%%%%%%%
%%%%%%%%%%%%%%%%%%%%%%%%%%%%%%%%%%%%%%%%%
%%%%%%%%%%%%%%%%%%%%%%%%%%%%%%%%%%%%%%%%%
\begin{lem} \label{Le6} 
Let $3$-multipliers $m_1^{(3)}$, $m_2^{(3)}$ and a $5$-multiplier $m_1^{(5)}$ be symmetric 
and a $5$-multiplier $m_3^{(5)}$ be symmetric with $(k_1, k_2, k_3, k_4)$. 
Then, it follows that 
\begin{align}
&\big[ [ m_1^{(3)} [3 \chi_{H1}^{(3)}]_{sym}^{(3)} ]_{ext1}^{(5)} 
[m_2^{(3)} ([3\chi_{H1}^{(3)}]_{sym}^{(3)} + [3 \chi_{H2,1}^{(3)}]_{sym}^{(3)} ) ]_{ext2}^{(5)}  \big]_{sym}^{(5)}  \nonumber \\
 & = 3 \big[ [ m_1^{(3)} ]_{ext1}^{(5)} [m_2^{(3)}]_{ext2}^{(5)} \chi_{NR(1,1)}^{(5)} \big]_{sym}^{(5)} 
+  6 \big[ [ m_1^{(3)} ]_{ext1}^{(5)} [m_2^{(3)}]_{ext2}^{(5)}  \chi_{NR(1,2)}^{(5)} \big]_{sym}^{(5)} \nonumber \\
& + 3 \big[ [ m_1^{(3)} ]_{ext1}^{(5)} [m_2^{(3)}]_{ext2}^{(5)} \chi_{NR(2,1)}^{(5)} \big]_{sym}^{(5)} 
+  6 \big[ [ m_1^{(3)} ]_{ext1}^{(5)} [m_2^{(3)}]_{ext2}^{(5)} \chi_{NR(2,2)}^{5)} \big]_{sym}^{(5)} 
\label{le211} 
\end{align}
and
\begin{align}
&\big[ [ m_1^{(5)}  [ 5  m_3^{(5)} \chi_{H1}^{(5)} ]_{sym}^{(5)} ]_{ext1}^{(7)} 
[m_2^{(3)} ( [3\chi_{H1}^{(3)}]_{sym}^{(3)}+ [3 \chi_{H2,1}^{(3)}]_{sym}^{(3)})   ]_{ext2}^{(7)}  \big]_{sym}^{(7)} \nonumber \\
& =3 \big[ [ m_1^{(5)} m_3^{(5)} ]_{ext1}^{(7)} [m_2^{(3)}]_{ext2}^{(7)} \chi_{NR(1,1)}^{(7)}  \big]_{sym}^{(7)}  \nonumber \\
%& + 12 \big[ [m_1^{(5)}]_{ext1}^{(7)} [ m_2^{(3)} ]_{ext2}^{(7)} [ [ 5 m_3^{(3)} \chi_{H1}^{(5)} ]_{sym}^{(5)} ]_{ext1}^{(7)} \chi_{NR(1,2)}^{(7)} \big]_{sym}^{(7)}  \nonumber \\
& + 12 \big[ [  [5 m_1^{(5)} m_{3}^{(5)} \chi_{H1}^{(5)} ]_{sym}^{(5)} ]_{ext1}^{(7)} [ m_2^{(3)} ]_{ext2}^{(7)} \chi_{NR(1,2)}^{(7)} \big]_{sym}^{(7)} \notag \\
& + 3 \big[ [m_1^{(5)} m_3^{(5)} ]_{ext1}^{(7)} [m_2^{(3)}]_{ext2}^{(7)} \chi_{NR(2,1)}^{(7)}  \big]_{sym}^{(7)} \notag \\
%& + 12 \big[ [m_1^{(5)}]_{ext1}^{(7)} [ m_2^{(3)} ]_{ext2}^{(7)} [ [ 5 m_3^{(3)} \chi_{H1}^{(5)} ]_{sym}^{(5)} ]_{ext1}^{(7)} \chi_{NR(2,2)}^{(7)} \big]_{sym}^{(7)}  \nonumber \\
& + 12 \big[ [ [5 m_1^{(5)} m_3^{(5)} \chi_{H1}^{(5)} ]_{sym}^{(5)} ]_{ext1}^{(7)} [ m_2^{(3)} ]_{ext2}^{(7)} \chi_{NR(2,2)}^{(7)} \big]_{sym}^{(7)}.
\label{le212} 
\end{align}
\end{lem}

%%%%%%%%%%%%%%%%%%%%%%%%%%%%%%%%%%%%%%%%%%%%%
%%%%%%%%%%%%%%%%%%%%%%%%%%%%%%%%%%%%%%%%%
\begin{proof}
First, we prove (\ref{le211}). 
% Since the left hand side of (\ref{le211}) is equal to 
% \begin{align*}
% & \big[ [ m_1^{(3)} [3 \chi_{H1}^{(3)}]_{sym}^{(3)} ]_{ext1}^{(5)} [m_2^{(3)} [3 \chi_{H1}^{(3)} ]_{sym}^{(3)} ]_{ext2}^{(5)}  \big]_{sym}^{(5)} \\
% +& \big[ [ m_1^{(3)} [3 \chi_{H1}^{(3)}]_{sym}^{(3)} ]_{ext1}^{(5)} [m_2^{(3)} [3 \chi_{H2,1}^{(3)} ]_{sym}^{(3)} ]_{ext2}^{(5)}  \big]_{sym}^{(5)}, 
% \end{align*}
It suffices to show 
\begin{align}
& \big[ [ m_1^{(3)} [3 \chi_{H1}^{(3)}]_{sym}^{(3)} ]_{ext1}^{(5)} 
[m_2^{(3)} [3 \chi_{H1}^{(3)} ]_{sym}^{(3)} ]_{ext2}^{(5)}  \big]_{sym}^{(5)} \nonumber \\
& =  3 \big[ [ m_1^{(3)} ]_{ext1}^{(5)} [m_2^{(3)}]_{ext2}^{(5)} \chi_{NR(1,1)}^{(5)} \big]_{sym}^{(5)} 
 +  6 \big[ [ m_1^{(3)} ]_{ext1}^{(5)} [m_2^{(3)}]_{ext2}^{(5)}  \chi_{NR(1,2)}^{(5)} \big]_{sym}^{(5)}, \label{le221} \\
& \big[ [ m_1^{(3)} [3 \chi_{H1}^{(3)}]_{sym}^{(3)} ]_{ext1}^{(5)} [m_2^{3)} [3 \chi_{H2,1}^{(3)}]_{sym}^{(3)} ]_{ext2}^{(5)} 
\big]_{sym}^{(5)} \nonumber \\
& = 3\big[ [ m_1^{(3)} ]_{ext1}^{(5)} [m_2^{(3)}]_{ext2}^{(5)} \chi_{NR(2,1)}^{(5)} \big]_{sym}^{(5)} 
+  6\big[ [ m_1^{(3)} ]_{ext1}^{(5)} [m_2^{(3)}]_{ext2}^{(5)}  \chi_{NR(2,2)}^{(5)} \big]_{sym}^{(5)}. \label{le222}
\end{align}
Put $M:= [ m_1^{(3)} [3 \chi_{H1}^{(3)}]_{sym}^{(3)} ]_{ext1}^{(5)} [m_2^{(3)} [3 \chi_{H1}^{(3)} ]_{sym}^{(3)} ]_{ext2}^{(5)}$. 
Then, by Remark~\ref{rem_sym}, 
\begin{align*}
M =&  [ m_1^{(3)}]_{ext1}^{(5)}  [m_2^{(3)} ]_{ext2}^{(5)} 
[ [ 3 \chi_{H1}^{(3)} ]_{sym}^{(3)} ]_{ext1}^{(5)}  [ [ 3 \chi_{H1}^{(3)} ]_{sym}^{(3)} ]_{ext2}^{(5)} \\
=&  [ m_1^{(3)}]_{ext1}^{(5)}  [m_2^{(3)} ]_{ext2}^{(5)} 
\{ \chi_{H1}^{(3)} (k_1, k_2, k_{3,4,5}) + \chi_{H1}^{(3)} (k_{3,4,5}, k_2, k_1) +\chi_{H1}^{(3)} (k_{3,4,5}, k_1, k_2) \} \\
& \times \{ \chi_{H1}^{(3)} (k_3, k_4, k_5) + \chi_{H1}^{(3)} (k_5, k_4, k_3)+ \chi_{H1}^{(3)} (k_3, k_5, k_4) \}.
\end{align*}
Since $m_1^{(3)}$ and $m_2^{(3)}$ are symmetric, 
$ [ m_1^{(3)}]_{ext1}^{(5)}  [m_2^{(3)} ]_{ext2}^{(5)}$ is symmetric with $(k_1, k_2)$ and $(k_3, k_4, k_5)$. 
Thus, $[M]_{sym}^{(5)}$ is equal to
\begin{align*}
& \big[ [ m_1^{(3)}]_{ext1}^{(5)}  [m_2^{(3)} ]_{ext2}^{(5)} 
\{ \chi_{H1}^{(3)} (k_1, k_2, k_{3,4,5}) + 2 \chi_{H1}^{(3)} (k_{3,4,5}, k_2, k_1 ) \} 3 \chi_{H1}^{(3)}(k_3, k_4, k_5) \big]_{sym}^{(5)} \\
& = 3 \big[ [ m_1^{(3)}]_{ext1}^{(5)}  [m_2^{(3)} ]_{ext2}^{(5)} \chi_{NR(1,1)}^{(5)} \big]_{sym}^{(5)}
+ 6 \big[ [ m_1^{(3)}]_{ext1}^{(3)}  [m_2^{(3)} ]_{ext2}^{(5)} \chi_{NR(1,2)}^{(5)} \big]_{sym}^{(5)},
\end{align*}
which implies (\ref{le221}). Similarly, we obtain (\ref{le222}). 

%%%%%%%%%%%%%%%%%%%%%%%%%%%%%%%%%%%%%%%%%%%%%%%%%%%%%%%
Next, we prove (\ref{le212}). 
% Since the left hand side of (\ref{le212}) is equal to 
% \begin{align*}
% & \big[ [ m_1^{(5)}  [5  m_3^{(5)} \chi_{H1}^{(5)}   ]_{sym}^{(5)} ]_{ext1}^{(7)} 
% [m_2^{(3)} [3 \chi_{H1}^{(3)} ]_{sym}^{(3)} ]_{ext2}^{(7)}  \big]_{sym}^{(7)} \\
% + & \big[ [  m_1^{(5)} [5  m_3^{(5)} \chi_{H1}^{(5)}   ]_{sym}^{(5)} ]_{ext1}^{(7)} [m_2^{(3)} [3 m_{H2,1}^{(3)}]_{sym}^{(3)} ]_{ext2}^{(7)}  \big]_{sym}^{(7)},
% \end{align*}
It suffices to show 
\begin{align}
& \big[  [  m_1^{(5)} [5  m_3^{(5)} \chi_{H1}^{(5)}   ]_{sym}^{(3)} ]_{ext1}^{(7)} [m_2^{(3)} [3 \chi_{H1}^{(3)} ]_{sym}^{(3)} ]_{ext2}^{(7)}  
\big]_{sym}^{(7)} 
\nonumber \\
&=3 \big[ [ m_1^{(5)} m_3^{(5)} ]_{ext1}^{(7)} [m_2^{(3)} ]_{ext2}^{(7)} \chi_{NR(1,1)}^{(7)}  \big]_{sym}^{(7)} \nonumber \\
& + 12 \big[ [ [ 5 m_1^{(5)} m_3^{(5)} \chi_{H1}^{(5)} ]_{sym}^{(5)} ]_{ext1}^{(7)} [m_{2}^{(3)}]_{ext2}^{(7)}
 \chi_{NR(1,2)}^{(7)} \big]_{sym}^{(7)}, \label{le231} \\
& \big[ [m_1^{(5)}  [5  m_3^{(5)} \chi_{H1}^{(5)}   ]_{sym}^{(5)} ]_{ext1}^{(7)} [m_2^{(3)} [ 3m_{H2,1}^{(3)} ]_{sym}^{(3)} ]_{ext2}^{(7)}  \big]_{sym}^{(7)}  \nonumber \\
&= 3 \big[ [m_1^{(5)} m_3^{(5)} ]_{ext1}^{(7)} [m_2^{(3)} ]_{ext2}^{(7)} \chi_{NR(2,1)}^{(7)} \big]_{sym}^{(7)} \nonumber \\
& + 12 \big[ [ [5 m_1^{(5)} m_3^{(5)} \chi_{H1}^{(5)} ]_{sym}^{(5)} ]_{ext1}^{(7)} [m_{2}^{(3)}]_{ext2}^{(7)}
 \chi_{NR(2,2)}^{(7)} \big]_{sym}^{(7)}. \label{le232} 
\end{align}
We only show (\ref{le231}) since (\ref{le232}) follows in a similar manner. 
Put 
\begin{equation*}
M:=[ m_1^{(5)} [ 5 m_3^{(3)} \chi_{H1}^{(5)}   ]_{sym}^{(5)} ]_{ext1}^{(7)} [m_2^{(3)} [3 \chi_{H1}^{(3)} ]_{sym}^{(3)} ]_{ext2}^{(7)}.
\end{equation*}
Since $\supp \,  [5 m_3^{(5)} \chi_{H1}^{(5)}]_{sym}^{(5)} \subset \supp \, [5 \chi_{H1}^{(5)}]_{sym}^{(5)}$, by Remark~\ref{rem_sym}, 
\begin{align*} %\label{le233}
[ [5 m_3^{(5)} \chi_{H1}^{(5)}]_{sym}^{(5)} ]_{ext1}^{(7)} 
& =[ [5 m_{3}^{(5)} \chi_{H1}^{(5)}]_{sym}^{(5)} [ 5 \chi_{H1}^{(5)} ]_{sym}^{(5)}  ]_{ext1}^{(7)}  \nonumber \\
& =[ [5 m_{3}^{(5)} \chi_{H1}^{(5)}]_{sym}^{(5)} ]_{ext1}^{(7)} [ [ 5 \chi_{H1}^{(5)} ]_{sym}^{(5)}  ]_{ext1}^{(7)}.
\end{align*}
Thus, by Remark~\ref{rem_sym}, it follows that 
\begin{align*}
M & = [m_1^{(5)}]_{ext1}^{(7)}  [m_2^{(3)}]_{ext2}^{(7)}  [ [5 m_{3}^{(5)} \chi_{H1}^{(5)} ]_{sym}^{(5)} ]_{ext1}^{(7)} 
 [ [3\chi_{H1}^{(3)} ]_{sym}^{(3)} ]_{ext2}^{(7)} \\
& = [m_1^{(5)}]_{ext1}^{(7)}  [m_{2}^{(3)} ]_{ext2}^{(7)}  [ [5 m_{3}^{(5)} \chi_{H1}^{(5)} ]_{sym}^{(5)} ]_{ext1}^{(7)} 
[ [5 \chi_{H1}^{(5)}]_{sym}^{(5)} ]_{ext1}^{(7)} [ [3 \chi_{H1}^{(3)} ]_{sym}^{(3)} ]_{ext2}^{(7)} \\
& = [m_1^{(5)}]_{ext1}^{(7)} [m_{2}^{(2)} ]_{ext2}^{(7)}  [ [5 m_{3}^{(3)} \chi_{H1}^{(5)} ]_{sym}^{(5)} ]_{ext1}^{(7)} \\
& \times \{ \chi_{H1}^{(3)}(k_5, k_6, k_7) + \chi_{H1}^{(3)} (k_7, k_6, k_5) +\chi_{H1}^{(3)} (k_5, k_7, k_6)   \} \\
& \times \{ \chi_{H1}^{(5)} (k_1, k_2, k_3, k_4, k_{5,6,7}) +\chi_{H1}^{(5)} (k_{5,6,7}, k_2, k_3, k_4, k_1) 
+ \chi_{H1}^{(5)} (k_{5,6,7}, k_1, k_3, k_4 , k_2) \\
& \hspace{0.5cm}+ \chi_{H1}^{(5)}(k_{5,6,7}, k_2, k_1, k_4, k_3) + \chi_{H1}^{(5)} (k_{5,6,7}, k_2, k_3, k_1, k_4) \}.
\end{align*}
Since $m_1^{(5)}$ and $m_2^{(3)}$ are symmetric, 
$ [ m_1^{(5)}]_{ext1}^{(7)} [m_2^{(3)}]_{ext2}^{(7)}  [[ 5 m_3^{(5)} \chi_{H1}^{(5)} ]_{sym}^{(5)} ]_{ext1}^{(7)}$ 
is symmetric with $(k_1, k_2, k_3, k_4)$ and $(k_5, k_6,  k_7)$. 
Therefore, we have  
\begin{align} 
[M]_{sym}^{(7)} & = 
\big[  [ m_1^{(5)} ]_{ext1}^{(7)} [m_2^{(3)} ]_{ext2}^{(7)} [ [5 m_{3}^{(5)} \chi_{H1}^{(5)} ]_{sym}^{(5)} ]_{ext1}^{(7)} \nonumber \\
&  \times \{ \chi_{H1}^{(5)} (k_1, k_2, k_3,  k_4, k_{5,6,7})+  4\chi_{H1}^{(5)} (k_{5,6,7}, k_2, k_3, k_4, k_1) \}  \notag \\
& \times  3 \chi_{H1}^{(3)} (k_5, k_6, k_7) \big]_{sym}^{(7)} \nonumber \\
&= 3 \big[  [m_1^{(5)}]_{ext1}^{(7)}  [m_{2}^{(3)} ]_{ext2}^{(7)}  [ [5 m_{3}^{(5)} \chi_{H1}^{(5)} ]_{sym}^{(5)} ]_{ext1}^{(7)} 
\chi_{NR(1,1)}^{(7)} \big]_{sym}^{(7)}  \label{le234} \\
& + 12 \big[  [ m_1^{(5)} ]_{ext1}^{(7)} [m_{2}^{(3)} ]_{ext2}^{(7)} [ [5 m_{3}^{(5)} \chi_{H1}^{(5)} ]_{sym}^{(5)} ]_{ext1}^{(7)}
\chi_{NR(1,2)}^{(7)} ]_{sym}^{(7)}. \label{le235}   
\end{align}
Since $m_1^{(5)}$ is symmetric, by Remark~\ref{rem_sym}, 
\begin{equation*}
[ m_1^{(5)} ]_{ext1}^{(7)} [ [5 m_{3}^{(5)} \chi_{H1}^{(5)} ]_{sym}^{(5)} ]_{ext1}^{(7)} = 
[ m_1^{(5)} \, [5 m_{3}^{(5)} \chi_{H1}^{(5)} ]_{sym}^{(5)}]_{ext1}^{(7)}= 
[ [5 m_1^{(5)} m_{3}^{(5)} \chi_{H1}^{(5)} ]_{sym}^{(5)} ]_{ext1}^{(7)},
\end{equation*}
which leads that (\ref{le235}) is equal to 
\begin{equation*}
12\big[ [ [5 m_1^{(5)} m_3^{(5)} \chi_{H1}^{(5)} ]_{sym}^{(5)}  ]_{ext1}^{(7)}  [m_2^{(3)}]_{ext2}^{(7)} \chi_{NR(1,2)}^{(7)}  \big]_{sym}^{(7)}.
\end{equation*}
By $\chi_{NR(1,1)}^{(7)}= [\chi_{H1}^{(5)}]_{ext1}^{(7)} \chi_{NR(1,1)}^{(7)}$ and Remark~\ref{rem_sym}, 
\begin{align}
& [ [ 5m_{3}^{(5)} \chi_{H1}^{(5)} ]_{sym}^{(5)} ]_{ext1}^{(7)} \chi_{NR(1,1)}^{(7)}
= [ [ 5m_{3}^{(5)} \chi_{H1}^{(5)} ]_{sym}^{(5)} ]_{ext1}^{(7)} [\chi_{H1}^{(5)} ]_{ext1}^{(7)} \chi_{NR(1,1)}^{(7)} \notag \\
&  = [ [ 5m_{3}^{(5)} \chi_{H1}^{(5)} ]_{sym}^{(5)} \chi_{H1}^{(5)} ]_{ext1}^{(7)} \chi_{NR(1,1)}^{(7)}
= [m_3^{(5)} \chi_{H1}^{(5)} ]_{ext1}^{(7)} \chi_{NR(1,1)}^{(7)} \notag \\
&= [ m_3^{(5)} ]_{ext1}^{(7)} [\chi_{H1}^{(5)}]_{ext1}^{(7)} \chi_{NR(1,1)}^{(7)}
=[m_3^{(5)}]_{ext1}^{(7)} \chi_{NR(1,1)}^{(7)}. \label{le236}
\end{align}
Here we used that $m_3^{(5)}$ is symmetric with $(k_1, k_2, k_3, k_4)$ in the third equality. 
By \eqref{le236} and Remark~\ref{rem_sym}, \eqref{le234} is equal to 
\begin{equation*}
3\big[ [ m_1^{(5)} m_3^{(5)} ]_{ext1}^{(7)} [ m_2^{(3)} ]_{ext2}^{(7)} \chi_{NR(1,1)}^{(7)} \big]_{sym}^{(7)}.
\end{equation*}
Therefore, we obtain \eqref{le231}.
\end{proof}
%%%%%%%%%%%%%%%%%%%%%%%%%%%%%%%%%%%%%%%%%%%%%%%%%%%%%%%%%%%%
%%%%%%%%%%%%%%%%%%%%%%%%%%%%%%%%%%%%%%%%%%%%%%%%%%%%%%%%%%%%
%%%%%%%%%%%%%%%%%%%%%%%%%%%%%%%%%%%%%%%%%%%%%%%%%%%%%%%%%%%%
Finally, we show variants of Sobolev's inequalities.
%%%%%%%%%%%%%%%%%%%%%%%%%%%%%%%%%%%%%%%%%%%%%%%%
\begin{lem} \label{lem_nl2}
Let $s \ge 3/2$. 
Then, for any $\{ v_l \}_{l=1}^N \subset H^s(\T)$, we have
\begin{align} \label{nl11}
\Big\| \sum_{k=k_{1,\cdots, N}} \langle k_{\max}  \rangle^{3} \prod_{l=1}^N |\ha{v}_l(k_l)|  \Big\|_{l_{s-3}^2  } \lesssim \prod_{l=1}^N \| v_l \|_{H^s}. 
\end{align}
\end{lem}

%%%%%%%%%%%%%%%%%%%%%%%%%%%%%%%%%%%%%%%%%%%%%%%%
\begin{proof}
By symmetry, we assume $|k_1| \le |k_2| \le \dots \le |k_{N-1}| \le |k_N|$. 
First, we suppose that $|k_{1,\dots ,N}| \sim |k_N|$ holds. 
Since $\langle k_{1,\dots, N}  \rangle^{s-3} \langle k_{N} \rangle^3 \sim \langle k_N  \rangle^s $, 
by the Young inequality and the Schwarz inequality, the left hand side of \eqref{nl11} is bounded by
\begin{equation*}
\big\| |\ha{v}_1|* \dots * |\ha{v}_{N-1}|* \langle \cdot \rangle^s |\ha{v}_N|  \big\|_{l^2} 
\lesssim \prod_{l=1}^{N-1} \| \ha{v}_l \|_{l^1} \|  \langle \cdot \rangle^s |\ha{v}_N| \|_{l^2} \lesssim \prod_{l=1}^N \| v_l \|_{H^s}.
% \prod_{l=1}^{N-1} \| v_l \|_{H^{1/2+}} \| v_N \|_{H^s}.
\end{equation*}  
Next, we suppose that $|k_{1,\dots, N}|  \ll |k_N|$ holds. Then, it follows that $|k_{N-1}| \sim |k_N|$, which leads that 
\begin{equation*}
\langle k_{1,\dots, N}  \rangle^{s-3} \langle k_N \rangle^3
\lesssim \langle k_{1,\dots, N} \rangle^{s-3} \langle k_{N-1} \rangle^s \langle k_N \rangle^s \langle k_N \rangle^{-2s+3} 
\lesssim \langle k_{1,\dots, N} \rangle^{-s} \langle k_{N-1} \rangle^s \langle k_N \rangle^s.  
\end{equation*}
Thus, by the Young inequality and the H\"{o}lder inequality, the left hand side of \eqref{nl11} is bounded by
\begin{align*} 
& \big\| \langle \cdot \rangle^{-s} \big( |\ha{v}_1|* \dots * |\ha{v}_{N-2}|* \langle \cdot  \rangle^s |\ha{v}_{N-1}| * \langle \cdot  \rangle^s |\ha{v}_N|  \big)  \big\|_{l^2} \\
& \lesssim \prod_{l=1}^{N-2} \| \ha{v}_l \|_{l^1} \| \langle \cdot \rangle^s| \ha{v}_{N-1}| * \langle \cdot \rangle^s |\ha{v}_N| \|_{l^{\infty}} 
\lesssim \prod_{l=1}^N \| v_l \|_{H^s}.
% \lesssim \prod_{l=1}^{N-2} \| v_l \|_{H^{1/2+}} \| v_{N-1} \|_{H^s} \| v_N \|_{H^s}. 
\end{align*}
\end{proof}
%%%%%%%%%%%%%%%%%%%%%%%%%%%%%%%%%%%%%%%%
%%%%%%%%%%%%%%%%%%%%%%%%%%%%%%%%%%%%%%%%
\begin{lem} \label{lem_nl3} 
Let $s >1$ and $i \in \{ 1, \dots, N-2 \}$. Then, for any $m \in \{ 1, \dots, N\}$, we have
\begin{align} \label{nl41}
\Big\| \sum_{k=k_{1,\dots, N}} \langle k_{1, \dots, N}  \rangle^{-1} \langle k_{\max} \rangle^{-2} \prod_{l=1}^N |\ha{v}_l (k_l) | \Big\|_{l_{s}^{2}} 
\lesssim \| v_m \|_{H^{s-3}} \prod_{l  \in \{ 1, \dots , N \} \setminus \{ m \} } \| v_l \|_{H^s}
\end{align}
and 
\begin{align} \label{nl42}
& \Big\| \sum_{k=k_{1, \dots, N}} \big[ \langle k_{i, i+1} \rangle^{-1} \langle k_i  \rangle \langle k_{i+1} \rangle \langle k_N \rangle^{-3} 
\, \chi_{H1}^{(N)}  \big]_{sym}^{(N)} \, \prod_{l=1}^N |\ha{v}_l (k_l)|  \Big\|_{l_s^2}  \notag  \\
& \hspace{1cm} \lesssim   \| v_m \|_{ H^{s-3} } 
\prod_{l  \in \{ 1, \dots , N \} \setminus \{ m \} } \| v_l \|_{H^s} 
\end{align}
\end{lem}
%%%%%%%%%%%%%%%%%%%%%%%%%%%%%%%%%%%%%%%%%%
\begin{proof}
Firstly, we prove (\ref{nl41}). It suffices to show 
\begin{equation} \label{nl410}
\Big\| \sum_{k=k_{1, \dots, N}} \langle k_{1, \dots, N}  \rangle^{s-1} \langle k_{\max} \rangle \, \prod_{l=1}^N |\ha{v}_l (k_l)| \Big\|_{l^2}
\lesssim \prod_{l=1}^N \| v_l \|_{H^s}
\end{equation}
for any $\{ v_l \}_{l=1}^N \subset H^s(\T)$. 
In a similar way to the proof of \eqref{nl11}, by the H\"{o}lder and Young inequalities, 
we obtain \eqref{nl410} for $s>1/2$. 

%%%%%%%%%%%%%%%%%%%%%%%%%%%%%%%%%%%%%%%%%%%%%%%%%%%%
% By symmetry, we assume $|k_1| \le |k_2| \le \dots \le |k_N|$. 
% First, we suppose that $|k_{1, \dots, N}| \sim |k_N| $ holds. Then, we easily check \eqref{nl410}. 
% Next, we suppose that $|k_{1, \dots, N}| \ll |k_N|$ holds. Then, it follows that $|k_{N-1}| \sim |k_N|$, which leads that  
% \begin{equation*}
% \langle k_{1, \dots , N} \rangle^{s-1} \langle k_N \rangle
% \lesssim \langle k_{1, \dots, N} \rangle^{s-1} \langle k_{N-1}  \rangle^s \langle k_N \rangle^s \langle k_N \rangle^{-2s+1} 
% \lesssim \langle k_{1,\dots, N} \rangle^{-s} \langle k_{N-1} \rangle^s \langle k_N \rangle^s. 
% \end{equation*}
% Thus, by the Young inequality and the H\"{o}lder inequality, we have \eqref{nl410}. 
%%%%%%%%%%%%%%%%%%%%%%%%%%%%%%%%%%%%%%%%%%%%%%%%%%%%%%%%%%%%%%%
Secondly, we prove \eqref{nl42}. It suffices to show 
\begin{equation} \label{nl420}
\Big\| \sum_{k=k_{1, \dots, N}} \langle k_{i, i+1}  \rangle^{-1} \langle k_i \rangle \langle k_{i+1} \rangle \langle k_N \rangle^s 
\, \chi_{H1}^{(N)} \prod_{l=1}^N |\ha{v}_l (k_l)|   \Big\|_{l^2} \lesssim \prod_{l=1}^N \| v_l \|_{H^s}
\end{equation}
for any $\{ v_l \}_{l=1}^N \subset H^s(\T)$. 
By the Young inequality and the H\"{o}lder inequality, the left hand side of \eqref{nl420} is bounded by 
\begin{align*}
& \prod_{l \in \{1 , \dots, N-1 \} \setminus \{ i, i+1 \} } \| \ha{v}_l \|_{l^1} 
\| \langle \cdot \rangle^{-1} ( \langle \cdot \rangle |\ha{v}_i| * \langle \cdot \rangle |\ha{v}_{i+1}| )  \|_{l^1} 
\| \langle \cdot \rangle^s |\ha{v}_N| \|_{l^2} \\
& \lesssim \prod_{l \in \{1 , \dots, N-1 \} \setminus \{ i, i+1 \} } \| v_l \|_{H^{1/2+}} \| v_i \|_{H^{1+}} \| v_{i+1} \|_{H^{1+}} 
\| v_N \|_{H^s}, 
\end{align*}
which implies that \eqref{nl420} holds.
\end{proof}
%%%%%%%%%%%%%%%%%%%%%%%%%%%%%%%%%%%%%%%%%%%%%%%%%%%%%%%
%%%%%%%%%%%%%%%%%%%%%%%%%%%%%%%%%%%%%%%%%%%%%%%%%%%%%%%
\begin{lem} \label{Le7}
Let  $s \ge 3/2$, $i \in \{ 2, \dots, N \}$ and an $N$-multiplier $m^{(N)}$ satisfy 
\begin{equation} \label{sb11}
\langle k_{1, \dots, N} \rangle^{s} |m^{(N)} (k_1, \cdots, k_N)| \lesssim \langle k_1 \rangle^{s-1/2+1/2i} \prod_{l=2}^i \langle k_l \rangle^{1+1/2i}.
\end{equation}
Then, for any $\{ v_l \}_{l=1}^N \subset H^s(\T)$, it follows that 
\begin{equation} \label{sb12}
\Big\| \sum_{k=k_{1, \dots, N}} |m^{(N)} (k_1, \dots, k_N) | \, \prod_{l=1}^N |\ha{v}_l (k_l)|  \Big\|_{l_s^2} \lesssim \prod_{l=1}^N \| v_l \|_{H^s}
\end{equation}
\end{lem}

%%%%%%%%%%%%%%%%%%%%%%%%%%%%%%%%%%%%%%%%%%%%%%%%%%%
%%%%%%%%%%%%%%%%%%%%%%%%%%%%%%%%%%%%%%%%%%%%%%%%%%%%%%%%%%%%%%%%%%%%%%
\begin{proof}
By \eqref{sb11}, the Young and the Schwarz inequalities, the left hand side of \eqref{sb12} is bounded by 
\begin{align} \label{sb13}
& \big\|  \langle \cdot  \rangle^{s-1/2+1/2i} |\ha{v}_1|* \langle \cdot \rangle^{1+1/2i} |\ha{v}_2|* \dots 
* \langle \cdot \rangle^{1+1/2i} |\ha{v}_i|* |\ha{v}_{i+1}|* \dots * |\ha{v}_N|  \big\|_{l^2} \notag \\
% & \lesssim \prod_{j=i+1}^N \| \ha{v}_j \|_{l^1} \big\|  \langle \cdot  \rangle^{s-1/2+1/2i} |\ha{v}_1|
% * \langle \cdot \rangle^{1+1/2i} |\ha{v}_2|* \dots * \langle \cdot \rangle^{1+1/2i} |\ha{v}_i| \big\|_{l^2} \notag \\
& \lesssim \prod_{j=i+1}^N \| v_j \|_{H^{1/2+}}
\big\|  \langle \cdot  \rangle^{s-1/2+1/2i} |\ha{v}_1|* \langle \cdot \rangle^{1+1/2i} |\ha{v}_2|
* \dots * \langle \cdot \rangle^{1+1/2i} |\ha{v}_i| \big\|_{l^2}. 
\end{align}
By the Plancherel theorem, the H\"{o}lder inequality and the Sobolev inequality, we have  
\begin{align} \label{sb14}
& \big\|  \langle \cdot  \rangle^{s-1/2+1/2i} |\ha{v}_1|* \langle \cdot \rangle^{1+1/2i} |\ha{v}_2|* \dots * \langle \cdot \rangle^{1+1/2i} |\ha{v}_i| \big\|_{l^2} \notag  \\
& =\big\| \mathcal{F}^{-1} \big[ \langle \cdot  \rangle^{s-1/2+1/2i} |\ha{v}_1|* \langle \cdot \rangle^{1+1/2i} |\ha{v}_2|* \dots * \langle \cdot \rangle^{1+1/2i} |\ha{v}_i| \big] \big\|_{L^2} \notag \\
& \le \| \mathcal{F}^{-1}[ \langle \cdot \rangle^{s-1/2+1/2i} |\ha{v}_1| ] \|_{L^{2i}} \prod_{l=2}^{i} \| \mathcal{F}^{-1} [ \langle \cdot \rangle^{1+1/2i} |\ha{v}_l| ]   \|_{L^{2i}} \lesssim \| v_1 \|_{H^s}\, \prod_{l=2}^{i}\| v_l \|_{H^{3/2}}.
\end{align}
Combining \eqref{sb13} with \eqref{sb14}, we obtain \eqref{sb12}. 
\end{proof}

%%%%%%%%%%%%%%%%%%%%%%%%%%%%%%%%%%%%%%%%%%%%%%%%%%%%%%%%
\begin{lem}\label{lem_go}
Let $f, g \in L^2(\T)$ and an $N$-multiplier $m^{(N)}$ satisfy
\begin{equation*} 
\Big\| \sum_{k=k_{1,\ldots,N}} |m^{(N)}(k_1,\ldots,k_N)| \prod_{l=1}^N |\ha{v}_l(t,k_l) | \Big\|_{L^\infty_Tl^2_s} \le C_0 \prod_{l=1}^N\|v_l\|_{L^\infty_TH^{s_l}}
\end{equation*}
for any $v_l \in C([-T,T];H^{s_l}(\T))$ with $l=1,2, \ldots,N$.
Then, for any \\
$v_l \in C([-T,T];H^{s_l}(\T))$ with $l=1,2,\ldots,N$, it follows that
\begin{equation*} 
\Big\| 
\Lambda_{f}^{(N)} ( m^{(N)} , \ha{v}_1,\ldots,\ha{v}_N ) -\Lambda_{g}^{(N)} ( m^{(N)} , \ha{v}_1,\ldots,\ha{v}_N )\Big\|_{L^\infty_Tl^2_s} \le C_*
\end{equation*}
where $C_*=C_*(C_0,v_1,\ldots,v_N,s_1,\ldots,s_N,|E_1(f)-E_1(g)|,|\ga|,T) \ge 0$ and $C_*\to 0$ when $|E_1(f)-E_1(g)|\to 0$.
\end{lem}

For details of the proof, see Lemma 2.11 in \cite{KTT}.

\section{the normal form reduction}

%%%%%%%%%%%%%%%%%%%%%%%%%%%%%%%
Our aim in this section is to remove the derivative losses in the right-hand side of \eqref{5mKdV3} by the normal form reduction, that is the differentiation by parts.
%%%%%%%%%%%%%%%%%%%%%%%%%%%%%%%%%%%%%%%%%%%%%
%%%%%%%%%%%%%%%%%%%%%%%%%%%%%%%%%%%%%%%%%%%%%
We  put
\EQQ{
& q_1^{(3)}:= -\frac{i}{3} k_{1,2,3} (k_{1,2}^2+ k_{2,3}^2  +k_{1,3}^2) , \quad 
q_2^{(3)}: = -\frac{i}{3} k_{1,2,3}(k_1k_2+k_2k_3+k_1 k_3),\\
& q_3^{(3)}:=-ik_1k_2 k_3,  \quad q_1^{(5)}:=ik_{1,2,3,4,5},\\
& Q_1^{(3)}:= \ga q_1^{(3)}, \quad  Q_2^{(3)}=  \be q_2^{(3)}, \quad Q_3^{(3)} := \al q_3^{(3)}, 
\quad Q^{(3)}:= Q_1^{(3)}+ Q_2^{(3)} + Q_3^{(3)}, \\
& Q_1^{(5)}:=  6 \de q_1^{(5)} (1-  [5\chi_{R1}^{(5)} ]_{sym}^{(5)}) 
+ \frac{4}{5} \ga^2 q_1^{(5)} [ \chi_{R1}^{(5)} (1-\chi_{R2}^{(5)} ) ]_{sym}^{(5)}.
} 
Note that all multipliers defined above are symmetric. Moreover, we define 
\EQQ{ 
q_2^{(5)}:= & \frac{9}{2} \Big\{ \frac{Q_1^{(3)}}{ \Phi_0^{(3)} } (k_1, k_2, k_{3,4,5}) Q_1^{(3)}(k_3, k_4, k_5)  
+ \frac{Q_1^{(3)}}{ \Phi_0^{(3)} } (k_3, k_4, k_{1,2,5}) Q_1^{(3)} (k_1, k_2, k_5)  \Big\}, \\
Q_2^{(5)} := &\frac{1}{3} \big\{ q_2^{(5)}\chi_{NR2}^{(5)} (1- \chi_{R4}^{(5)})  (k_1, k_2, k_3, k_4, k_5) \\
& +  q_2^{(5)} \chi_{NR2}^{(5)} (1- \chi_{R4}^{(5)})  (k_1, k_3, k_2, k_4, k_5) \\
& + q_2^{(5)} \chi_{NR2}^{(5)} (1- \chi_{R4}^{(5)}) (k_1, k_4, k_2, k_3, k_5) \big\}.
}
We notice that $Q_2^{(5)}$ is symmetric with $(k_1, k_2, k_3, k_4)$. 

%%%%%%%%%%%%%%%%%%%%%%%%%%%%%%%%%%%%%%%%%%%%
The main proposition in this section is  as below. 
%%%%%%%%%%%%%%%%%%%%%%%%%%%%%%%%%%%%%%%%%%
\begin{prop} \label{prop_NF2}
Let $s \ge 3/2$, $\vp \in L^2(\T)$, $L \gg \max \{ 1, |\ga| E_1(\vp) \}$, $T>0$ and 
$ u \in C([-T, T]:H^s(\T))$ be a solution of \eqref{5mKdV3}. 
Then $\ha{v}(t,k):= e^{-t \phi_{\vp}(k)} \ha{u} (t,k) $ satisfies the following equation for each $k \in \Z$:
\begin{align} \label{NF21}
\p_t (\ha{v}(t,k)+ F_{\vp, L} (\ha{v}) (t,k)) = G_{\vp, L} (\ha{v}) (t, k), 
\end{align} 
where
\begin{align*}
F_{\vp, L} (\ha{v})(t,k) := & \sum_{i=1}^{4} \La_{\vp}^{(3)} ( \ti{L}_{i, \vp}^{(3)} \chi_{>L}^{(3)},  \ha{v}(t) ) (t,k) 
+ \sum_{i=1}^{8} \Lambda_{\vp}^{(5)} ( \ti{L}_{i, \vp}^{(5)} \chi_{>L}^{(5)},  \ha{v}(t)) (t,k) \\ 
& + \sum_{i=1}^2  \Lambda_{\vp}^{(7)} ( \ti{L}_{i, \vp}^{(7)} \chi_{>L}^{(7)} ,  \ha{v} (t)) (t,k) 
\end{align*}
and 
\begin{align*}
& G_{\vp, L}(\ha{v}) (t,k) \\ 
& \hspace{0.5cm} := \sum_{i=1}^4 \Lambda_{\vp}^{(3)} ( \ti{L}_{i, \vp}^{(3)} \Phi_{\vp}^{(3)} \chi_{\le L}^{(3)} , \ha{v} (t) ) (t,k) 
+ \sum_{i=1}^{8} \Lambda_{\vp}^{(5)} (\ti{L}_{i, \vp}^{(5)} \Phi_{\vp}^{(5)} \chi_{\le L}^{(5)} ,  \ha{v} (t) ) (t,k) \\
& \hspace{0.5cm} + \sum_{i=1}^2 \Lambda_{\vp}^{(7)} (\ti{L}_{i, \vp}^{(7)} \Phi_{\vp}^{(7)} \chi_{ \le L}^{(7)}, \ha{v}(t)) (t,k)
+  \Lambda_{\vp}^{(3)} (\ti{M}_{1, \vp}^{(3)}, \ha{v}(t) ) (t,k) \\ 
& \hspace{0.5cm} + \sum_{i=1}^{23} \Lambda_{\vp}^{(5)} (\ti{M}_{i, \vp}^{(5)}, \ha{v} (t) ) (t,k)
+ \sum_{i=1}^{15} \Lambda_{\vp}^{(7)} (\ti{M}_{i,  \vp}^{(7)}, \ha{v}(t)) (t,k) \\
& \hspace{0.5cm} + \sum_{i=1}^3 \La_{\vp}^{(9)} (\ti{M}_{i, \vp}^{(9)}, \ha{v} (t) ) (t,k)
+  \Lambda_{\vp}^{(11)} (\ti{M}_{1, \vp}^{(11)} , \ha{v}(t) ) (t,k).
\end{align*}
%%%%%%%%%%%%%%%%%%%%%%%%%%%%%%%%%%%%%%%%%%%%%%%%%%%%%%%%%
%%%%%%%%%%%%%%%%%%%%%%%%%%%%%%%%%%%%%%%%%%%%%%%%%%%%%%%%%%%%%%
(i) The multipliers $\{L_{i, \vp}^{(3)}\}_{i=1}^4$, $\{L_{i, \vp}^{(5)}\}_{i=1}^{8}$ and $\{ L_{i, \vp}^{(7)} \}_{i=1}^2$ are defined as below:
\begin{align*}
L_{1, \vp}^{(3)}  & := -Q^{(3)} \chi_{NR1}^{(3)} \, 3 \chi_{H1}^{(2)}/\Phi_{\vp}^{(3)}, \hspace{0.5cm} 
L_{2, \vp}^{(3)}  := -Q^{(3)} \chi_{NR1}^{(3)} \, 3 \chi_{H2,1}^{(3)}/\Phi_{\vp}^{(3)}, \\ 
L_{3, \vp}^{(3)}  & := -Q^{(3)} \chi_{NR1}^{(3)} \, 3\chi_{H2,2}^{(3)}/\Phi_{\vp}^{(3)},
\hspace{0.5cm} L_{4, \vp}^{(3)}  := -Q^{(3)} \chi_{NR1}^{(3)} \chi_{H3}^{(3)}  /\Phi_{\vp}^{(3)},  
\end{align*}
and 
\begin{align*}
L_{1, \vp}^{(5)}  &:= -30 \de q_1^{(5)} \chi_{H1}^{(5)} (1 -\chi_{R1}^{(5)}) (1- \chi_{R5}^{(5)}) /\Phi_{\vp}^{(5)}, \\
L_{2, \vp}^{(5)}  &:= q_2^{(5)} \chi_{H1}^{(5)} \chi_{NR2}^{(5)} (1-\chi_{R4}^{(5)})/ \Phi_{\vp}^{(5)}, \\
L_{3, \vp}^{(5)}&:=  9 \Big[ \frac{Q_1^{(3)}}{ \Phi_0^{(3)} } \chi_{NR1}^{(3)}  \Big]_{ext1}^{(5)} [ Q_2^{(3)} \chi_{NR1}^{(3)} ]_{ext2}^{(5)} 
\, \chi_{H1}^{(5)} (1 -\chi_{R1}^{(5)}) (1-\chi_{R4}^{(5)})/ \Phi_{\vp}^{(5)}, \\
L_{4, \vp}^{(5)}&:=  9 \Big[ \frac{Q_2^{(3)}}{ \Phi_0^{(3)} } \chi_{NR1}^{(3)}  \Big]_{ext1}^{(5)} [ Q_1^{(3)} \chi_{NR1}^{(3)} ]_{ext2}^{(5)} 
\, \chi_{H1}^{(5)} (1- \chi_{R1}^{(5)}) (1- \chi_{R4}^{(5)}) / \Phi_{\vp}^{(5)}, \\
L_{5, \vp}^{(5)}&:=  9 \Big[ \frac{Q_2^{(3)}}{ \Phi_0^{(3)} } \chi_{NR1}^{(3)}  \Big]_{ext1}^{(5)} [ (Q_2^{(3)}+Q_3^{(3)}) \chi_{NR1}^{(3)} ]_{ext2}^{(5)} 
\, \chi_{H1}^{(5)} \chi_{A1}^{(5)} / \Phi_{\vp}^{(5)}, \\
L_{6, \vp}^{(5)}&:=  9 \Big[ \frac{Q_1^{(3)}}{ \Phi_0^{(3)} } \chi_{NR1}^{(3)}  \Big]_{ext1}^{(5)} [ Q_3^{(3)} \chi_{NR1}^{(3)} ]_{ext2}^{(5)} 
\, \chi_{H1}^{(5)} \chi_{A2}^{(5)} / \Phi_{\vp}^{(5)}, \\
L_{7, \vp}^{(5)}&:=  9 \Big[ \frac{Q^{(3)}}{ \Phi_0^{(3)} } \chi_{NR1}^{(3)}  \Big]_{ext1}^{(5)} [ Q^{(3)} \chi_{NR1}^{(3)} ]_{ext2}^{(5)} 
\, \chi_{NR(1,1)}^{(5)} (1- \chi_{H1}^{(5)}) \chi_{A1}^{(5)} / \Phi_{\vp}^{(5)}, \\
L_{8, \vp}^{(5)} &:= 9 \Big[ \frac{Q^{(3)}}{\Phi_{\vp}^{(3)}} \chi_{NR1}^{(3)} \chi_{>L}^{(3)} \Big]_{ext1}^{(5)} 
[ Q^{(3)} \chi_{NR1}^{(3)} ]_{ext2}^{(5)} \, \chi_{NR(2,1)}^{(5)} \chi_{A3}^{(5)} / \Phi_{\vp}^{(5)}, \\
L_{1, \vp}^{(7)} &:= -3 \Big[ \frac{ Q_2^{(5)} }{ \Phi_0^{(5)}  } \Big]_{ext1}^{(7)} [ Q_1^{(3)} \chi_{NR1}^{(3)} ]_{ext2}^{(7)} 
\, \chi_{H1}^{(7)} (1-\chi_{R1}^{(7)}) (1-  \chi_{R5}^{(7)})/ \Phi_{\vp}^{(7)}, \\
L_{2, \vp}^{(7)} &:= -3 \Big[ \frac{ Q_2^{(5)} }{ \Phi_0^{(5)}  } \Big]_{ext1}^{(7)} [ (Q_2^{(3)}+Q_3^{(3)}) \chi_{NR1}^{(3)} ]_{ext2}^{(7)} 
\, \chi_{H1}^{(7)} \chi_{A1}^{(7)}/ \Phi_{\vp}^{(7)}.
\end{align*}
%%%%%%%%%%%%%%%%%%%%%%%%%%%%%%%%%%%%%%%%%%%%%%%%%%%%%%%%%%
(ii) The multipliers $M_{1, \vp}^{(3)} $ and $\{M_{i, \vp}^{(5)}\}_{i=1}^{23}$ are defined as below:
\begin{equation*}
M_{1,\vp}^{(3)}:= Q^{(3)} \, 3\chi_{R3}^{(3)}
% M_{2, \vp}^{(3)}:= -Q^{(3)} \chi_{NR1}^{(3)} \chi_{H3}^{(3)} (1-\chi_{L1}^{(3)})
\end{equation*}
and
\begin{align*}
M_{1,\vp}^{(5)}&:= -\frac{4}{5} \ga^2 q_1^{(5)} \chi_{H1}^{(5)} \chi_{R1}^{(5)} (1- \chi_{R2}^{(5)}), \hspace{0.5cm}
M_{2, \vp}^{(5)} := - \frac{4}{5} \ga^2 q_1^{(5)} (1-\chi_{H1}^{(5)}) \chi_{R1}^{(5)} (1-\chi_{R2}^{(5)}), \\
M_{3, \vp}^{(5)}&:= -6 \de q_1^{(5)} (1- [5 \chi_{H1}^{(5)}]_{sym}^{(5)}), \hspace{0.5cm}
M_{4, \vp}^{(5)}:= -30 \de q_1^{(5)}  \chi_{H1}^{(5)} (1- \chi_{R1}^{(5)}) \chi_{R5}^{(5)}, \\
M_{5, \vp}^{(5)}&:= -30 \de q_1^{(5)} \chi_{R1}^{(5)} (1- \chi_{H1}^{(5)}), 
\end{align*}
\begin{align*}
M_{6, \vp}^{(5)} &:= [ 3 (\ti{L}_{1, \vp}^{(3)}+ \ti{L}_{2, \vp}^{(3)}+ \ti{L}_{3, \vp}^{(3)}+ \ti{L}_{4,  \vp}^{(3)} ) \chi_{>L}^{(3)} ]_{ext1}^{(5)} 
\big[Q^{(3)} [3 \chi_{R3}^{(3)}]_{sym}^{(3)} \big]_{ext2}^{(5)}, \\
M_{7, \vp}^{(5)} & := [ 3 ( \ti{L}_{2, \vp}^{(3)}+\ti{L}_{3, \vp}^{(3)}+ \ti{L}_{4, \vp}^{(3)}) \chi_{>L}^{(3)} ]_{ext1}^{(5)} 
\big[ - Q^{(3)} \chi_{NR1}^{(3)} \big]_{ext2}^{(5)}, \\
M_{8, \vp}^{(5)} &:= [ 3 \tilde{L}_{1, \vp}^{(3)} \chi_{>L}^{(3)} ]_{ext1}^{(5)} 
\big[ -Q^{(3)} \chi_{NR1}^{(3)} ( [3 \chi_{H2,2}^{(3)}]_{sym}^{(3)}+\chi_{H3}^{(3)}) \big]_{ext2}^{(5)}, \\ 
M_{9, \vp}^{(5)} &:= q_2^{(5)} \chi_{H1}^{(5)} \chi_{R1}^{(5)} (1- \chi_{R2}^{(5)}) , \hspace{0.5cm} 
M_{10, \vp}^{(5)}:= q_2^{(5)} \chi_{H1}^{(5)} \chi_{NR2}^{(5)} \chi_{R4}^{(5)},   \\
M_{11, \vp}^{(5)} &:= 9  \Big[  \frac{Q_1^{(3)}}{ \Phi_{0}^{(3)} } \chi_{NR1}^{(3)} \Big]_{ext1}^{(5)} [Q_2^{(3)} \chi_{NR1}^{(3)} ]_{ext2}^{(5)}  \, 
\chi_{H1}^{(5)} \chi_{R1}^{(5)}, \\
M_{12, \vp}^{(5)} &:= 9  \Big[  \frac{Q_1^{(3)}}{ \Phi_{0}^{(3)} } \chi_{NR1}^{(3)} \Big]_{ext1}^{(5)} [Q_2^{(3)} \chi_{NR1}^{(3)} ]_{ext2}^{(5)}  \, 
\chi_{H1}^{(5)} (1-\chi_{R1}^{(5)}) \chi_{R4}^{(5)}, \\
M_{13, \vp}^{(5)} &:= 9  \Big[  \frac{Q_2^{(3)}}{ \Phi_{0}^{(3)} } \chi_{NR1}^{(3)} \Big]_{ext1}^{(5)} [Q_1^{(3)} \chi_{NR1}^{(3)} ]_{ext2}^{(5)}  \, 
\chi_{H1}^{(5)} \chi_{R1}^{(5)}, \\
M_{14, \vp}^{(5)} &:= 9  \Big[  \frac{Q_2^{(3)}}{ \Phi_{0}^{(3)} } \chi_{NR1}^{(3)} \Big]_{ext1}^{(5)} [Q_1^{(3)} \chi_{NR1}^{(3)} ]_{ext2}^{(5)}  \, 
\chi_{H1}^{(5)} (1-\chi_{R1}^{(5)}) \chi_{R4}^{(5)}, \\
M_{15, \vp}^{(5)} &:= 9  \Big[  \frac{Q_2^{(3)}}{ \Phi_{0}^{(3)} } \chi_{NR1}^{(3)} \Big]_{ext1}^{(5)} 
[(Q_2^{(3)}+Q_3^{(3)} ) \chi_{NR1}^{(3)} ]_{ext2}^{(5)}  \, \chi_{H1}^{(5)} (1-\chi_{A1}^{(5)}), \\
M_{16, \vp}^{(5)} &:= 9  \Big[  \frac{Q_1^{(3)}}{ \Phi_{0}^{(3)} } \chi_{NR1}^{(3)} \Big]_{ext1}^{(5)} [Q_3^{(3)} \chi_{NR1}^{(3)} ]_{ext2}^{(5)}  \, 
\chi_{H1}^{(5)} (1-\chi_{A2}^{(5)}), \\
M_{17, \vp}^{(5)} &:= 9  \Big[  \frac{Q_3^{(3)}}{ \Phi_{0}^{(3)} } \chi_{NR1}^{(3)} \Big]_{ext1}^{(5)} [Q^{(3)} \chi_{NR1}^{(3)} ]_{ext2}^{(5)}  \, 
\chi_{H1}^{(5)}, \\
M_{18, \vp}^{(5)} &:= 9  \Big[  \frac{Q^{(3)}}{ \Phi_{0}^{(3)} } \chi_{NR1}^{(3)} \Big]_{ext1}^{(5)} [Q^{(3)} \chi_{NR1}^{(3)} ]_{ext2}^{(5)}  \, \chi_{NR(1,1)}^{(5)} (1-\chi_{H1}^{(5)}) (1-\chi_{A1}^{(5)}), \\
M_{19, \vp}^{(5)} &:= 9  \Big[   \Big( \frac{Q^{(3)}}{ \Phi_{\vp}^{(3)}}- \frac{Q^{(3)}}{ \Phi_0^{(3)} } \Big) \chi_{NR1}^{(3)} \chi_{>L}^{(3)} \Big]_{ext1}^{(5)} [Q^{(3)} \chi_{NR1}^{(3)} ]_{ext2}^{(5)}  \, \chi_{NR(1,1)}^{(5)}, \\
M_{20, \vp}^{(5)} &:= 9  \Big[   \Big(- \frac{Q^{(3)}}{ \Phi_{0}^{(3)}} \Big) \chi_{NR1}^{(3)} \chi_{\le L}^{(3)} \Big]_{ext1}^{(5)} 
[Q^{(3)} \chi_{NR1}^{(3)} ]_{ext2}^{(5)}  \, \chi_{NR(1,1)}^{(5)}, \\
M_{21, \vp}^{(5)} &:=  18 \Big[ \frac{Q^{(3)}}{ \Phi_{\vp}^{(3)} } \chi_{NR1}^{(3)} \chi_{>L}^{(3)}  \Big]_{ext1}^{(5)} [ -Q^{(3)} \chi_{NR1}^{(3)} ]_{ext2}^{(5)} \, \chi_{NR(1,2)}^{(5)}, \\
M_{22, \vp}^{(5)} &:=  9 \Big[ \frac{Q^{(3)}}{ \Phi_{\vp}^{(3)} } \chi_{NR1}^{(3)} \chi_{>L}^{(3)}  \Big]_{ext1}^{(5)} [ -Q^{(3)} \chi_{NR1}^{(3)} ]_{ext2}^{(5)} \, \chi_{NR(2,1)}^{(5)} (1- \chi_{A3}^{(5)})  , \\
M_{23, \vp}^{(5)} &:=  18 \Big[ \frac{Q^{(3)}}{ \Phi_{\vp}^{(3)} } \chi_{NR1}^{(3)} \chi_{>L}^{(3)}  \Big]_{ext1}^{(5)} [ -Q^{(3)} \chi_{NR1}^{(3)} ]_{ext2}^{(5)} \, \chi_{NR(2,2)}^{(5)}.
\end{align*}
%%%%%%%%%%%%%%%%%%%%%%%%%%%%%%%%%%%%%%%
(iii) The multipliers $\{M_{i, \vp}^{(7)}\}_{i=1}^{15}$ are defined as below:
\begin{align*}
M_{1, \vp}^{(7)} &:= [ 3 ( \tilde{L}_{1, \vp}^{(3)}+\ti{L}_{2, \vp}^{(3)}+ \ti{L}_{3, \vp}^{(3)}+ \ti{L}_{4, \vp}^{(3)} ) \chi_{>L}^{(3)}  ]_{ext1}^{(7)} 
[-Q_1^{(5)}]_{ext2}^{(7)} ,  \\
M_{2, \vp}^{(7)}& := 
\big[ 5 \sum_{i=1}^8 \ti{L}_{i,  \vp}^{(5)} \chi_{>L}^{(5)} \big]_{ext1}^{(7)} 
\big[Q^{(3)} [ 3 \chi_{R3}^{(3)} ]_{sym}^{(3)} \big]_{ext2}^{(7)} , \\  
M_{3, \vp}^{(7)}& := 
\big[5 \big( \ti{L}_{1, \vp}^{(5)}+ \ti{L}_{7, \vp}^{(5)}+ \ti{L}_{8, \vp}^{(5)} \big) \chi_{>L}^{(5)} \big]_{ext1}^{(7)} 
[-Q^{(3)} \chi_{NR1}^{(3)}  ]_{ext2}^{(7)} , \\  
M_{4, \vp}^{(7)}& := 
\big[5 \big( \ti{L}_{3, \vp}^{(5)}+ \ti{L}_{4, \vp}^{(5)}+ \ti{L}_{5, \vp}^{(5)} +\ti{L}_{6, \vp}^{(5)} \big) \chi_{>L}^{(5)} \big]_{ext1}^{(7)} 
[-Q^{(3)} \chi_{NR1}^{(3)}  ]_{ext2}^{(7)} , \\  
M_{5, \vp}^{(7)}& := [ 5 \ti{L}_{2, \vp}^{(5)} \chi_{>L}^{(5)}  ]_{ext1}^{(7)} 
[ -Q^{(3)} \chi_{NR1}^{(3)} ([ 3 \chi_{H2,2}^{(3)} ]_{sym}^{(3)}+ \chi_{H3}^{(3)})  ]_{ext2}^{(7)}, \\
M_{6, \vp}^{(7)}&:= -3  \Big[ \frac{Q_2^{(5)}}{ \Phi_0^{(5)} } \Big]_{ext1}^{(7)} [Q_1^{(3)} \chi_{NR1}^{(3)} ]_{ext2}^{(7)} \,
\chi_{H1}^{(7)} \chi_{R1}^{(7)} \chi_{A4}^{(7)}, \\ 
M_{7, \vp}^{(7)}&:= -3  \Big[ \frac{Q_2^{(5)}}{ \Phi_0^{(5)} } \Big]_{ext1}^{(7)} [Q_1^{(3)} \chi_{NR1}^{(3)} ]_{ext2}^{(7)} \,
\chi_{H1}^{(7)} \chi_{R1}^{(7)} (1-\chi_{A4}^{(7)}), \\ 
M_{8, \vp}^{(7)}&:= -3  \Big[ \frac{Q_2^{(5)}}{ \Phi_0^{(5)} } \Big]_{ext1}^{(7)} [Q_1^{(3)} \chi_{NR1}^{(3)} ]_{ext2}^{(7)} \,
\chi_{H1}^{(7)} (1-\chi_{R1}^{(7)}) \chi_{R5}^{(7)}, \\ 
M_{9, \vp}^{(7)}&:= -3  \Big[ \frac{Q_2^{(5)}}{ \Phi_0^{(5)} } \Big]_{ext1}^{(7)} [(Q_2^{(3)}+ Q_3^{(3)}) \chi_{NR1}^{(3)} ]_{ext2}^{(7)} \,
\chi_{H1}^{(7)} (1-\chi_{A1}^{(7)}), \\ 
M_{10, \vp}^{(7)}&:= -3  \Big[ \frac{Q_2^{(5)}}{ \Phi_0^{(5)} } \Big]_{ext1}^{(7)} [Q^{(3)} \chi_{NR1}^{(3)} ]_{ext2}^{(7)} \,
\chi_{NR(1,1)}^{(7)} (1- \chi_{H1}^{(7)}), \\ 
M_{11, \vp}^{(7)}&:= -3  \Big[\Big( \frac{Q_2^{(5)}}{ \Phi_{\vp}^{(5)} }- \frac{Q_2^{(5)}}{\Phi_{0}^{(5)}} \Big) \chi_{>L}^{(5)} \Big]_{ext1}^{(7)} 
[Q^{(3)} \chi_{NR1}^{(3)} ]_{ext2}^{(7)} \, \chi_{NR(1,1)}^{(7)}, \\
M_{12, \vp}^{(7)}&:= -3  \Big[\Big( - \frac{Q_2^{(5)}}{\Phi_{0}^{(5)}} \Big) \chi_{\le L}^{(5)} \Big]_{ext1}^{(7)} 
[Q^{(3)} \chi_{NR1}^{(3)} ]_{ext2}^{(7)} \, \chi_{NR(1,1)}^{(7)}, \\
M_{13, \vp}^{(7)}&:= -12  \Big[ \big[5 \frac{Q_2^{(5)}}{ \Phi_{\vp}^{(5)} } \chi_{>L}^{(5)} \chi_{H1}^{(5)} \big]_{sym}^{(5)} \Big]_{ext1}^{(7)} 
[Q^{(3)} \chi_{NR1}^{(3)} ]_{ext2}^{(7)} \, \chi_{NR(1,2)}^{(7)}, \\
M_{14, \vp}^{(7)}&:= -3  \Big[ \frac{Q_2^{(5)}}{ \Phi_{\vp}^{(5)} } \chi_{>L}^{(5)} \Big]_{ext1}^{(7)} [Q^{(3)} \chi_{NR1}^{(3)} ]_{ext2}^{(7)} \, \chi_{NR(2,1)}^{(7)}, \\
M_{15, \vp}^{(7)}&:= -12  \Big[ \big[5 \frac{Q_2^{(5)}}{ \Phi_{\vp}^{(5)} } \chi_{>L}^{(5)} \chi_{H1}^{(5)} \big]_{sym}^{(5)} \Big]_{ext1}^{(7)}
 [Q^{(3)} \chi_{NR1}^{(3)} ]_{ext2}^{(7)} \, \chi_{NR(2,2)}^{(7)}. \\
\end{align*}

(iv) The multipliers $\{M_{i, \vp}^{(9)}\}_{i=1}^{3}$ and $M_{1, \vp}^{(11)}$ are defined as below:
\begin{align*}
M_{1, \vp}^{(9)} & := \big[ 5 \sum_{i=1}^{8} \ti{L}_{i, \vp}^{(5)} \chi_{>L}^{(5)} ]_{ext1}^{(9)} 
[-Q_1^{(5)}]_{ext2}^{(9)}, \\
M_{2, \vp}^{(9)} & := [ 7 (\ti{L}_{1, \vp}^{(7)}+ \ti{L}_{2, \vp}^{(7)}) \chi_{>L}^{(7)} ]_{ext1}^{(9)} 
[Q^{(3)} [3 \chi_{R3}^{(3)} ]_{sym}^{(3)} ]_{ext2}^{(9)}, \\
M_{3, \vp}^{(9)} & := [ 7 (\ti{L}_{1, \vp}^{(7)}+\ti{L}_{2, \vp}^{(7)}) \chi_{>L}^{(7)} ]_{ext1}^{(9)} [-Q^{(3)} \chi_{NR1}^{(3)} ]_{ext2}^{(9)}, \\
M_{1, \vp}^{(11)}  & := [ 7 (\ti{L}_{1, \vp}^{(7)} + \ti{L}_{2, \vp}^{(7)}) \chi_{>L}^{(7)} ]_{ext1}^{(11)} [-Q_1^{(5)}]_{ext2}^{(11)}.
\end{align*}
\end{prop}
%%%%%%%%%%%%%%%%%%%%%%%%%%%%%%%%%%%%%%%%%%%%%%
\begin{rem}
Precisely speaking, $L_{1, \vp}^{(3)}  := -Q^{(3)} \chi_{NR1}^{(3)} \, 3 \chi_{H1}^{(3)}/\Phi_{\vp}^{(3)}$ means
\begin{align*}
L_{1, \vp}^{(3)} :=
\begin{cases}
-Q^{(3)} \chi_{NR1}^{(3)} \, 3 \chi_{H1}^{(3)}/\Phi_{\vp}^{(3)} \ \ &\text{ when } \Phi_{\vp}^{(3)}(k_1,k_2, k_3)\neq 0\\
0\ \ &\text{ when } \Phi_{\vp}^{(3)}(k_1,k_2, k_3)= 0.
\end{cases}
\end{align*}
We adapt the same rule in the definition of each multiplier when its denominator is equal to $0$.
\end{rem}
%%%%%%%%%%%%%%%%%%%%%%%%%%%%%%%%%%%%%%%%%%%%%

Before we prove Proposition \ref{prop_NF2}, we prepare some propositions and lemmas.
%%%%%%%%%%%%%%%%%%%%%%%%%%%%%%%%%%%%%%%%%%%%
\begin{prop} \label{prop_req1}
Let $s \ge 3/2$, $\vp \in L^2(\T)$, $T>0$ and $u \in C([-T, T]:H^s(\T))$ be a solution of \eqref{5mKdV3}. 
Put $v:= U_\vp(-t)u$. Then, $v \in C([-T, T]: H^s(\T))\cap C^1([-T, T]: H^{s-3}(\T))$ and 
$\ha{v}(t,k)$ satisfies 
\begin{align} \label{eq20} 
\p_t \ha{v} (t,k) =&  \La_{\vp}^{(3)}( -Q^{(3)} \chi_{NR1}^{(3)}, \ha{v} (t) )(t,k) 
+ \La_{\vp}^{(3)}( Q^{(3)} [3 \chi_{R3}^{(3)}]_{sym}^{(3)}, \ha{v} (t) )(t,k) \notag \\
& + \La_{\vp}^{(5)} (-Q_1^{(5)}, \ha{v} (t)) (t,k)
\end{align}
for each $k \in \Z$, where $U_{\vp}(t)= \mathcal{F}^{-1} \exp( t \phi_{\vp} (k)) \mathcal{F}_x$. Moreover, we have 
\begin{align}
\|\p_t v\|_{L_T^\infty H^{s-3}}\lec \|v\|_{L_T^\infty H^{s}}^3+\|v\|_{L_T^\infty H^{s}}^5. \label{eq_es}
\end{align}
\end{prop}
%%%%%%%%%%%%%%%%%%%%%%%%%%%%%%%%%%%
\begin{proof}
Since $u\in C([-T, T]:H^s(\T))$ satisfies \eqref{5mKdV3} and $\ha{v}(t,k)=e^{-t\phi_{\vp} (k)} \ha{u}(t,k)$, it follows that
$v\in C([-T,T]: H^{s}(\T))$ and
\begin{equation}\label{EE2}
\p_t \ha{v} (t,k) = \sum_{j=1,2,3,4} e^{-t \phi_{\vp} (k) } \mathcal{F}_x [J_j(u)](t, k)
\end{equation}
By Lemma~\ref{lem_nl2}, it follows that $J_1(u), J_2(u), J_3(u), J_4(u) \in C([-T, T]: H^{s-3}(\T))$. 
Thus, we obtain $v\in C^1([-T,T]: H^{s-3}(\T))$ and \eqref{eq_es}. 
A direct computation yields that  
\EQQ{
& e^{-t \phi_{\vp} (k) } \mathcal{F}_x \big[ -\al (\p_x u)^3 - \be \p_x (u (\p_x u)^2)- \ga \p_x (u \p_x^2 (u^2))  \big] (t,k) \\
&\hspace{0.5cm}  =\sum_{k=k_{1,2,3}} e^{- t \Phi_{\vp}^{(3)}} 
\big( -\al q_3^{(3)} - \be q_2^{(3)}  -\ga q_1^{(3)} \big) \prod_{j=1}^3 \ha{v}(t,k_j) \\
& \hspace{0.5cm} =\sum_{k=k_{1,2,3 }} e^{- t \Phi_{\vp}^{(3)}}
(-Q^{(3)}) \prod_{j=1}^3 \ha{v}(t,k_j) 
}
and 
\EQQ{
& e^{-t \phi_{\vp} (k) } \mathcal{F}_x \Big[ (3 \al + \be) \int_{\T} (\p_x u)^2 \, dx \, \p_x u 
+ 2 \ga \int_{\T} u^2 \, dx \, \p_x^3 u - 2 \ga \int_{\T} (\p_x u)^2 \, dx \, \p_x u \Big](t,k) \\
&\hspace{0.5cm} = \sum_{k=k_{1,2,3}} e^{- t \Phi_{\vp}^{(3)}} Q^{(3)}  [3 \chi_{R1}^{(3)}]_{sym}^{(3)} 
\prod_{j=1}^3 \ha{v}(t,k_j).
% & \hspace{0.5cm} = \La_{\vp}^{(3)} ( Q^{(3)} [3 \chi_{R1}^{(3)}]_{sym}^{(3)}, \ha{v}(t) ) (t, k).
} 
Note that $1=\chi_{NR1}^{(3)} + [3 \chi_{R1}^{(3)}]_{sym}^{(3)} - [3 \chi_{R3}^{(3)}]_{sym}^{(3)} + \chi_{R6}^{(3)}$ and 
$Q^{(3)} \chi_{R6}^{(3)}=0$ where 
\begin{align*}
\chi_{R6}^{(3)} = 
\begin{cases}
\dis 1, \hspace{0.2cm} \text{when} \hspace{0.2cm} k_1=k_2=k_3=0 \\
0, \hspace{0.2cm} \text{otherwise}
\end{cases}.
\end{align*}
Thus, we have 
\begin{align} \label{EE4}
& e^{-t \phi_{\vp} (k) } \mathcal{F}_x [ J_2 (u)+J_3(u) ] (t,k) 
= \sum_{k=k_{1,2,3}} e^{- t \Phi_{\vp}^{(3)}} 
(-Q^{(3)}) (1- [3 \chi_{R1}^{(3)}]_{sym}^{(3)}) \prod_{j=1}^3 \ha{v} (t, k_j) \notag \\
& \hspace{0.5cm} = \sum_{k=k_{1,2, 3}} e^{- t \Phi_{\vp}^{(3)}} 
(-Q^{(3)}) (\chi_{NR1}^{(3)} - [3 \chi_{R3}^{(3)}]_{sym}^{(3)} ) \prod_{j=1}^3 \ha{v}(t,k_j) \notag \\
& \hspace{0.5cm}=
\La_{\vp}^{(3)} (- Q^{(3)} \chi_{NR1}^{(3)}, \hat{v}(t) ) (t,k)+ \La_{\vp}^{(3)} ( Q^{(3)} [3 \chi_{R3}^{(3)} ]_{sym}^{(3)}, \ha{v}(t)) (t, k).
\end{align}
%% Here we used $Q^{(3)} \chi_{R6}^{(3)} =0$ in the second equality. 
Since
\EQQ{
& e^{-t \phi_{\vp} (k) } \mathcal{F}_x \big[ -30 \de\,  u^4 \, \p_x u \big] (t,k)
= \sum_{k=k_{1,2,3,4,5}} e^{- t \Phi_{\vp}^{(5)}} (- 6 \de q_1^{(5)}) \prod_{j=1}^5 \ha{v}(t,k_j) \\
& e^{-t \phi_{\vp} (k) } \mathcal{F}_x \Big[  \int_{\T} u^4 \, dx \,  \p_x u  \Big](t,k)
= \sum_{k=k_{1,2,3,4,5}}  e^{- t \Phi_{\vp}^{(5)}}  q_1^{(5)}  [\chi_{R1}^{(5)} ]_{sym}^{(5)} \prod_{j=1}^5 \ha{v}(t,k_j),
}
and
\EQQ{
e^{-t \phi_{\vp} (k) } \mathcal{F}_x \Big[ \Big( \int_{\T} u^2 \, dx \Big)^2  \p_x u  \Big](t,k)
= \sum_{k=k_{1,2,3,4,5}}  e^{- t \Phi_{\vp}^{(5)}}  q_1^{(5)}  [\chi_{R1}^{(5)} \chi_{R2}^{(5)} ]_{sym}^{(5)} \prod_{j=1}^5 \ha{v}(t,k_j),
}
we have
\begin{align} \label{EE5}
 e^{-t \phi_{\vp} (k) } \mathcal{F}_x [ J_1 (u ) ] (t,k)
 &= \sum_{k=k_{1,2,3,4,5} } e^{- t \Phi_{\vp}^{(5)}} (-6 \de) q_1^{(5)} (1-[5 \chi_{R1}^{(5)}]_{sym}^{(5)}) 
 \prod_{j=1}^5 \ha{v}(t,k_j)  \notag \\
&= \La_{\vp}^{(5)} \big( -6 \de q_1^{(5)} (1- [5 \chi_{R1}^{(5)}]_{sym}^{(5)})   , \ha{v} (t) \big) (t,k). 
\end{align}
and 
\begin{align} \label{EE6}
e^{-t \phi_{\vp} (k) } \mathcal{F}_x [J_4(u)] (t,k) & = 
\sum_{k=k_{1,2,3,4,5} } e^{- t \Phi_{\vp}^{(5)}} \Big( -\frac{4}{5} \ga^2 \Big) q_1^{(5)}  [\chi_{R1}^{(5)} (1-\chi_{R2}^{(5)}) ]_{sym}^{(5)} 
 \prod_{j=1}^3 \ha{v}(t ,k_j) \notag \\
& = \La_{\vp}^{(5)} \Big( -\frac{4}{5} \ga^2 q_1^{(5)} [\chi_{R1}^{(5)} (1- \chi_{R2}^{(5)}) ]_{sym}^{(5)}, \ha{v} (t)  \Big) (t,k) .
\end{align}
Therefore, collecting (\ref{EE2})--(\ref{EE6}), we obtain (\ref{eq20}). 
\end{proof}

%%%%%%%%%%%%%%%%%%%%%%%%%%%%%%%%%%%%%
%%%%%%%%%%%%%%%%%%%%%%%%%%%%%%%%%%%%%%
\begin{prop} \label{prop_NF11}
Let $N \in \{ 3,5,7 \}$, $s \ge 3/2$, $\vp \in L^2(\T)$, $T>0$ and an $N$-multiplier $m^{(N)}$ satisfy
\begin{align*}
|m^{(N)} (k_1, \dots , k_{N} ) | \lesssim 1, 
\hspace{0.5cm} | (m^{(N)} \Phi_{\vp}^{(N)}) (k_1, \dots, k_N) | \lesssim \langle k_{\max} \rangle^3 .
\end{align*} 
If $ u \in C([-T, T]:H^s(\T))$ is a solution of \eqref{5mKdV3}, then 
$\ha{v}(t,k)=e^{-t \phi_{\vp} (k)} \ha{u}(t,k) $ satisfies the following equation for each $k \in \Z$:
\begin{align}
\label{NF11}
& \p_t \Lambda_{\vp}^{(N)} (\ti{m}^{(N)}, \ha{v}(t))(t,k) = 
\Lambda_{\vp}^{(N)} (-\ti{m}^{(N)} \Phi_{\vp}^{(N)}, \ha{v} (t) )(t,k) \notag \\
& \hspace{0.3cm} 
+ \Lambda_{\vp}^{(N+2)} \big( \big[ [N \ti{m}^{(N)} ]_{ext1}^{(N+2)} [ -Q^{(3)} \chi_{NR1}^{(3)} ]_{ext2}^{(N+2)}  \big]_{sym}^{(N+2)}, 
\ha{v} (t) \big)(t,k) \nonumber \\
& \hspace{0.3cm} 
+ \Lambda_{\vp}^{(N+2)} \big( \big[ [N \ti{m}^{(N)} ]_{ext1}^{(N+2)} [ Q^{(3)} [3 \chi_{R3}^{(3)}]_{sym}^{(3)} ]_{ext2}^{(N+2)}  \big]_{sym}^{(N+2)}, \ha{v} (t) \big)(t,k) \nonumber \\
& \hspace{0.3cm} 
+ \Lambda_{\vp}^{(N+4)} \big( \big[ [N \ti{m}^{(N)} ]_{ext1}^{(N+4)} [ -Q_1^{(5)} ]_{ext2}^{(N+4)}  \big]_{sym}^{(N+4)}, \ha{v} (t) \big)(t,k).
\end{align} 
\end{prop}

%%%%%%%%%%%%%%%%%%%%%%%%%%%%%%%%%%%%%%%%%%%%%%
Note that any $L_{j, \vp}^{(N)}$ defined in Proposition~\ref{prop_NF2} satisfies 
\begin{align*}
|L_{j, \vp}^{(N)} \chi_{>L}^{(N)} | \lesssim 1,
\hspace*{2em} 
|L_{j, \vp}^{(N)} \Phi_{\vp}^{(N)} \chi_{>L}^{(N)} | \lesssim \langle k_{\max} \rangle^3
\end{align*}
for $L \gg \max\{ 1, |\ga| E_1(\vp)  \}$. Thus, we can apply Proposition~\ref{prop_NF11} 
with $\ti{m}^{(N)}=\ti{L}_{j, \vp}^{(N)} \chi_{>L}^{(N)}$. 
For details, see Remark~\ref{rem_pwb11}. 

%%%%%%%%%%%%%%%%%%%%%%%%%%%%%%%%%%%%%%%%%%%%%%%%%%%
\begin{proof}[Proof of Proposition~\ref{prop_NF11}] 
At least formally, we have
\begin{align}\label{NF112}
\p_t \La_{\vp}^{(N)} (\ti{m}^{(N)}, \ha{v} (t))(t, k)
&=\sum_{k=k_{1, \dots, N}} \frac{\p}{\p t} \bigg\{ e^{-t \Phi_{\vp}^{(N)}} \ti{m}^{(N)} \prod_{j=1}^N \ha{v} (t, k_j) \bigg\}\\
&= \mathcal{N}_1(t,k) +\mathcal{N}_2(t, k) \notag
\end{align}
where
\begin{align*}
\mathcal{N}_1(t,k)&=  \sum_{k=k_{1,\cdots, N}}  e^{-t \Phi_{\vp}^{(N)}} (-\ti{m}^{(N)} \Phi_{\vp}^{(N)} ) \prod_{j=1}^N \ha{v}(t , k_j) \\
& =\Lambda_{\vp}^{(N)} (-\ti{m}^{(N)} \Phi_{\vp}^{(N)}, \ha{v})(t,k) \\
\mathcal{N}_2(t, k)&= \sum_{k=k_{1,\cdots, N}}  e^{-t \Phi_{\vp}^{(N)}} \ti{m}^{(N)} \frac{\p}{\p t} \, \prod_{j=1}^N \ha{v}(t , k_j).
\end{align*}
By $| \ti{m}^{(N)} \Phi_{\vp}^{(N)} | \le [ |m^{(N)} \Phi_{\vp}^{(N)} | ]_{sym}^{(N)}\lesssim  \langle k_{\max} \rangle^3 $ and Lemma~\ref{lem_nl2}, 
\begin{align*} 
\Big\| \sum_{k=k_{1,\cdots, N}}  |e^{-t \Phi_{\vp}^{(N)}}| |(-\ti{m}^{(N)} \Phi_{\vp}^{(N)} )| \prod_{j=1}^N |\ha{v}(t , k_j)| \Big\|_{L_T^{\infty} l_{s-3}^2} \lesssim \| v \|_{L_T^{\infty} H^s}^N
\end{align*}
and by $|\ti{m}^{(N)}| \lesssim 1 $,
Lemma~\ref{lem_nl2} and \eqref{eq_es} in Proposition \ref{prop_req1},
\begin{align*}
&\Big\| \sum_{k=k_{1,\dots, N}}  |e^{-t \Phi_{\vp}^{(N)}}| |\ti{m}^{(N)} | \sum_{i=1}^N |\p_t \ha{v}(t , k_i)|\, \prod_{j\in \{1,\ldots,N\}\setminus \{i\}} |\ha{v}(t , k_j)|  \Big\|_{L_T^{\infty} l_{s-3}^2  }\\
& \lesssim \Big\| \sum_{k=k_{1,\cdots, N}}  
\langle k_{\max} \rangle^3  \sum_{i=1}^N \langle k_i \rangle^{-3}|\p_t \ha{v}(t , k_i)|\, \prod_{j\in \{1,\ldots,N\}\setminus \{i\}} |\ha{v}(t , k_j)|  \Big\|_{L_T^{\infty} l_{s-3}^2  }\\
& \lesssim \| \p_t v \|_{L_T^{\infty} H^{s-3}} \| v \|_{L_T^{\infty} H^s }^{N-1} 
\lesssim \| v \|_{L_T^{\infty} H^s}^{N+2} + \| v \|_{L_T^{\infty} H^s}^{N+4}. 
\end{align*}
Thus, for each $k\in \Z$, the convergence  of $\mathcal{N}_1(t,k)$ and $\mathcal{N}_2(t,k)$ is absolute and uniform in $t \in [-T, T]$. Therefore, changing the sum and the time differentiation in \eqref{NF112} can be verified strictly. 
Moreover,
\begin{align*} 
\mathcal{N}_2 (t, k) &= \sum_{k=k_{1,\cdots, N}}  e^{-t \Phi_{\vp}^{(N)}}  N \ti{m}^{(N)} \, \p_t \ha{v} (t, k_N)\prod_{j=1}^{N-1} \ha{v}(t , k_j)  \nonumber \\
&= \La_{\vp}^{(N)} \big(N \ti{m}^{(N)}, \ha{v}(t), \dots, \ha{v}(t), \p_t \ha{v}(t) \big) (t,k)
\end{align*}
Substituting \eqref{eq20} for it, we get
\begin{align*}
\mathcal{N}_2(t,k)  
 = & \La_{\vp}^{(N)} \big(N \ti{m}^{(N)}, \ha{v}(t), \dots, \ha{v}(t), \La_{\vp}^{(3)} (-Q^{(3)} \chi_{NR1}^{(3)}, \ha{v} (t)  )   \big) (t,k) \nonumber \\
 & \hspace{0.3cm}
+ \La_{\vp}^{(N)} \big(N \ti{m}^{(N)}, \ha{v}(t), \dots, \ha{v}(t), \La_{\vp}^{(3)} (Q^{(3)} [ 3 \chi_{R3}^{(3)} ]_{sym}^{(3)}, \ha{v} (t)  ) \big) (t,k) \notag \\ 
 & \hspace{0.3cm}
 +   \La_{\vp}^{(N)} \big(N \ti{m}^{(N)}, \ha{v}(t), \dots, \ha{v}(t), \La_{\vp}^{(5)}  (-Q_1^{(5)}, \ha{v} (t) )  \big) (t,k) \nonumber \\
 = &  \La_{\vp}^{(N+2)} \big( [ N \ti{m}^{(N)}]_{ext1}^{(N+2)} [ -Q^{(3)} \chi_{NR1}^{(3)} ]_{ext2}^{(N+2)}, \ha{v}(t) \big) (t,k) \nonumber \\
 & \hspace{0.3cm} 
 +  \La_{\vp}^{(N+2)} \big( [ N \ti{m}^{(N)}]_{ext1}^{(N+2)} [ Q^{(3)} [3 \chi_{R3}^{(3)} ]_{sym}^{(3)} ]_{ext2}^{(N+2)}, \ha{v}(t) \big) (t,k) 
 \notag \\
 & \hspace{0.3cm} +  \La_{\vp}^{(N+4)} \big( [ N \ti{m}^{(N)}]_{ext1}^{(N+4)} [ -Q_1^{(5)}]_{ext2}^{(N+4)}, \ha{v}(t) \big) (t,k) \nonumber \\
 = &  \La_{\vp}^{(N+2)} 
 \big( \big[  [ N \ti{m}^{(N)}]_{ext1}^{(N+2)} [ -Q^{(3)} \chi_{NR1}^{(3)} ]_{ext2}^{(N+2)} \big]_{sym}^{(N+2)}, \ha{v}(t) \big) (t,k) \nonumber \\
& \hspace{0.3cm} 
+   \La_{\vp}^{(N+2)} 
 \big( \big[  [ N \ti{m}^{(N)}]_{ext1}^{(N+2)} [ Q^{(3)} [3 \chi_{R3}^{(3)}]_{sym}^{(3)} ]_{ext2}^{(N+2)} \big]_{sym}^{(N+2)}, \ha{v}(t) \big) (t,k) 
 \notag \\
 & \hspace{0.3cm} 
 + \La_{\vp}^{(N+4)} \big( \big[ [ N \ti{m}^{(N)}]_{ext1}^{(N+4)} [ -Q_1^{(5)}]_{ext2}^{(N+4)} \big]_{sym}^{(N+4)} , \ha{v}(t) \big) (t,k).
\end{align*}
Therefore, we obtain \eqref{NF11}.
\end{proof}

%%%%%%%%%%%%%%%%%%%%%%%%%%%%%%%%%%%%%%%%%%%%%%%%%%
%%%%%%%%%%%%%%%%%%%%%%%%%%%%%%%%%%%%%%%%%%%%%%%%%%
\begin{lem}\label{lem_sym}
Let a $5$-multiplier $m^{(5)}$ %be symmetric $(k_1, k_2)$ and $(k_3, k_4)$ and 
satisfy 
$$m^{(5)} (k_1, k_2, k_3, k_4, k_5)=m^{(5)} (k_3, k_4, k_1, k_2, k_5). $$
Then, it follows that
\begin{equation} \label{sym11}
9 \Big[ \Big[ \frac{Q_1^{(3)}}{\Phi_{0}^{(3)} } \chi_{NR1}^{(3)} \Big]_{ext1}^{(5)}
[ Q_1^{(3)} \chi_{NR1}^{(3)} ]_{ext2}^{(5)} \, \chi_{H1}^{(5)} m^{(5)} \Big]_{sym}^{(5)}
=\big[q_2^{(5)} \chi_{H1}^{(5)} (1-\chi_{R2}^{(5)}) m^{(5)} ]_{sym}^{(5)}. 
\end{equation}
\end{lem}

%%%%%%%%%%%%%%%%%%%%%%%%%%%%%%%%%%
\begin{proof}
By Remark~\ref{rem_sym}, 
\begin{equation*}
9  \big[ \frac{Q_1^{(3)}}{\Phi_{0}^{(3)} } \chi_{NR1}^{(3)} \big]_{ext1}^{(5)}
[ Q_1^{(3)} \chi_{NR1}^{(3)} ]_{ext2}^{(5)} \, \chi_{H1}^{(5)} m^{(5)} 
=9  \big[ \frac{Q_1^{(3)}}{\Phi_{0}^{(3)} } \big]_{ext1}^{(5)}
[ Q_1^{(3)} ]_{ext2}^{(5)} \, \chi_{H1}^{(5)} \chi_{NR1}^{(5)} m^{(5)}.
\end{equation*}
We put $M:=\chi_{H1}^{(5)} \chi_{NR1}^{(5)}  m^{(5)}$. 
For $(k_1, k_2, k_3, k_4, k_5) \in \supp \chi_{H1}^{(5)}$, it  follows that $\chi_{NR1}^{(5)}= 1 -\chi_{R2}^{(5)}$, 
which leads that $M=\chi_{H1}^{(5)} (1 -\chi_{R2}^{(5)}) m^{(5)} $ and 
\begin{equation*}
M(k_1, k_2, k_3, k_4, k_5)=M(k_3, k_4, k_1, k_2, k_5). 
\end{equation*}
Therefore, the left hand side of \eqref{sym11} is equal to  
\begin{align*}
& 9 \Big[ \Big[ \frac{Q_1^{(3)}}{ \Phi_0^{(3)} }  \Big]_{ext1}^{(3)} [Q_1^{(3)}]_{ext2}^{(3)} \, M  \Big]_{sym}^{(5)} \\
& = \frac{9}{2} \Big[\Big\{ \frac{Q_1^{(3)}}{ \Phi_0^{(3)} } (k_1, k_2, k_{3,4,5}) Q_1^{(3)}(k_3, k_4, k_5 ) 
+ \frac{Q_1^{(3)}}{ \Phi_0^{(3)} } (k_3, k_4, k_{1,2,5}) Q_1^{(3)} (k_1, k_2, k_5) \Big\} \, M  \Big]_{sym}^{(5)} \\
&=[ q_2^{(5)} M ]_{sym}^{(5)}=[ q_2^{(5)} \chi_{H1}^{(5)} (1-\chi_{R2}^{(5)}) m^{(5)} ]_{sym}^{(5)}. 
\end{align*}
\end{proof}

%%%%%%%%%%%%%%%%%%%%%%%%%%%%%%%%%%%%%%%%%%%%%%%%%%%
%%%%%%%%%%%%%%%%%%%%%%%%%%%%%%%%%%%%%%%%%%%%%%%%%%%
\begin{lem}\label{L31}
\begin{align} 
&\partial_t \Lambda_{\vp}^{(3)} (  \ti{L}_{1, \vp}^{(3)} \chi_{>L}^{(3)}, \ha{v} (t) ) (t,k)
=\Lambda_{\vp}^{(3)} (  -\ti{L}_{1, \vp}^{(3)} \Phi_{\vp}^{(3)}  \chi_{>L}^{(3)}, \ha{v} (t) ) (t,k) \notag \\
&+ \Lambda_{\vp}^{(5)} \big( \sum_{i=2}^{8} \ti{L}_{i ,\vp}^{(5)} \Phi_{\vp}^{(5)}  \chi_{ > L}^{(5)} + 
\sum_{i=2}^{8} \ti{L}_{i, \vp}^{(5)} \Phi_{\vp}^{(5)} \chi_{ \le L}^{(5)} +
\sum_{i=8}^{23} \ti{M}_{i, \vp}^{(5)}, \ha{v} (t) \big) (t,k) \notag \\ 
&+ \La_{\vp}^{(5)} \big( \big[ [ 3 \ti{L}_{1, \vp}^{(3)} \chi_{>L}^{(3)} ]_{ext1}^{(5)} 
[Q^{(3)} [3 \chi_{R3}^{(3)}]_{sym}^{(3)} ]_{ext2}^{(5)} \big]_{sym}^{(5)}, \ha{v} (t) \big) (t, k) \notag \\
&  + \Lambda_{\vp}^{(7)} 
\big( \big[ [ 3 \ti{L}_{1, \vp}^{(3)}  \chi_{>L}^{(3)} ]_{est1}^{(7)} [-Q_1^{(5)}]_{ext2}^{(7)} \big]_{sym}^{(7)}, \ha{v} (t) \big) (t,k). 
\label{eq36}
\end{align}
\end{lem}
%%%%%%%%%%%%%%%%%%%%%%%%%%%%%%%%%%%%%%%%%%%
\begin{proof}
By Proposition \ref{prop_NF11} with $N=3$ and $\ti{m}^{(N)}=\ti{L}_{1, \vp}^{(3)} \chi_{>L}^{(3)}$, we have
\begin{align*} 
&\partial_t \Lambda_{\vp}^{(3)} (  \ti{L}_{1, \vp}^{(3)} \chi_{>L}^{(3)}, \ha{v} (t) ) (t,k)
=\Lambda_{\vp}^{(3)} (  -\ti{L}_{1, \vp}^{(3)} \Phi_{\vp}^{(3)}  \chi_{>L}^{(3)}, \ha{v} (t) ) (t,k)\\
&+ \Lambda_{\vp}^{(5)} \big( \big[ [3 \ti{L}_{1, \vp}^{(3)} \chi_{>L}^{(3)}  ]_{ext1}^{(5)} [ -Q^{(3)} \chi_{NR1}^{(3)} ]_{ext2}^{(5)} \big]_{sym}^{(5)}, \ha{v} (t) \big) (t,k) \nonumber \\ 
&+ \Lambda_{\vp}^{(5)} \big( \big[ [3 \ti{L}_{1, \vp}^{(3)} \chi_{>L}^{(3)}  ]_{ext1}^{(5)} [ Q^{(3)} [3 \chi_{R3}^{(3)}]_{sym}^{(3)} ]_{ext2}^{(5)} \big]_{sym}^{(5)}, \ha{v} (t) \big) (t,k) \nonumber \\ 
& + \Lambda_{\vp}^{(7)} 
\big( \big[ [ 3 \ti{L}_{1, \vp}^{(3)}  \chi_{>L}^{(3)} ]_{est1}^{(7)} [-Q_1^{(5)}]_{ext2}^{(7)} \big]_{sym}^{(7)}, \ha{v} (t) \big) (t,k).
\end{align*}
Thus, we only need to show
\begin{align} \label{eq38}
\big[ [3 \ti{L}_{1, \vp}^{(3)} \chi_{>L}^{(3)}  ]_{ext1}^{(5)} [ -Q_1^{(3)} \chi_{NR1}^{(3)} ]_{ext2}^{(5)} \big]_{sym}^{(5)} 
= \sum_{i=2}^{8} \ti{L}_{i, \vp}^{(5)} \Phi_{\vp}^{(5)}+  \sum_{i=8}^{23} \ti{M}_{i, \vp}^{(5)}. % \label{eq38}
\end{align}
By $[3 \chi_{H1}^{(3)} ]_{sym}^{(3)} + [3 \chi_{H2,1}^{(3)} ]_{sym}^{(3)}+ [3 \chi_{H2,2}^{(3)}]_{sym}^{(3)}+ \chi_{H3}^{(3)} =1$, 
the left hand side of \eqref{eq38} is equal to 
\begin{equation*}
\ti{M}_{8, \vp}^{(5)}+ \big[ [3 \ti{L}_{1, \vp}^{(3)} \chi_{>L}^{(3)}  ]_{ext1}^{(5)} 
[ -Q^{(3)} \chi_{NR1}^{(3)} ([3 \chi_{H1}^{(3)}]_{sym}^{(3)} + [3  \chi_{H2,1}^{(3)}]_{sym}^{(3)} ) ]_{ext2}^{(5)} \big]_{sym}^{(5)}. 
\end{equation*}
By \eqref{le211} with $m_1^{(3)} =Q^{(3)} \chi_{NR1}^{(3)} \chi_{>L}^{(3)} / \Phi_{\vp}^{(3)}$ and $m_2^{(3)} = Q^{(3)} \chi_{NR1}^{(3)}$, we have
\begin{align}
& \big[  [3 \ti{L}_{1,\vp}^{(3)} \chi_{>L}^{(3)} ]_{ext1}^{(3)} [-Q^{(3)} \chi_{NR1}^{(3)} ([3 \chi_{H1}^{(3)}]_{sym}^{(3)} +[3 \chi_{H2,1}^{(3)} ]_{sym}^{(3)}) ]_{ext2}^{(5)} \big]_{sym}^{(5)}\notag\\
& = 3 \Big[ \Big[ \frac{Q^{(3)} }{ \Phi_{\vp}^{(3)}} \chi_{NR1}^{(3)} \chi_{>L}^{(3)} [ 3 \chi_{H1}^{(3)} ]_{sym}^{(3)}  \Big]_{ext1}^{(5)} 
[Q^{(3)} \chi_{NR1}^{(3)}  ([3 \chi_{H1}^{(3)}]_{sym}^{(3)}+ [ 3\chi_{H2,1}^{(3)} ]_{sym}^{(3)} )   ]_{ext2}^{(5)} \Big]_{sym}^{(5)}  \nonumber \\
& =9 \Big[ \Big[ \frac{Q^{(3)} }{\Phi_{\vp}^{(3)} } \chi_{NR1}^{(3)} \chi_{>L}^{(3)}  \Big]_{ext1}^{(5)} [Q^{(3)} \chi_{NR1}^{(3)}  ]_{ext2}^{(5)} \, \chi_{NR(1,1)}^{(5)}  \Big]_{sym}^{(5)} \notag \\
&+9 \Big[ \Big[ \frac{Q^{(3)} }{\Phi_{\vp}^{(3)} } \chi_{NR1}^{(3)} \chi_{>L}^{(3)}  \Big]_{ext1}^{(5)} [Q^{(3)} \chi_{NR1}^{(3)}  ]_{ext2}^{(5)} \, \chi_{NR(2,1)}^{(5)}  \Big]_{sym}^{(5)} \label{eqn312}\\
& + \ti{M}_{21, \vp}^{(5)}+ \ti{M}_{23, \vp}^{(5)}. \notag 
\end{align}
Since
$\chi_{NR(2,1)}^{(5)}=\chi_{NR(2,1)}^{(5)} \chi_{A3}^{(5)}+ \chi_{NR(2,1)}^{(5)} (1-\chi_{A3}^{(5)})$, \eqref{eqn312} is equal to
\[
\ti{L}_{8, \vp}^{(5)} \Phi_{\vp}^{(5)}+  \ti{M}_{22, \vp}^{(5)}. 
\]
Therefore, we only need to show
\begin{align} \label{eq39}
9 \Big[ \Big[ \frac{Q^{(3)} }{\Phi_{\vp}^{(3)} } \chi_{NR1}^{(3)} \chi_{>L}^{(3)}  \Big]_{ext1}^{(5)} [Q^{(3)} \chi_{NR1}^{(3)}  ]_{ext2}^{(5)} \, \chi_{NR(1,1)}^{(5)}  \Big]_{sym}^{(5)} 
=\sum_{i=2}^{7} \ti{L}_{i, \vp}^{(5)} \Phi_{\vp}^{(5)}+\sum_{i=9}^{20} \ti{M}_{i, \vp}^{(5)}.
\end{align}
By the definition,
\begin{align} \label{eq33}
 &9 \Big[ \Big[ \frac{Q^{(3)} }{\Phi_{\vp}^{(3)} } \chi_{NR1}^{(3)} \chi_{>L}^{(3)}  \Big]_{ext1}^{(5)} [Q^{(3)} \chi_{NR1}^{(3)}  ]_{ext2}^{(5)} \, \chi_{NR(1,1)}^{(5)}  \Big]_{sym}^{(5)} \notag \\
& = \sum_{i=19}^{20} \ti{M}_{i, \vp}^{(5)} 
+ 9 \Big[ \Big[ \frac{Q^{(3)} }{\Phi_{0}^{(3)} } \chi_{NR1}^{(3)}   \Big]_{ext1}^{(5)} [Q^{(3)} \chi_{NR1}^{(3)}  ]_{ext2}^{(5)} \, 
\chi_{NR(1,1)}^{(5)}  \Big]_{sym}^{(5)}.
\end{align}
Since $\supp \chi_{H1}^{(5)} \subset \supp \chi_{NR(1,1)}^{(5)}$, 
\begin{align*}
\chi_{NR(1,1)}^{(5)}& = \chi_{NR(1,1)}^{(5)}(1- \chi_{H1}^{(5)})+ \chi_{H1}^{(5)} \\
& = \chi_{NR(1,1)}^{(5)} (1- \chi_{H1}^{(5)}) \chi_{A1}^{(5)} + \chi_{NR(1,1)}^{(5)} (1-\chi_{H1}^{(5)}) (1- \chi_{A1}^{(5)})+ \chi_{H1}^{(5)},
\end{align*}
which leads that 
\begin{align} \label{eqn321}
& 9 \Big[ \Big[ \frac{Q^{(3)} }{\Phi_{0}^{(3)} } \chi_{NR1}^{(3)}   \Big]_{ext1}^{(5)} [Q^{(3)} \chi_{NR1}^{(3)}  ]_{ext2}^{(5)} \, 
\chi_{NR(1,1)}^{(5)}  \Big]_{sym}^{(5)} \notag \\
& = \ti{L}_{7, \vp}^{(5)} \Phi_{\vp}^{(5)} + \ti{M}_{18, \vp}^{(5)} 
+  9 \Big[ \Big[ \frac{Q^{(3)} }{\Phi_{0}^{(3)} } \chi_{NR1}^{(3)}   \Big]_{ext1}^{(5)} [Q^{(3)} \chi_{NR1}^{(3)}  ]_{ext2}^{(5)} \, 
\chi_{H1}^{(5)} \Big]_{sym}^{(5)}.
\end{align}
By $Q^{(3)}=Q_1^{(3)} +Q_2^{(3)}+ Q_3^{(3)}$, 
\begin{align} \label{eqn320}
& 9 \Big[ \Big[ \frac{Q^{(3)} }{\Phi_{0}^{(3)} } \chi_{NR1}^{(3)}   \Big]_{ext1}^{(5)} [Q^{(3)} \chi_{NR1}^{(3)}  ]_{ext2}^{(5)} \, 
\chi_{H1}^{(5)} \Big]_{sym}^{(5)} \notag \\
& = \sum_{i=3}^6 \ti{L}_{i, \vp}^{(5)} \Phi_{\vp}^{(5)}+ \sum_{i=11}^{17} \ti{M}_{i, \vp}^{(5)}
+ 9 \Big[ \Big[ \frac{Q_1^{(3)} }{\Phi_{0}^{(3)} } \chi_{NR1}^{(3)}   \Big]_{ext1}^{(5)} [Q_1^{(3)} \chi_{NR1}^{(3)}  ]_{ext2}^{(5)} \, 
\chi_{H1}^{(5)} \Big]_{sym}^{(5)}.
\end{align}
Since $\chi_{H1}^{(5)}=  \chi_{H1}^{(5)} \chi_{R1}^{(5)}+ \chi_{H1}^{(5)} (1- \chi_{R1}^{(5)}) \chi_{R4}^{(5)}
+ \chi_{H1}^{(5)} (1- \chi_{R1}^{(5)} )(1- \chi_{R4}^{(5)})$, 
by $\chi_{NR2}^{(5)}= (1-\chi_{R1}^{(5)}) (1- \chi_{R2}^{(5)})$ and 
Lemma~\ref{lem_sym} with $m^{(5)}= \chi_{R1}^{(5)}, (1- \chi_{R1}^{(5)}) \chi_{R4}^{(5)}$ or $(1-\chi_{R1}^{(5)}) (1-\chi_{R4}^{(5)})$, 
it follows that 
\begin{equation} \label{eqn322}
 9 \Big[ \Big[ \frac{Q_1^{(3)} }{\Phi_{0}^{(3)} } \chi_{NR1}^{(3)}   \Big]_{ext1}^{(5)} [Q_1^{(3)} \chi_{NR1}^{(3)}  ]_{ext2}^{(5)} \, 
\chi_{H1}^{(5)}  \Big]_{sym}^{(5)}
= \ti{L}_{2, \vp}^{(5)} \Phi_{\vp}^{(5)}+ \sum_{i=9}^{10} \ti{M}_{i, \vp}^{(5)}.
\end{equation}
Collecting \eqref{eq33}--\eqref{eqn322}, we obtain \eqref{eq39}. 
\end{proof}

%%%%%%%%%%%%%%%%%%%%%%%%%%%%%%%%
\begin{lem}\label{L32}
\begin{align} \label{eq46}
&\partial_t \Lambda_{\vp}^{(5)} (\ti{L}_{2, \vp}^{(5)} \chi_{>L}^{(5)}, \ha{v} (t) ) (t,k)
=\Lambda_{\vp}^{(5)} (-\ti{L}_{2, \vp}^{(5)} \Phi_{\vp}^{(5)} \chi_{>L}^{(5)} , \ha{v} (t) ) (t,k) \nonumber \\
&+  \Lambda_{\vp}^{(7)} \Big( \sum_{i=1}^2 \ti{L}_{i, \vp}^{(7)} \Phi_{\vp}^{(7)} \chi_{>L}^{(7)
}+ \sum_{i=1}^2 \ti{L}_{i, \vp}^{(7)} \Phi_{\vp}^{(7)} \chi_{\le L}^{(7)} 
+ \sum_{i=5}^{15} \ti{M}_{i, \vp}^{(7)}, \ha{v} (t) \Big) (t,k) \nonumber \\
& + \La_{\vp}^{(7)} \big( \big[ [  5 \ti{L}_{2, \vp}^{(5)} \chi_{>L}^{(5)} ]_{ext1}^{(7)} [ Q^{(3)} [ 3 \chi_{R3}^{(3)} ]_{sym}^{(3)} ]_{ext2}^{(7)}  \big]_{sym}^{(7)}, \ha{v}(t) \big) (t, k) \notag \\
& +\Lambda_{\vp}^{(9)} \big( \big[ [ 5 \ti{L}_{2, \vp}^{(5)} \chi_{>L}^{(5)}]_{ext1}^{(9)} [-Q_1^{(5)}]_{ext2}^{(9)} \big]_{sym}^{(9)}, \ha{v} (t) \big) (t,k).
\end{align}
\end{lem}
%%%%%%%%%%%%%%%%%%%%%%%%%%%%%%%%%%
\begin{proof}
By Proposition~\ref{prop_NF11} with $N=5$ and $\ti{m}^{(N)}=\ti{L}_{2, \vp}^{(5)} \chi_{>L}^{(5)}$, we have
\begin{align*}
&\partial_t \Lambda_{\vp}^{(5)} (\ti{L}_{2, \vp}^{(5)} \chi_{>L}^{(5)}, \ha{v} (t) ) (t,k)
=\Lambda_{\vp}^{(5)} (-\ti{L}_{2, \vp}^{(5)} \Phi_{\vp}^{(5)} \chi_{>L}^{(5)} , \ha{v} (t) ) (t,k)\\
&+  \Lambda_{\vp}^{(7)} \big( \big[ [5 \ti{L}_{2, \vp}^{(5)} \chi_{>L}^{(5)} ]_{ext1}^{(7)} [-Q^{(3)} \chi_{NR1}^{(3)} ]_{ext2}^{(7)} 
 \big]_{sym}^{(7)}, \ha{v}(t) \big) (t,k)\\
 & +\Lambda_{\vp}^{(7)} (\big[ [ 5 \ti{L}_{2, \vp}^{(5)} \chi_{>L}^{(5)}]_{ext1}^{(7)} [Q^{(3)} [3 \chi_{R3}^{(3)} ]_{sym}^{(3)} ]_{ext2}^{(7)} \big]_{sym}^{(7)}, \ha{v} (t) ) (t,k) \\
& +\Lambda_{\vp}^{(9)} (\big[ [ 5 \ti{L}_{2, \vp}^{(5)} \chi_{>L}^{(5)}]_{ext1}^{(9)} [-Q_1^{(5)}]_{ext2}^{(9)} \big]_{sym}^{(9)}, \ha{v} (t) ) (t,k).
\end{align*}
Thus, we only need to show
\begin{equation}\label{eq49}
 \big[ [5 \ti{L}_{2, \vp}^{(5)} \chi_{>L}^{(5)} ]_{ext1}^{(7)} [-Q^{(3)} \chi_{NR1}^{(3)} ]_{ext2}^{(7)}  \big]_{sym}^{(7)}
= \sum_{i=1}^2 \ti{L}_{i, \vp}^{(7)} \Phi_{\vp}^{(7)}  + \sum_{i=5}^{15} \ti{M}_{i, \vp}^{(7)}.
\end{equation}
By $[3 \chi_{H1}^{(3)}]_{sym}^{(3)}+ [ 3 \chi_{H2,1}^{(3)} ]_{sym}^{(3)} + [3 \chi_{H2,2}^{(3)}]_{sym}^{(3)} + \chi_{H3}^{(3)}=1 $, 
the left hand side of \eqref{eq49} is equal to 
\begin{equation*}
\ti{M}_{5, \vp}^{(7)}+ \big[ [5 \ti{L}_{2, \vp}^{(5)} \chi_{>L}^{(5)} ]_{ext1}^{(7)} 
[-Q^{(3)} \chi_{NR1}^{(3)} ( [3\chi_{H1}^{(3)} ]_{sym}^{(3)}+ [3 \chi_{H2,1}^{(3)} ]_{sym}^{(3)} ) ]_{ext2}^{(7)}  \big]_{sym}^{(7)}.
\end{equation*}
Note that 
\begin{equation*}
\ti{L}_{2, \vp}^{(5)} \Phi_{\vp}^{(5)} = [ q_2^{(5)} \chi_{NR2}^{(5)} (1- \chi_{R4}^{(5)}) \chi_{H1}^{(5)}  ]_{sym}^{(5)}
=[ Q_2^{(5)} \chi_{H1}^{(5)} ]_{sym}^{(5)}  
\end{equation*}
and $\ti{L}_{2, \vp}^{(5)} \chi_{>L}^{(5)}= [Q_2^{(5)} \chi_{H1}^{(5)} ]_{sym}^{(5)} \, \chi_{>L}^{(5)} / \Phi_{\vp}^{(5)} $. 
By \eqref{le212} with $m_1^{(5)}=\chi_{>L}^{(5)} /  \Phi_{\vp}^{(5)} $, $m_2^{(3)}= Q^{(3)} \chi_{NR1}^{(3)}$ and 
$m_3^{(5)} =Q_2^{(5)}$, it follows that 
\begin{align*}
& \big[ [5 \ti{L}_{2, \vp}^{(5)} \chi_{>L}^{(5)} ]_{ext1}^{(7)} 
[-Q^{(3)} \chi_{NR1}^{(3)} ( [3\chi_{H1}^{(3)} ]_{sym}^{(3)}+ [3 \chi_{H2,1}^{(3)} ]_{sym}^{(3)} ) ]_{ext2}^{(7)}  \big]_{sym}^{(7)} \notag \\
&= \big[ \big[ [ 5 Q_2^{(5)} \chi_{H1}^{(5)}]_{sym}^{(5)} \chi_{>L}^{(5)}/ \Phi_{\vp}^{(5)} \big]_{ext1}^{(7)} 
[-Q^{(3)} \chi_{NR1}^{(3)} ( [3\chi_{H1}^{(3)} ]_{sym}^{(3)}+ [3 \chi_{H2,1}^{(3)} ]_{sym}^{(3)} ) ]_{ext2}^{(7)}  \big]_{sym}^{(7)} \notag \\
&= (-3) \Big[ \Big[ \frac{Q_2^{(5)}}{ \Phi_{\vp}^{(5)} } \chi_{>L}^{(5)} \Big]_{ext1}^{(7)} [ Q^{(3)} \chi_{NR1}^{(3)} ]_{ext2}^{(7)} 
\, \chi_{NR(1,1)}^{(7)} \Big]_{sym}^{(7)} 
+ \sum_{i=13}^{15} \ti{M}_{i, \vp}^{(5)}.
\end{align*}
Therefore, we only need to show 
\begin{equation} \label{eq42}
(-3) \Big[ \Big[ \frac{Q_2^{(5)}}{ \Phi_{\vp}^{(5)} } \chi_{>L}^{(5)} \Big]_{ext1}^{(7)} [ Q^{(3)} \chi_{NR1}^{(3)} ]_{ext2}^{(7)} 
\, \chi_{NR(1,1)}^{(7)} \Big]_{sym}^{(7)} 
= \sum_{i=1}^2 \ti{L}_{i, \vp}^{(7)} \Phi_{\vp}^{(7)}+ \sum_{i=6}^{12} \ti{M}_{i, \vp}^{(7)}. 
\end{equation}
By the definition, the left hand side of (\ref{eq42}) is equal to 
\begin{align} \label{eqn422}
\sum_{i=11}^{12} \ti{M}_{i, \vp}^{(7)}
+(-3) \Big[ \Big[ \frac{Q_2^{(5)}}{ \Phi_{0}^{(5)} } \Big]_{ext1}^{(7)} [ Q^{(3)} \chi_{NR1}^{(3)} ]_{ext2}^{(7)} 
\, \chi_{NR(1,1)}^{(7)} \Big]_{sym}^{(7)}. 
\end{align}
Since $\supp \, \chi_{H1}^{(7)} \subset \supp \, \chi_{NR(1,1)}^{(7)}$, 
it follows that $\chi_{NR(1,1)}^{(7)}=\chi_{NR(1,1)}^{(7)} (1-\chi_{H1}^{(7)})+\chi_{H1}^{(7)}$, which leads that 
\begin{align} \label{eqn423}
& (-3) \Big[ \Big[ \frac{Q_2^{(5)}}{ \Phi_{0}^{(5)} } \Big]_{ext1}^{(7)} [ Q^{(3)} \chi_{NR1}^{(3)} ]_{ext2}^{(7)} 
\, \chi_{NR(1,1)}^{(7)} \Big]_{sym}^{(7)} \notag \\
& = \ti{M}_{10, \vp}^{(7)}+
(-3) \Big[ \Big[ \frac{Q_2^{(5)}}{ \Phi_{0}^{(5)} } \Big]_{ext1}^{(7)} [ Q^{(3)} \chi_{NR1}^{(3)} ]_{ext2}^{(7)} 
\, \chi_{H1}^{(7)} \Big]_{sym}^{(7)}.
\end{align}
By $Q^{(3)}=Q_1^{(3)} + Q_2^{(3)}+ Q_3^{(3)}$, 
\begin{align}
& (-3) \Big[ \Big[ \frac{Q_2^{(5)}}{ \Phi_{0}^{(5)} } \Big]_{ext1}^{(7)} [ Q^{(3)} \chi_{NR1}^{(3)} ]_{ext2}^{(7)} 
\, \chi_{H1}^{(7)} \Big]_{sym}^{(7)} \notag \\
& = (-3) \Big[ \Big[ \frac{Q_2^{(5)}}{ \Phi_{0}^{(5)} } \Big]_{ext1}^{(7)} [ Q_1^{(3)} \chi_{NR1}^{(3)} ]_{ext2}^{(7)} 
\, \chi_{H1}^{(7)} \Big]_{sym}^{(7)} \label{eqn42a} \\
& +(-3) \Big[ \Big[ \frac{Q_2^{(5)}}{ \Phi_{0}^{(5)} } \Big]_{ext1}^{(7)} [ (Q_2^{(3)} +Q_3^{(3)}) \chi_{NR1}^{(3)} ]_{ext2}^{(7)} 
\, \chi_{H1}^{(7)} \Big]_{sym}^{(7)}. \label{eqn42b}
\end{align}
Since
\begin{align*}
\chi_{H1}^{(7)}  = & \chi_{H1}^{(7)} \chi_{R1}^{(7)} \chi_{A4}^{(7)}+ \chi_{H1}^{(7)}  \chi_{R1}^{(7)} (1-\chi_{A4}^{(7)}) \\
& + \chi_{H1}^{(7)} (1-\chi_{R1}^{(7)}) \chi_{R5}^{(7)}+ \chi_{H1}^{(5)} (1-\chi_{R1}^{(7)}) (1-\chi_{R5}^{(7)}),
\end{align*}
\eqref{eqn42a} is equal to 
\begin{equation} \label{eqn424}
\ti{L}_{1, \vp}^{(7)} \Phi_{\vp}^{(7)}+ \sum_{i=6}^8 \ti{M}_{i, \vp}^{(7)}. 
\end{equation}
Since $\chi_{H1}^{(7)}= \chi_{H1}^{(7)} \chi_{A1}^{(7)}+ \chi_{H1}^{(7)} (1-\chi_{A1}^{(7)})$, 
\eqref{eqn42b} is equal to 
\begin{equation} \label{eqn425}
\tilde{L}_{2, \vp}^{(7)} \Phi_{\vp}^{(7)}+ \ti{M}_{9, \vp}^{(7)}.
\end{equation}
Collecting (\ref{eqn422})--(\ref{eqn425}), we obtain (\ref{eq42}). 
\end{proof}

%%%%%%%%%%%%%%%%%%%%%%%%%%%%%%%%%%%%%%%%%%%%%%%%%%%%%%%%%%%%%%%%%%%%%%%%%
%%%%%%%%%%%%%%%%%%%%%%%%%%%%%%%%%%%%%%%%%%%%%%%%%%%%%%%%%%%%%%%%%%%%%%%%%
%%%%%%%%%%%%%%%%%%%%%%%%%%%%%%%%%%%%%%%%%%%%%%%%%%%%%%%%%%%%%%%%%%%%%%%%%
Now, we prove Proposition \ref{prop_NF2}.
\begin{proof}[Proof of Proposition \ref{prop_NF2}]
By direct computation, it follows that
\begin{align*}
& -\frac{4}{5} \ga^2 q_1^{(5)} [ \chi_{R1}^{(5)} (1- \chi_{R2}^{(5)}) ]_{sym}^{(5)}= \ti{M}_{1, \vp}^{(5)}+ \ti{M}_{2, \vp}^{(5)}, \\
& -6 \de q_1^{(5)} (1- [5 \chi_{R1}^{(5)}]_{sym}^{(5)})
= \ti{L}_{1, \vp}^{(5)} \Phi_{\vp}^{(5)} \chi_{>L}^{(5)}+ \ti{L}_{1, \vp}^{(5)} \Phi_{\vp}^{(5)} \chi_{\le L}^{(5)} +\sum_{i=3}^5 \ti{M}_{i, \vp}^{(5)}
\end{align*}
By $[3 \chi_{H1}^{(3)} ]_{sym}^{(3)} + [3  \chi_{H2,1}^{(3)}]_{sym}^{(3)}+ [3 \chi_{H2,2}^{(3)}]_{sym}^{(3)} +\chi_{H3}^{(3)} = 1$, we have
\begin{align*}
-Q^{(3)} \chi_{NR1}^{(3)}
= \sum_{i=1}^4 \ti{L}_{i, \vp}^{(3)} \Phi_{\vp}^{(3)} \chi_{>L}^{(3)}+ \sum_{i=1}^4 \ti{L}_{i, \vp}^{(3)} \Phi_{\vp}^{(3)} \chi_{\le L}^{(3)}. 
\end{align*}
Thus, by Proposition \ref{prop_req1}, we have
\begin{equation}\label{eq370}
\begin{split}
\p_t \ha{v} (t,k) = & \La_{\vp}^{(3)} \Big(\sum_{i=1}^4 \ti{L}_{i, \vp}^{(3)} \Phi_{\vp}^{(3)} \chi_{> L}^{(3)}
+ \sum_{i=1}^4 \ti{L}_{i, \vp}^{(3)} \Phi_{\vp}^{(3)} \chi_{\le L}^{(3)}+ \ti{M}_{1, \vp}^{(3)} , \ha{v}(t) \Big)(t,k)\\
& + \La_{\vp}^{(5)} \Big( \ti{L}_{1, \vp}^{(5)} \Phi_{\vp}^{(5)} \chi_{> L}^{(5)} 
+ \ti{L}_{1,\vp}^{(5)} \Phi_{\vp}^{(5)} \chi_{\le L}^{(5)}+ \sum_{i=1}^5 \ti{M}_{i, \vp}^{(5)}, \ha{v} (t) \Big)(t,k). 
\end{split}
\end{equation}
By Proposition~\ref{prop_NF11} with $N=3$ and $\ti{m}^{(N)}=\sum_{i=2}^4 \ti{L}_{i, \vp}^{(3)} \chi_{>L}^{(3)}$, we have
\begin{equation} \label{eq37}
\begin{split}
&\partial_t \Lambda_{\vp}^{(3)} \Big(\sum_{i=2}^4 \ti{L}_{i, \vp}^{(3)} \chi_{>L}^{(3)}, \ha{v} (t) \Big) (t,k) \\
& =\Lambda_{\vp}^{(3)} \Big(- \sum_{i=2}^4 \ti{L}_{i, \vp}^{(3)} \Phi_{\vp}^{(3)} \chi_{>L}^{(3)} , \ha{v}(t) \Big) (t,k)
+ \Lambda_{\vp}^{(5)} \big( \ti{M}_{7, \vp}^{(5)}, \ha{v}(t) \big) (t,k) \\
& + \Lambda_{\vp}^{(5)} \Big(\big[ [ 3 \sum_{i=2}^4 \ti{L}_{i, \vp}^{(3)} \chi_{>L}^{(3)}]_{ext1}^{(5)} 
[Q^{(3)} [3 \chi_{R3}^{(3)}]_{sym}^{(3)}  ]_{ext2}^{(5)} \big]_{sym}^{(5)}, \ha{v} (t) \Big) (t,k) \\
& + \Lambda_{\vp}^{(7)} \Big(\big[ [ 3 \sum_{i=2}^4 \ti{L}_{i, \vp}^{(3)} \chi_{>L}^{(3)}]_{ext1}^{(7)} 
[-Q_1^{(5)}]_{ext2}^{(7)} \big]_{sym}^{(7)}, \ha{v} (t) \Big) (t,k).
\end{split}
\end{equation}
%%%%%%%%%%%%%%%%%%%%%%%%%%%%%%%%%%%%%%%%%%%%%%%%
By Proposition~\ref{prop_NF11} with $N=5$ and 
$\ti{m}^{(N)}= (\ti{L}_{1, \vp}^{(5)}+ \sum_{i=3}^8 \ti{L}_{i, \vp}^{(5)}) \chi_{>L}^{(5)}$, we have
\begin{equation} \label{eq31}
\begin{split}
&\partial_t \Lambda_{\vp}^{(5)}  \big( \big(\ti{L}_{1, \vp}^{(5)}+ \sum_{i=3}^8 \ti{L}_{i, \vp}^{(5)} \big) \chi_{>L}^{(5)}, \ha{v}(t) \big) (t,k) \\
& =\Lambda_{\vp}^{(5)}  \big(- \big( \ti{L}_{1, \vp}^{(5)}+ \sum_{i=3}^8 \ti{L}_{i, \vp}^{(5)} \big) \chi_{>L}^{(5)}, \ha{v}(t) \big) (t,k)
+  \Lambda_{\vp}^{(7)} (\ti{M}_{3, \vp}^{(7)}+ \ti{M}_{4, \vp}^{(5)} , \ha{v}(t) ) (t,k) \\
& + \Lambda_{\vp}^{(7)} 
\big( \big[ [ 5 \big( \ti{L}_{1, \vp}^{(5)}+ \sum_{i=3}^8 \ti{L}_{i, \vp}^{(5)} \big) \chi_{>L}^{(5)} ]_{ext1}^{(7)} 
[Q^{(3)} [3 \chi_{R3}^{(3)}]_{sym}^{(3)} ]_{ext2}^{(7)} \big]_{sym}^{(7)}, \ha{v} (t) \big)  (t,k) \\
& + \Lambda_{\vp}^{(9)} 
\big( \big[ [ 5\big( \ti{L}_{1, \vp}^{(5)}+ \sum_{i=3}^8 \ti{L}_{i, \vp}^{(5)} \big) \chi_{>L}^{(5)} ]_{ext1}^{(9)} 
[-Q_1^{(5)}]_{ext2}^{(9)} \big]_{sym}^{(9)}, \ha{v} (t) \big)  (t,k).
\end{split}
\end{equation}
%%%%%%%%%%%%%%%%%%%%%%%%%%%%%%%%%%%%%%
%%%%%%%%%%%%%%%%%%%%%%%%%%%%%%%%%%%%%
\iffalse

By Proposition~\ref{prop_NF11} with $N=5$ and 
$\ti{m}^{(N)}= (\ti{L}_{3, \vp}^{(5)}+ \ti{L}_{4, \vp}^{(5)} +\ti{L}_{5, \vp}^{(5)} +\ti{L}_{6, \vp}^{(5)} ) \chi_{>L}^{(5)}$, we have
\begin{equation} \label{eq371}
\begin{split}
&\partial_t \Lambda_{\vp}^{(5)}  \big( \big(\ti{L}_{3, \vp}^{(5)}+ \ti{L}_{4,\vp}^{(5)}+ \ti{L}_{5, \vp}^{(5)} + \ti{L}_{6, \vp}^{(5)} \big) \chi_{>L}^{(5)}, \ha{v}(t) \big) (t,k) \\
& =\Lambda_{\vp}^{(5)}  \big(- \big( \ti{L}_{3, \vp}^{(5)}+ \ti{L}_{4, \vp}^{(5)} +\ti{L}_{5, \vp}^{(5)} +\ti{L}_{6, \vp}^{(5)} \big) \Phi_{\vp}^{(5)} \chi_{>L}^{(5)}, \ha{v}(t) \big) (t,k)
+  \Lambda_{\vp}^{(7)} (\ti{M}_{4, \vp}^{(7)} , \ha{v}(t) ) (t,k) \\
& + \Lambda_{\vp}^{(7)} 
\big( \big[ [ 5 \big( \ti{L}_{3, \vp}^{(5)}+ \ti{L}_{4, \vp}^{(5)} +\ti{L}_{5, \vp}^{(5)} +\ti{L}_{6, \vp}^{(5)} \big) \chi_{>L}^{(5)} ]_{ext1}^{(7)} 
[Q^{(3)} [3 \chi_{R3}^{(3)}]_{sym}^{(3)} ]_{ext2}^{(7)} \big]_{sym}^{(7)}, \ha{v} (t) \big)  (t,k) \\
& + \Lambda_{\vp}^{(9)} 
\big( \big[ [ 5\big( \ti{L}_{3, \vp}^{(5)}+ \ti{L}_{4, \vp}^{(5)} +\ti{L}_{5, \vp}^{(5)} +\ti{L}_{6, \vp}^{(5)} \big) \chi_{>L}^{(5)} ]_{ext1}^{(9)} 
[-Q_1^{(5)}]_{ext2}^{(9)} \big]_{sym}^{(9)}, \ha{v} (t) \big)  (t,k).
\end{split}
\end{equation}

\fi
%%%%%%%%%%%%%%%%%%%%%%%%%%%%%%%%%%%%%%%%%%%
%%%%%%%%%%%%%%%%%%%%%%%%%%%%%%%%%%%%%%%%%%%
By Proposition~\ref{prop_NF11} with $N=7$ and $\ti{m}^{(N)}=(\ti{L}_{1, \vp}^{(7)}+\ti{L}_{2, \vp}^{(5)} ) \chi_{>L}^{(7)}$, we have
\begin{equation} \label{eq51}
\begin{split}
&\partial_t \Lambda_{\vp}^{(7)} ( (\ti{L}_{1, \vp}^{(7)}+\ti{L}_{2, \vp}^{(7)}) \chi_{>L}^{(7)}, \ha{v}(t) ) (t,k)
=\Lambda_{\vp}^{(7)} ( -(\ti{L}_{1, \vp}^{(7)}+ \ti{L}_{2, \vp}^{(7)} ) \Phi_{\vp}^{(7)} \chi_{>L}^{(7)} , \ha{v}(t) ) (t,k) \\
&+ \Lambda_{\vp}^{(9)} (\ti{M}_{2, \vp}^{(9)}+ \ti{M}_{3, \vp}^{(9)}, \ha{v}(t) ) (t,k)+
 \Lambda_{\vp}^{(11)} (\ti{M}_{1, \vp}^{(11)}, \ha{v}(t) ) (t,k).
\end{split}
\end{equation}
%%%%%%%%%%%%%%%%%%%%%%%%%%%%%%%%%%%%%%%%%%%%%%%%%
By \eqref{eq370}--\eqref{eq51}, Lemmas \ref{L31} and \ref{L32}, we conclude \eqref{NF21}.
%%%%%%%%%%%%%%%%%%%%%%%%%%%%%%%%%%%%%%%%%%%%%%%
\end{proof}

%%%%%%%%%%%%%%%%%%%%%%%%%%%%%%%%%%%%%%%%%%%%%%%%%%%
%%%%%%%%%%%%%%%%%%%%%%%%%%%%%%%%%%%%%%%%%%%%%%%%%%%
%%%%%%%%%%%%%%% cancellation property%%%%%%%%%%%%%%%%%%%
%%%%%%%%%%%%%%%%%%%%%%%%%%%%%%%%%%%%%%%%%%%%%%%%%%%
%%%%%%%%%%%%%%%%%%%%%%%%%%%%%%%%%%%%%%%%%%%%%%%%%%%
\section{cancellation properties}

In Lemma \ref{lem_pwb2} and Lemmas \ref{lem_mle}--\ref{lem_nl10}, we show that all multipliers $\{ M_{j, \vp}^{(N)} \}$ except 
for $M_{1,\vp}^{(5)}$, $M_{9,\vp}^{(5)}$, $M_{11,\vp}^{(5)}$, $M_{13, \vp}^{(5)}$ and $M_{6,\vp}^{(7)}$ have no derivative loss.
As we explain below, there are some difficulties to estimate 
$M_{1,\vp}^{(5)}$, $M_{9, \vp}^{(5)}$, $M_{11,\vp}^{(5)}$, $M_{13, \vp}^{(5)}$ and $M_{6,\vp}^{(7)}$
and the normal form reduction does not work to overcome the difficulties since these are resonant parts.
Therefore, we use a kinds of cancellation properties. \\
%%%%%%%%%%%%%%%%%%%%%%%%%%%%%%%%%%%%%%%%%%%
%%%%%%%%%%%%%%%%%%%%%%%%%%%%%%%%%%%%%%%%%%%
(i) $M_{1,\vp}^{(5)}$ has one derivative loss and $M_{9,\vp}^{(5)}$ has one derivative loss.
That is, when $(k_1,k_2,k_3, k_4, k_5)\in \supp M_{1, \vp}^{(5)}= \supp M_{9, \vp}^{(5)} $ and $|k_5|^{1/2}  \gg \max_{1 \le j \le 4} \{ |k_j| \}$, 
it follows that $8^3 \max_{1 \le j \le 4}\{ |k_j| \} <|k_5|$ and
\begin{equation*}
|M_{1,\vp}^{(5)}| \sim |k_5|, \hspace{0.5cm} |M_{9, \vp}^{(5)}| \sim |k_5|. 
\end{equation*}
In Proposition \ref{prop_res1}, we compute the sum of these two multipliers and show it has no derivative loss. \\
%%%%%%%%%%%%%%%%%%%%%%%%%%%%%%%%%%%%%%%%%%%%%
(ii) $M_{6, \vp}^{(7)}$ has one derivative loss.
That is, when $(k_1, k_2, k_3, k_4, k_5, k_6, k_7) \in \supp M_{6, \vp}^{(7)}$ and $|k_7|^{1/2} \gg \max_{1 \le j \le 6} \{ |k_j| \} $, 
it follows that $8^{5} \max_{1 \le j \le 6} \{ |k_j| \} < |k_7|$ and 
\begin{equation*}
|M_{6,\vp}^{(7)}| \sim  \Big| \frac{1}{k_{1,2} k_{3,4}} + \frac{1}{k_{1,3} k_{2,4}} + \frac{1}{ k_{1,4} k_{2,3} }   \Big| \, |k_7|. 
\end{equation*}
For instance, by taking $k_1=1$, $k_2=1$, $k_3=1$, $k_4=1$, $k_5=-2$, $k_6=-2$ and $|k_7|$ sufficiently large, it follows that 
$(k_1, k_2, k_3, k_4, k_5, k_6, k_7) \in \supp M_{6, \vp}^{(7)}$ and
\begin{equation*}
|M_{6, \vp}^{(7)}| \sim |k_7|. 
\end{equation*}
In Proposition \ref{prop_res2}, we compute the symmetrization of it and show it has no derivative loss. \\
%%%%%%%%%%%%%%%%%%%%%%%%%%%%%%%%%%%%%%%%%%%%%%%%%%%%%%%%%%%%%%
(iii) $M_{11,\vp}^{(5)}$ has one derivative loss and $M_{13, \vp}^{(5)}$ has one derivative loss. 
That is, when $(k_1, k_2, k_3, k_4, k_5) \in \supp M_{11, \vp}^{(5)}= \supp M_{13, \vp}^{(5)}$ and 
$ |k_5|^{1/2} \gg \max_{1 \le j \le 4} \{ |k_j| \}$, it follows that $8^3 \max_{1 \le j \le 4} \{ |k_j| \} < |k_5|$ and 
\begin{equation*} 
|M_{11, \vp}^{(5)}| \sim |k_5|, \hspace{0.5cm} |M_{13, \vp}^{(5)}| \sim |k_5|. 
\end{equation*}
In Proposition \ref{prop_res3}, we compute the sum of these two terms and show it has no derivative loss.

\begin{lem} \label{lem_key}
It follows that 
\begin{align} \label{sym21}
q_2^{(5)}=  
-\frac{2}{5} i \ga^2  \frac{1}{ k_{1,2} k_{3,4} k_{1,2,3,5} k_{1,2,4,5} k_{1,3,4,5} k_{2,3,4,5} } \sum_{l=1}^7 p_l(k_1, k_2, k_3, k_4) k_5^{7-l}
\end{align}
where each $p_l(k_1, k_2, k_3, k_4)$ is a polynomial of degree $l$ for $l=1, \dots, 7$ and 
\begin{align} \label{def_q2}
& p_1(k_1, k_2, k_3, k_4):=k_{1,2,3,4}, \hspace{0.5cm} p_2(k_1, k_2, k_3, k_4):= 4 k_{1,2,3,4}^2- 2k_{1,2} k_{3,4}, \notag \\
& p_{3} (k_1, k_2, k_3, k_4):=7k_{1,2,3,4}^3-8k_{1,2} k_{3,4} k_{1,2,3,4}, \notag \\
& p_{4} (k_1, k_2, k_3, k_4):= 7k_{1,2,3,4}^4- 12k_{1,2} k_{3,4} k_{1,2,3,4}^2 -2k_{1,2}^2 k_{3,4}^2 \notag \\
& \hspace{0.5cm}  -2 (k_1 k_2 k_{1,2} +k_3k_4 k_{3,4}) k_{1,2,3,4}+2( k_1 k_2 k_{1,2}^2 +  k_3 k_4 k_{3,4}^2),  \notag \\
& p_5(k_1k_2, k_3, k_4):= 4k_{1,2,3,4}^5-7k_{1,2} k_{3,4} k_{1,2,3,4}^3 -7 k_{1,2}^2 k_{3,4}^2 k_{1,2,3,4} \notag \\
& \hspace{0.5cm} -3 (k_1k_2k_{1,2}+k_3 k_4 k_{3,4}) k_{1,2,3,4}^2 +  (k_1k_2 k_{1,2}^2 + k_3 k_4 k_{3,4}^2) k_{1,2,3,4} \notag \\
& \hspace{0.5cm} +3 (k_1 k_2 k_{1,2}^3+k_3 k_4 k_{3,4}^3) -(k_1^2 k_2^2 k_{1,2} +k_3^2 k_4^2 k_{3,4} ), \notag \\
&p_6(k_1, k_2,  k_3, k_4):= k_{1,2,3,4}^6- 8k_{1,2}^2 k_{3,4}^2 k_{1,2,3,4}^2- (k_1k_2 k_{1,2}+ k_3 k_4 k_{3,4}) k_{1,2,3,4}^3 \notag \\
&\hspace{0.5cm}  -2 (k_1 k_2 k_{1,2}^2 +k_3 k_4 k_{3,4}^2) k_{1,2,3,4}^2 +4 (k_1 k_2 k_{1,2}^3 + k_3 k_4 k_{3,4}^3) k_{1,2,3,4}
\notag \\
& \hspace{0.5cm}-(k_1^2k_2^2 k_{1,2}+ k_3^2 k_4^2 k_{3,4}) k_{1,2,3,4}+ (k_1 k_2 k_{1,2}^4 + k_3 k_4 k_{3,4}^4) \notag \\
& \hspace{0.5cm} - (k_1^2 k_2^2 k_{1,2}^2 + k_3^2 k_4^2 k_{3,4}^2),  \notag \\
& p_7(k_1, k_2, k_3, k_4):= 
k_{1,2} k_{3,4} k_{1,2,3,4}^5- 3 k_{1,2}^2 k_{3,4}^2 k_{1,2,3,4}^3 \notag \\
& \hspace{0.5cm} -(k_1k_2 k_{1,2}^2 +k_3 k_4 k_{3,4}^2) k_{1,2,3,4}^3 + (k_1 k_2 k_{1,2}^3 + k_3 k_4 k_{3,4}^3) k_{1,2,3,4}^2 \notag \\
&  \hspace{0.5cm}+ (k_1 k_2 k_{1,2}^4 + k_3 k_4 k_{3,4}^4  ) k_{1,2,3,4} -(k_1^2 k_2^2 k_{1,2}^2 +k_3^2 k_4^2 k_{3,4}^2) k_{1,2,3,4}. 
\end{align}
\end{lem}
%%%%%%%%%%%%%%%%%%%%%%%%%%%%%%%%%%%%%%%
%%%%%%%%%%%%%%%%%%%%%%%%%%%%%%%%%%%%%%%
\begin{proof}
Since 
\begin{equation*}
 \frac{Q_1^{(3)}}{\Phi_{0}^{(3)}}= \frac{2}{15} \gamma \, \frac{k_{1,2,3}}{k_{1,2} k_{2,3} k_{1,3}}, \hspace{0.5cm}
 Q_1^{(3)}=-\frac{2}{3} i \ga k_{1,2,3} (k_3^2 +k_{1,2}k_3+k_{1,2}^2-k_1 k_2),
\end{equation*}
it follows that 
\begin{align}
q_2^{(5)} 
= -\frac{2}{5} i \ga^2  \frac{ L(k_1, k_2, k_3, k_4, k_5)+ L(k_3, k_4, k_1, k_2, k_5)  }{ k_{1,2} k_{3,4} k_{1,2,3,5} k_{1,2,4,5} k_{1,3,4, 5} k_{2,3,4,5}  } \label{sym12}
\end{align}
where
\begin{equation*}
L(k_1, k_2, k_3, k_4, k_5)= k_{3,4} k_{1,2,3,5} k_{1,2,4,5} k_{1,2,3,4,5} k_{3,4,5} (k_5^2 +k_{3,4} k_5 + k_{3,4}^2 -k_3 k_4).  
\end{equation*}
Since
\begin{align*}
&k_{1,2,3,5} k_{1,2,4,5} k_{1,2,3,4,5} k_{3,4,5}
  =k_5^4+3k_{1,2,3,4}k_5^3+(3k_{1,2,3,4}^2+k_{3,4}k_{1,2,3,4}+(k_3k_4-k_{3,4}^2))k_5^2\\
  &+(k_{1,2,3,4}^3+2k_{3,4}k_{1,2,3,4}^2+(k_3k_4-2k_{3,4}^2)k_{1,2,3,4}+k_3k_4k_{3,4})k_5\\
  &+(k_{3,4}k_{1,2,3,4}^3-k_{3,4}^2k_{1,2,3,4}^2+k_3k_4k_{3,4}k_{1,2,3,4}),
\end{align*}
a direct computation yields that 
\begin{equation*}
L(k_1,k_2, k_3, k_4, k_5)= \sum_{l=1}^7 r_l (k_1, k_2, k_3, k_4) k_5^{7-l}
\end{equation*}
where each $r_l (k_1, k_2, k_3, k_4) $ is a polynomial of degree $l$ for $l=1,2 \dots, 7$ and  
\begin{align*} % \label{def_q21}
& r_1(k_1, k_2, k_3, k_4)= k_{3,4}, \hspace{0.5cm} r_2(k_1, k_2, k_3, k_4)= 3 k_{3,4} k_{1,2,3,4} +k_{3,4}^2,  \notag \\
& r_3(k_1, k_2, k_3, k_4)= 3k_{3,4} k_{1,2,3,4}^2 + 4 k_{3,4}^2 k_{1,2,3,4}, \notag \\
& r_4(k_1, k_2, k_3, k_4)= k_{3,4} k_{1,2,3,4}^3 +5k_{3,4}^2 k_{1,2,3,4}^2 +2 k_{3,4}^3 k_{1,2,3,4} -k_{3,4}^4 \notag  \\
& \hspace{0.5cm} -2 k_3  k_4 k_{3,4} k_{1,2,3,4} + 2 k_3 k_4 k_{3,4}^2, \notag \\
& r_5(k_1, k_2, k_3, k_4)= 2 k_{3,4}^2 k_{1,2,3,4}^3+ 4 k_{3,4}^3 k_{1,2,3,4}^2 -k_{3,4}^4 k_{1,2,3,4}-k_{3,4}^5 \notag \\
&  \hspace{0.5cm} -3 k_3 k_4 k_{3,4} k_{1,2,3,4}^2+ k_3 k_4 k_{3,4}^2 k_{1,2,3,4}  - k_3^2 k_4^2 k_{3,4} +3 k_3 k_4 k_{3,4}^3 , \notag    \\
& r_6(k_1, k_2, k_3, k_4)= 2 k_{3,4}^3 k_{1,2,3,4}^3 + k_{3,4}^4 k_{1,2,3,4}^2 -2 k_{3,4}^5 k_{1,2,3,4}-k_3 k_4 k_{3,4} k_{1,2,3,4}^3 \notag  \\
& \hspace{0.5cm} -2 k_3 k_4 k_{3,4}^2  k_{1,2,3,4}^2 +4k_3 k_4  k_{3,4}^3 k_{1,2,3,4} -k_3^2 k_4^2 k_{3,4} k_{1,2,3,4}
+k_3 k_4 k_{3,4}^4 - k_3^2 k_4^2 k_{3,4}^2 ,  \notag \\  
& r_7(k_1, k_2, k_3,k_4)= k_{3,4}^4  k_{1,2,3,4}^3-k_{3,4}^5 k_{1,2,3,4}^2 -k_3 k_4 k_{3,4}^2 k_{1,2,3,4}^3 +k_3 k_4 k_{3,4}^3 k_{1,2,3,4}^2 \notag \\
& \hspace{0.5cm} +k_3 k_4 k_{3,4}^4 k_{1,2,3,4} -k_3^2 k_4^2 k_{3,4}^2 k_{1,2,3,4}.
\end{align*}
By
\begin{align*}
& k_{1,2}^2 + k_{3,4}^2 =k_{1,2,3,4}^2 - 2 k_{1,2} k_{3,4}, \hspace{0.5cm} k_{1,2}^3 + k_{3,4}^3 = k_{1,2,3,4}^3 - 3 k_{1,2} k_{3,4} k_{1,2,3,4}, \\
& k_{1,2}^4 + k_{3,4}^4 = k_{1,2,3,4}^4-4 k_{1,2} k_{3,4} k_{1,2,3,4}^2 + 2k_{1,2}^2 k_{3,4}^2, \\
& k_{1,2}^5 + k_{3,4}^5 = k_{1,2,3,4}^5 -5 k_{1,2} k_{3,4} k_{1,2,3,4}^3 + 5 k_{1,2}^2 k_{3,4}^2 k_{1,2,3,4},
\end{align*}
it follows that 
\begin{align} \label{sym13}
& L(k_1, k_2, k_3, k_4, k_5)+ L(k_3, k_4, k_1, k_2, k_5)
%& = \sum_{l=1}^7 (r_l(k_1, k_2, k_3, k_4)+ r_l (k_3, k_4, k_1, k_2) ) k_5^{7-l} 
= \sum_{l=1}^7 p_l(k_1, k_2, k_3, k_4) k_5^{7-l}.
\end{align}
Substituting \eqref{sym13} into \eqref{sym12}, we obtain \eqref{sym21}. 
\end{proof}

%%%%%%%%%%%%%%%%%%%%%%%%%%%%%%%%%%%%%%%%%%%%%%%%%%%%%%%%%%%%
%%%%%%%%%%%%%%%%%%%%%%%%%%%%%%%%%%%%%%%%%%%%%%%%%%%%%%%%%%
\begin{prop} \label{prop_res1}
It follows that 
\begin{equation} \label{re11}
| M_{1,\vp}^{(5)} + M_{9, \vp}^{(5)}  | 
\lesssim \Big( \frac{\max \{|k_1k_2|, |k_3 k_4| \}  }{ |k_{1,2}|  } +\max_{1 \le j \le 4} \{ |k_j| \} \Big) \, 
 \chi_{H1}^{(5)} \chi_{R1}^{(5)} (1- \chi_{R2}^{(5)}). 
\end{equation}
\end{prop}
%%%%%%%%%%%%%%%%%%%%%%%%%%%%%%%%%%%%%%%%%
\begin{proof}
Put $M:= \chi_{H1}^{(5)} \chi_{R1}^{(5)} (1- \chi_{R2}^{(5)}) $. 
Since $(k_1, k_2, k_3, k_4, k_5) \in \supp M$ implies that $k_{1,2,3,4}=0$ and 
$k_{2,3,4,5} k_{1,3,4,5} k_{1,2,4,5} k_{1,2,3,5}= \prod_{i=1}^4 (k_5-k_i)$, 
it follows that $M_{1, \vp}^{(5)}= - \frac{4}{5} i \ga^2 k_5 \, M  $ and 
by Lemma~\ref{lem_key}
\begin{align*}
M_{9, \vp}^{(5)}= & - \frac{2}{5} i \ga^2 \frac{M}{ k_{1,2} k_{3,4} \prod_{i=1}^4 (k_5- k_i) } 
\sum_{l=1}^7 p_l (k_1, k_2, k_3, k_4) k_5^{7-l} \notag \\
=& \frac{4}{5} i \ga^2 \frac{k_5^5}{ \prod_{i=1}^4 (k_5-k_i) } M \\
& -\frac{2}{5} i \ga^2 \frac{M}{k_{1,2} k_{3,4} \prod_{i=1}^4 (k_5-k_i)} \sum_{l=4}^6 f_{l}(k_1, k_2, k_3, k_4) k_5^{7-l}
\end{align*}
where
\begin{align*}
& f_4(k_1,  k_2, k_3, k_4)=-2k_{1,2}^2 k_{3,4}^2+2 (k_1 k_2 k_{1,2}^2+  k_3 k_4 k_{3,4}^2) , \\
& f_5(k_1,  k_2, k_3, k_4)=3 (k_1 k_2 k_{1,2}^3 + k_3 k_4 k_{3,4}^3) -(k_1^2 k_2^2 k_{1,2} + k_3^2 k_4^2 k_{3,4}) , \\
& f_6(k_1,  k_2, k_3, k_4)=(k_1 k_2 k_{1,2}^4 +  k_3 k_4 k_{3,4}^4) -(k_1^2 k_2^2 k_{1,2}^2+ k_3^2 k_4^2 k_{3,4}^2). 
\end{align*}
We notice that 
\begin{equation*}
\frac{4}{5} i \ga^2 \frac{k_5^5}{ \prod_{i=1}^4 (k_5-k_i)  } M - \frac{4}{5} i \ga^2 k_5 M
= \frac{4}{5} i \ga^2 \frac{M}{ \prod_{i=1}^4 (k_5-k_i)  } \sum_{l=1}^4 g_l(k_1, k_2, k_3,  k_4) k_5^{5-l}
\end{equation*}
where 
\begin{align*}
& g_1(k_1, k_2, k_3, k_4)  = k_{1,2,3,4}, \hspace{0.5cm} g_2(k_1, k_2, k_3,k_4)  = -(k_1 k_2+k_3k_4 +k_{1,2}k_{3,4}), \\
& g_3(k_1, k_2, k_3, k_4) = k_1k_2 k_{3,4} + k_3 k_4 k_{1,2}, \hspace{0.5cm} g_4(k_1, k_2, k_3, k_4) = -k_1k_2k_3k_4.  
\end{align*}
Thus, it follows that 
\begin{align*} % \label{res01}
M_{1, \vp}^{(5)}+ M_{9, \vp}^{(5)}= &
\frac{4}{5} i \ga^2 \frac{M}{ \prod_{i=1}^4 (k_5 -k_i) } \sum_{l=2}^4 g_l(k_1,  k_2, k_3, k_4) k_5^{5-l} \notag \\
& - \frac{2}{5} i \ga^2 \frac{M}{k_{1,2} k_{3,4} \prod_{i=1}^4 (k_5-k_i)} \sum_{l=4}^6 f_l(k_1, k_2, k_3, k_4) k_5^{7-l}. 
\end{align*}
By the definition of $\{ f_l \}_{l=4}^6$ and $\{ g_l \}_{l=2}^4$, we have 
\begin{align*} % \label{res02}
& \sum_{l=3}^6 |f_l(k_1, K_2, k_3, k_4) k_5^{7-l} | \, M 
\lesssim \big( \max\{ |k_1 k_2 k_{1,2}|, |k_3 k_4 k_{3,4}| \} + |k_{1,2}^2 k_{3,4}|\big) |k_5|^4 M, \notag  \\
& \sum_{l=2}^4 |g_l(k_1, k_2, k_3, k_4) k_5^{5-l} | \, M \lesssim \max \{ |k_1 k_2|, |k_3 k_4|, |k_{1,2}|^2 \} |k_5|^3 M.
\end{align*}
Therefore, we obtain \eqref{re11}. 
% By \eqref{res01} and \eqref{res02}, we obtain \eqref{re11}.  
\end{proof}

%%%%%%%%%%%%%%%%%%%%%%%%%%%%%%%%%%%%%%%%%%%%%%%%%%%%%
%%%%%%%%%%%%%%%%%%%%%%%%%%%%%%%%%%%%%%%%%%%%%%%%%%%%%
\begin{prop} \label{prop_res2}
It follows that $|\ti{M}_{6, \vp}^{(7)}| \lesssim 1$. 
\end{prop}
%%%%%%%%%%%%%%%%%%%%%%%%%%%%%%%%%%%%%%%%%%%%%%%%%%%%%%%%%
For the proof of Proposition~\ref{prop_res2}, we prepare the following lemma. 
%%%%%%%%%%%%%%%%%%%%%%%%%%%%%%%%%%%%%%%%%%%%%%%%%%%%%%%%%%%%%%%%%
\begin{lem} \label{L41}
Put
\begin{align} \label{def41}
& q_1^{(7)}(k_1,k_2, k_3, k_4, k_5, k_6, k_7)=
k_{5,6} k_{1,2,3,4,5,7} k_{1,2,3,4,6,7} (k_{5,6}^2 +k_{6,7}^2 + k_{5,7}^2) \notag \\
& \times \sum_{l=1}^7 p_{l}(k_1, k_2, k_3, k_4) k_{5,6,7}^{8-l} \Phi_0^{(5)} (k_3, k_4, k_5, k_6, k_{1,2,7}) \Phi_0^{(5)} (k_1, k_2, k_5,  k_6, k_{3,4,7})
\end{align}
where $\{ p_l \}_{l=1}^{7}$ is as in \eqref{def_q2} and 
\begin{align*}
\chi_{R7}^{(7)}=
\begin{cases}
1, \hspace{0.3cm} \text{when} \hspace{0.3cm} k_{1,2,3,4,5,6}=0, ~~k_{1,2,3,4} k_{3,4,5,6} k_{1,2,5,6}  \neq 0 \\
\hspace{2cm} |k_7|^{3/5} > 8^5 \max_{j=1,2,3,4,5,6} \{  |k_j|  \} \\
0, \hspace{0.3cm} \text{otherwise}
\end{cases}.
\end{align*}
Then, it follows that
\begin{align} \label{L410}
& q_1^{(7)} = -50 k_{5,6} k_{1,2,3,4} k_{1,2,5,6} k_{3,4,5,6} k_7^{19} 
+ R_1(k_1,k_2., k_3, k_4, k_5, k_6, k_7)
% &  \hspace{0.3cm} -50 k_{5,6} k_{1,2,3,4} k_{3,4,5,6} k_{1,2,5,6}  (12 k_{1,2,3,4,5,6} +k_{5,6} )k_7^{18} 
% +100 k_{1,2} k_{3,4} k_{5,6} k_{1,2,5,6} k_{3,4,5,6} k_7^{18} \notag \\
% & \hspace{0.3cm} +R_1(k_1,k_2, k_3, k_4, k_5, k_6, k_7) 
\end{align}
where $R_1$ is a polynomial of degree $23$ and satisfies
\begin{equation*} % \label{L411}
|R_1(k_1,k_2, k_3, k_4, k_5, k_6, k_7) \chi_{R7}^{(7)}| 
\lesssim |k_{1,2} k_{3,4} k_{5,6} | ( |k_{1,2}|^2+|k_{3,4}|^2 + |k_{5,6}|^2 ) |k_7|^{18} \chi_{R7}^{(7)}. 
\end{equation*}
\end{lem}
%%%%%%%%%%%%%%%%%%%%%%%%%%%%%%%%%%%%%%%%%%%%%
%%%%%%%%%%%%%%%%%%%%%%%%%%%%%%%%%%%%%%%%%%%%%%
\begin{proof} 
For $(k_1,k_2, k_3, k_4, k_5, k_6, k_7) \in \supp \chi_{R7}^{(7)}$, it follows that 
\begin{equation} \label{L412}
|k_{1,2,3,4}|=|k_{5,6}| < 2 \cdot 8^{-5} |k_7|
\end{equation}
and $|k_5^2 k_6^2| \le \max_{1\le j \le 6} \{  |k_j|^4 \} < 8^{-20} |k_7|^3 $. Thus, we have 
\begin{align} \label{L401}
& k_{5,6} k_{1,2,3,4,5,7} k_{1,2,3,4,6,7} (k_{5,6}^2+ k_{6,7}^2 + k_{5,7}^2 ) \notag \\ 
& \hspace{0.5cm}= 2k_{5,6} k_7^4 + R_2(k_1, k_2, k_3, k_4, k_5, k_6, k_7)
\end{align}
where $R_2$ is a polynomial of degree $5$ and satisfies 
\begin{equation} \label{L402}
|R_2(k_1, k_2, k_3, k_4, k_5, k_6, k_7) \chi_{R7}^{(7)} | \lesssim |k_{5,6}|^2 |k_7|^3 \chi_{R7}^{(7)}. 
\end{equation}
%%%%%%%%%%%%%%%%%%%%%%%%%%%%%%%%%%%%%%%%%%%%%%%% 
We notice that $(k_1, k_2, k_3, k_4, k_5, k_6, k_7) \in \supp \chi_{R7}^{(7)}$ leads that \eqref{L412} and
\begin{equation*}
\max \{ |k_1^2 k_2^2|, |k_3^2 k_4^2| \} \le \max_{1 \le j \le 6} \{  |k_j|^4 \} < 8^{-20} |k_7|^3.
\end{equation*}
Thus, by \eqref{def_q2}, we have
\begin{align} \label{L403}
\sum_{l=1}^7 p_l(k_1, k_2, k_3,  k_4) k_{5,6,7}^{8-l} = k_{1,2,3,4} k_{7}^7 + R_3 (k_1, k_2, k_3, k_4, k_5,  k_6, k_7)
\end{align}
where $R_3$ is a polynomial of degree $8$ and satisfies 
%% $R_5(k_1, k_2, k_3, k_4, k_5, k_6, k_7)=R_4(k_1,k_2, k_3, k_4, k_5, k_6, k_7)+ \sum_{l=3}^7 p_l(k_1, k_2, k_3, k_4) k_{5,6,7}^{8-l}$ and 
\begin{equation} \label{L404}
| R_3 (k_1,  k_2, k_3, k_4,  k_5, k_6, k_7) \chi_{R7}^{(7)}| \lesssim (|k_{1,2}|^2 + |k_{3,4}|^2 +  |k_{5,6}|^2) |k_7|^6 \chi_{R7}^{(7)}. 
\end{equation}
%%%%%%%%%%%%%%%%%%%%%%%%%%%%%%%%%%%%%%%%%%%%%
Since $(k_1, k_2, k_3, k_4, k_5,k_6, k_7) \in \supp \chi_{R7}^{(7)}$ leads that 
\begin{equation*}
|k_{1,2,7}|^{3/5} > 8^4 \max_{j=3,4,5,6} \{ |k_j| \}, \hspace{0.5cm} k_{3,4,5.6} \neq 0, 
\end{equation*}
we use Remark~\ref{rem_fac1} to have 
\begin{align*}
\Phi_0^{(5)}(k_3, k_4, k_5, k_6, k_{1,2,7})= -5i k_{3,4,5,6} k_{1,2,7}^4 +R_4(k_1, k_2, k_3, k_4, k_5, k_6, k_7)
\end{align*}
where $R_4$ is a polynomial of degree $5$ and satisfies 
\begin{equation*}
| R_4 (k_1,  k_2, k_3, k_4,  k_5, k_6, k_7) \chi_{R7}^{(7)}| \lesssim |k_{3,4,5,6}|^2 |k_{1,2,7}|^3  \chi_{R7}^{(7)}. 
\end{equation*}
Since $| k_{3,4,5,6} | =|k_{1,2}| < 2 \cdot 8^{-5} |k_7| $ holds for $(k_1, k_2,  k_3, k_4, k_5, k_6, k_7) \in \supp \chi_{R7}^{(7)}$, 
it follows that 
\begin{equation} \label{L413}
\Phi_0^{(5)} (k_3, k_4, k_5, k_6, k_{1,2,7}) =
-5i k_{3,4,5,6} k_7^4 +R_5(k_1, k_2, k_3, k_4, k_5, k_6, k_7)
\end{equation}
and 
\begin{equation*}
| R_5 (k_1,  k_2, k_3, k_4,  k_5, k_6, k_7) \chi_{R7}^{(7)}| \lesssim |k_{1,2}|^2 |k_{7}|^3  \chi_{R7}^{(7)}. 
\end{equation*}
Similarly, we have 
\begin{equation} \label{L414}
\Phi_{0}^{(5)} (k_1, k_2, k_5, k_6, k_{3,4,7})= -5i k_{1,2,5,6} k_7^4 + R_6(k_1, k_2, k_3, k_4, k_5, k_6, k_7)
\end{equation}
and 
\begin{equation*}
| R_6 (k_1,  k_2, k_3, k_4,  k_5, k_6, k_7) \chi_{R7}^{(7)}| \lesssim |k_{3,4}|^2 |k_{7}|^3  \chi_{R7}^{(7)}. 
\end{equation*}
By \eqref{L413} and \eqref{L414}, it follows that 
\begin{align} \label{L405}
& \Phi_0^{(5)} (k_1, k_2, k_5, k_6, k_{3,4,7}) \Phi_0^{(5)}(k_3, k_4, k_5, k_6, k_{1,2,7})  \notag \\
& \hspace{0.5cm} = -25 k_{3,4,5,6} k_{1,2,5,6}  k_7^8 +R_7(k_1, k_2, k_3, k_4, k_5, k_6, k_7)
\end{align}
where $R_7$ is a polynomial of degree $10$ and satisfies 
\begin{equation} \label{L406}
| R_7 (k_1,  k_2, k_3, k_4,  k_5, k_6, k_7) \chi_{R7}^{(7)}| \lesssim |k_{1,2}| |k_{3,4}| (|k_{1,2}|+|k_{3,4}| )   |k_{7}|^7  \chi_{R7}^{(7)}. 
\end{equation}
%%%%%%%%%%%%%%%%%%%%%%%%%%%%%%%%%%%%%%%%%%%%%
Collecting \eqref{L401}--\eqref{L404} and \eqref{L405}--\eqref{L406}, we obtain \eqref{L410}. 
\end{proof}
%%%%%%%%%%%%%%%%%%%%%%%%%%%%%%%%%%%%%%%%%%%%%%%%%%%%
%%%%%%%%%%%%%%%%%%%%%%%%%%%%%%%%%%%%%%%%%%%%%%%%%%%%
\begin{proof}[Proof of Proposition~\ref{prop_res2}]
We notice that $(k_1,  k_2, k_3, k_4, k_5, k_6, k_7) \in \supp \chi_{H1}^{(7)} \chi_{A4}^{(7)}$ implies 
\begin{align*}
& (1-\chi_{R4}^{(5)}) (k_1, k_2, k_3, k_4, k_{5,6,7} )= (1-\chi_{R4}^{(5)}) (k_1, k_3, k_2, k_4, k_{5,6,7}) \\
&  =(1- \chi_{R4}^{(5)}) (k_1, k_4, k_2, k_3, k_{5,6,7})=1.
\end{align*}
Thus, by the  definition of $Q_2^{(5)}$, it follows that 
\begin{align*}
M_{6, \vp}^{(7)}=& K^{(7)}(k_1,k_2,k_3,k_4,k_5,k_6,k_7)+K^{(7)}(k_1,k_3,k_2,k_4,k_5,k_6,k_7)\\
&+K^{(7)}(k_1,k_4,k_2,k_3,k_5,k_6,k_7)
\end{align*}
where
\begin{equation} \label{L407}
K^{(7)}:=- \left(\frac{q_2^{(5)}}{ \Phi_0^{(5)} } \chi_{NR2}^{(5)}\right) (k_1,  k_2, k_3, k_4,  k_{5,6,7}) \left(Q_1^{(3)} \chi_{NR1}^{(3)}\right) (k_5, k_6, k_7) 
\, \chi_{H1}^{(7)} \chi_{A4}^{(7)} \chi_{R1}^{(7)}.
\end{equation}
Since $\ti{M}_{6, \vp}^{(7)}= 3\ti{K}^{(7)}$, we only need to show $| \ti{K}^{(7)} | \lesssim 1 $. 
%% \begin{equation*}  %\label{res30}
%% |\ti{K}^{(7)}| \lesssim 1.
%% \end{equation*}

%%%%%%%%%%%%%%%%%%%%%%%%%%%%%%%%%%%%
Put 
\begin{equation*}
m^{(7)} :=\chi_{NR2}^{(5)}(k_1, k_2, k_3, k_4, k_{5,6,7}) \chi_{NR1}^{(3)} (k_5, k_6, k_7) \chi_{H1}^{(7)} \chi_{A4}^{(7)} \chi_{R1}^{(7)}. 
\end{equation*}
For $(k_1, k_2, k_3, k_4, k_5, k_6, k_7) \in \supp \chi_{H1}^{(7)} \chi_{A4}^{(7)} \chi_{R1}^{(7)}$, it follows that 
\begin{align*}
& \chi_{NR2}^{(5)} (k_1, k_2, k_3, k_4, k_{5,6,7}) \chi_{NR1}^{(3)} (k_5, k_6, k_7) \\
&  =
\begin{cases}
1 \hspace{0.3cm} \text{when} \hspace{0.3cm} k_{1,2} k_{3,4} k_{5,6} k_{1,2,3,4} k_{1,2,5,6} k_{3,4,5,6} \neq 0 \\
0 \hspace{0.3cm} \text{otherwise}
\end{cases},
\end{align*}
which leads that $m^{(7)}$ satisfies 
\begin{align*}
&m^{(7)} (k_1, k_2, k_3, k_4, k_5, k_6, k_7)= m^{(7)}(k_3, k_4, k_5, k_6, k_1, k_2, k_7) \notag \\
& =m^{(7)} (k_1, k_2, k_5, k_6, k_3, k_4, k_7). 
\end{align*}
Thus, it follows that 
\begin{align*}
\ti{K}^{(7)}=& 
-\frac{1}{3} \Big[   \Big\{  \frac{q_2^{(5)}}{\Phi_0^{(5)}} (k_1, k_2, k_3, k_4, k_{5,6,7}) Q_1^{(3)} (k_5, k_6, k_7)  \\
&\hspace{0.8cm}  + \frac{q_2^{(5)}}{\Phi_0^{(5)}} (k_3, k_4,k_5, k_6, k_{1,2,7}) Q_1^{(3)} (k_1, k_2, k_7) \\
& \hspace{0.8cm} +\frac{q_2^{(5)}}{\Phi_0^{(5)}} (k_1, k_2, k_5, k_6, k_{3,4,7}) Q_1^{(3)} (k_3, k_4, k_7)  \Big\} m^{(7)} \Big]_{sym}^{(7)}.
\end{align*}
By the  definition of $Q_1^{(3)}$ and Lemma~\ref{lem_key}, we have 
\begin{align} \label{res31}
& \ti{K}^{(7)}
=  \frac{2}{45} \ga^3 \Big[  \frac{m^{(7)}}{ k_{1,2} k_{3,4} k_{5,6} k_{1,2,3,4,5,7} k_{1,2,3,4,6,7} k_{1,2,3,5,6,7} k_{1,2,4,5,6,7} k_{1,3,4,5,6,7} k_{2,3,4,5,6,7}  }  \notag \\ 
& \times \frac{ q_1^{(7)}(k_1, k_2, k_3,  k_4, k_5, k_6, k_7)+ q_1^{(7)}(k_3, k_4, k_5, k_6, k_1, k_2, k_7)+ q_1^{(7)}(k_1, k_2, k_5, k_6, k_3, k_4, k_7) }
{ \Phi_0^{(5)} (k_1, k_2, k_3,k_4, k_{5,6,7}) \Phi_0^{(5)} (k_3, k_4, k_5, k_6, k_{1,2,7})\Phi_0^{(5)} (k_1, k_2, k_5, k_6, k_{3,4,7} ) }  
\Big]_{sym}^{(7)}
\end{align} %% error/overfull 
where $q_1^{(7)}$ is as in \eqref{def41}. 
By $\supp m^{(7)} \subset \supp \chi_{R7}^{(7)} $, we use Lemma~\ref{L41} to have
\begin{align} \label{res50}
q_1^{(7)}& = (-50 k_{5,6} k_{1,2,3,4} k_{1,2,5,6} k_{3,4,5,6} ) k_7^{19} +R_1(k_1,k_2, k_3, k_4, k_5, k_6, k_7) \notag \\
& =: r_4(k_1, k_2, k_3, k_4, k_5, k_6) k_7^{19} +R_1(k_1, k_2, k_3, k_4, k_5, k_6, k_7) 
\end{align} 
where $R_1$ is a polynomial of degree $23$ and satisfies 
\begin{equation} \label{res51}
|R_1(k_1, k_2, k_3, k_4, k_5, k_6, k_7) \, m^{(7)} | \lesssim  |k_{1,2} k_{3,4} k_{5,6}| (|k_{1,2}|^2 + |k_{3,4}|^2 +|k_{5,6}|^2) |k_7|^{18} m^{(7)}. 
\end{equation} 
For $ (k_1, k_2, k_3, k_4, k_5, k_6, k_7) \in \supp \chi_{R1}^{(7)}$, it follows that 
\begin{equation*}
r_4(k_1, k_2, k_3, k_4, k_5, k_6)= 50 k_{1,2} k_{3,4} k_{5,6}^2
\end{equation*}
which leads  
\begin{align} \label{res52}
& r_4(k_1, k_2, k_3, k_4, k_5, k_6)+ r_4(k_3, k_4, k_5, k_6, k_1, k_2) + r_4(k_1, k_2, k_5, k_6, k_3, k_4) \notag \\
& = 50 k_{1,2} k_{3,4} k_{5,6} k_{1,2,3,4,5,6}=0.   
\end{align}
By \eqref{res50} and \eqref{res52}, we get  
\begin{align} \label{res35}
& \big|\{ q_1^{(7)} (k_1, k_2, k_3, k_4, k_5, k_6, k_7)+ q_1^{(7)}(k_3, k_4, k_5, k_6, k_1, k_2, k_7 )
+ q_1^{(7)} (k_1, k_2, k_5, k_6, k_3, k_4, k_7) \} m^{(7)} \big| \notag \\
& =\big| \{R_1(k_1, k_2, k_3, k_4, k_5, k_6, k_7)+R_1 (k_3, k_4, k_5, k_6, k_1, k_2,k_7)
 + R_1 (k_1, k_2, k_5, k_6, k_3, k_4, k_7) \} m^{(7)} \big| \notag \\ 
& \lesssim |k_{1,2} k_{3,4} k_{5,6}| (|k_{1,2}|^2 + |k_{3,4}|^2 +|k_{5,6}|^2) |k_7|^{18} m^{(7)}. 
\end{align}  %  error//overfull
Here we used \eqref{res51} and symmetry in the last inequality. 
By Lemma~\ref{Le2}, it follows that
\begin{align}
& \big| k_{1,2} k_{3,4} k_{5,6} k_{1,2,3,4,5,7} k_{1,2,3,4,6,7} k_{1,2,3,5,6,7} k_{1,2,4,5,6,7} k_{1,3,4,5,6,7} k_{2,3,4,5,6,7} \notag \\
& \hspace{0.3cm} \times
\Phi_0^{(5)} (k_1, k_2, k_3, k_4, k_{5,6,7}) \Phi_0^{(5)} (k_3, k_4, k_5, k_6, k_{1,2,7}) \Phi_0^{(5)} (k_1, k_2, k_5, k_6, k_{3,4,7}) \big| m^{(7)} 
\notag \\
& \gtrsim |k_{1,2}^2 k_{3,4}^2 k_{5,6}^2| |k_7|^{18} \, m^{(7)}. \label{res38}
\end{align}
By \eqref{res31}, \eqref{res35} and \eqref{res38}, we obtain 
\begin{align*}
| \ti{K}^{(7)} | \lesssim \Big[  \frac{ |k_{1,2}|^2 +|k_{3,4}|^2 + |k_{5,6}|^2  }{ |k_{1,2} k_{3,4} k_{5,6}|  } \, m^{(7)} \Big]_{sym}^{(7)}
\lesssim 1.
\end{align*}
\end{proof}
%%%%%%%%%%%%%%%%%%%%%%%%%%%%%%%%%%%%%%%%%%
%%%%%%%%%%%%%%%%%%%%%%%%%%%%%%%%%%%%%%%%
%%%%%%%%%%%%%%%%%%%%%%%%%%%%%%%%%%%%%%%%
\begin{prop} \label{prop_res3} 
It follows that 
\begin{equation} \label{res41}
| M_{11, \vp}^{(5)}+M_{13, \vp}^{(5)}| \lesssim \frac{ \max_{1 \le j \le 4} \{ |k_j|^2 \} }{ |k_{1,2}| } 
\, \chi_{H1}^{(5)} \chi_{R1}^{(5)} (1- \chi_{R2}^{(5)}). 
\end{equation}
\end{prop}

%%%%%%%%%%%%%%%%%%%%%%%%%%%%%%%%%%%%%%%%
\begin{proof}
By Remark \ref{rem_sym}, we have 
\EQQ{
& M_{11, \vp}^{(5)}= 9 \frac{Q_1^{(3)}}{ \Phi_0^{(3)} } (k_1, k_2, k_{3,4,5}) Q_2^{(3)} (k_3, k_4, k_5) \, \chi_{H1}^{(5)} \chi_{R1}^{(5)} \chi_{NR1}^{(5)}, \\
& M_{13, \vp}^{(5)}= 9 \frac{Q_2^{(3)}}{ \Phi_0^{(3)} } (k_1, k_2, k_{3,4,5}) Q_1^{(3)} (k_3, k_4, k_5) \, \chi_{H1}^{(5)} \chi_{R1}^{(5)} \chi_{NR1}^{(5)}. 
}
Put $M=\chi_{H1}^{(5)} \chi_{R1}^{(5)} \chi_{NR1}^{(5)}$. 
For $(k_1, k_2, k_3, k_4, k_5) \in \supp \chi_{H1}^{(5)}$, it follows that $\chi_{NR1}^{(5)}=1- \chi_{R2}^{(5)}$, which leads that 
\begin{equation} \label{res410}
M=\chi_{H1}^{(5)} \chi_{R1}^{(5)} (1-\chi_{R2}^{(5)}). 
\end{equation}
By direct computation, we have 
\begin{align*}
& M_{11, \vp}^{(5)}+ M_{13, \vp}^{(5)} \\
& = \frac{M}{ \Phi_{0}^{(3)} (k_1, k_2, k_{3,4,5}) }
 \Big\{ 9 Q_1^{(3)} (k_1, k_2, k_{3,4,5}) Q_2^{(3)} (k_3, k_4, k_5) \\
& \hspace{3cm} + 9 Q_2^{(3)} (k_1, k_2, k_{3,4,5}) Q_1^{(3)} (k_3, k_4, k_5)  \Big\} \\
&=(- \be \ga) \frac{M}{ \Phi_0^{(3)} (k_1, k_2, k_{3,4,5})  } 
\Big\{ k_{1,2,3,4,5} (k_{1,2}^2+ k_{1,3,4,5}^2 +k_{2,3,4,5}^2) k_{3,4,5} (k_3k_4+k_{3,4} k_5) \\
& \hspace{3cm} +k_{1,2,3,4,5} (k_1 k_2 +k_{1,2} k_{3,4,5}) k_{3,4,5} (k_{3,4}^2 +k_{3,5}^2 +k_{4,5}^2) \Big\} \\
& = (- \be  \ga) \frac{M}{ \Phi_0^{(3)} (k_1,  k_2, k_{3,4,5}) } \sum_{l=1}^6 h_l (k_1, k_2, k_3, k_4) k_5^{6-l}
\end{align*}
where each $h_l(k_1, k_2, k_3, k_4)$ is  a polynomial of degree $l$ for $l=1,2,\dots, 6$ and 
\begin{align*}
h_1(k_1, k_2, k_3, k_4)=2 k_{1,2,3,4}.  
% & h_2(k_1, k_2, k_3, k_4)=2k_{1,2,3,4}^2 + 6k_{3,4} k_{1,2,3,4}+  2 (k_1k_2+k_3k_4+k_{1,2} k_{3,4}). 
\end{align*}
Since $h_1(k_1, k_2, k_3, k_4)=0$ holds for $(k_1, k_2, k_3, k_4, k_5) \in \supp \chi_{R1}^{(5)}$, we have 
\begin{align} \label{res411}
M_{11, \vp}^{(5)}+ M_{13, \vp}^{(5)}= 
(- \be \ga) \frac{M}{\Phi_0^{(3)} (k_1, k_2, k_{3,4,5}) } \sum_{l=2}^6 h_l(k_1,  k_2, k_3, k_4) k_5^{6-l}.
\end{align}
For $(k_1, k_2, k_3,  k_4,  k_5) \in \supp M$, it follows that 
\begin{equation} \label{res412}
\sum_{l=2}^6 |h_l(k_1, k_2, k_3, k_4)| |k_5|^{6-l} \lesssim \max_{1 \le j \le 4} \{ |k_j|^2 \} \, |k_5|^4
% + \max\{ |k_1k_2|, |k_3 k_4| \} |k_5|^4
\end{equation}
and 
\begin{equation} \label{res413}
|\Phi_{0}^{(3)} (k_1, k_2, k_{3,4,5}) | \gtrsim |k_{1,2}| |k_5|^4.
\end{equation}
Collecting \eqref{res410}--\eqref{res413}, we obtain \eqref{res41}. 
\end{proof}

%%%%%%%%%%%%%%%%%%%%%%%%%%%%%%%%%%%%%
%%%%%%%%%%%%%%%%%%%%%%%%%%%%%%%%%%%%%
%%%%%%% multilinear estimates %%%%%%%%%%%%%%
%%%%%%%%%%%%%%%%%%%%%%%%%%%%%%%%%%%%%
%%%%%%%%%%%%%%%%%%%%%%%%%%%%%%%%%%%%%

\section{multilinear estimates}

In this section, we present several multilinear estimates in order to prove main estimates 
which are stated in Section 7. 

%%%%%%%%%%%%%%%%%%%%%%%%%%%%%%%%%%%%%%%%%%%%%%%%%%%%%%%
%%%%%%%%%%%%%%%%%%%%%%%%%%%%%%%%%%%%%%%%%%%%%%%%%%%%%%%
\begin{lem} \label{lem_nes01}
Let $s \ge 3/2$ and $L>0$. Then, for any $m=1,2,3$, it follows that 
\begin{align}
& \Big\| \sum_{k=k_{1,2,3}} \big[ \langle k_{2,3} \rangle^{-1} \langle k_{\max} \rangle^{-1} \chi_{H2,1}^{(3)} \big]_{sym}^{(3)} \,
 \prod_{l=1}^3 |\ha{v}_l (k_l)|   \Big\|_{l_s^2}
 \lesssim  \| v _m\|_{H^{s-3}} \prod_{l \in \{ 1,2,3  \} \setminus \{ m \} } \|  v_l \|_{H^s} , \label{cnl1} \\
& \Big\| \sum_{k=k_{1,2,3}} \big[ \langle k_{2,3} \rangle^{-1} \langle k_{\max} \rangle^{-1} \chi_{H2,2}^{(3)} \big]_{sym}^{(3)} \, 
\prod_{l=1}^3 |\ha{v}_l (k_l)|   \Big\|_{l_s^2}
 \lesssim  \| v _m\|_{H^{s-3}} \prod_{l \in \{ 1,2,3  \} \setminus \{ m \} } \|  v_l \|_{H^s}, \label{cnl2} \\
 & \Big\| \sum_{k=k_{1,2,3}} \Lambda_1^{-1} \chi_{H3}^{(3)}  \, \prod_{l=1}^3 |\ha{v}_l (k_l)|   \Big\|_{l_s^2}
 \lesssim  \| v _m\|_{H^{s-3}} \prod_{l \in \{ 1,2,3  \} \setminus \{ m \} } \|  v_l \|_{H^s}, \label{cnl3} 
\end{align}
where $\La_1= \min \{ \langle k_{1,2} \rangle \langle k_{1,3} \rangle, \langle k_{1,2} \rangle \langle k_{2,3} \rangle, 
\langle k_{1,3} \rangle \langle k_{2,3} \rangle  \}$. Moreover, we have 
\begin{equation}
\Big\| \sum_{k=k_{1,2,3}} \La_1^{-1} \chi_{H3}^{(3)} \chi_{>L}^{(3)} \, \prod_{l=1}^3 |\ha{v}_l (k_l)|   \Big\|_{l_s^2}
 \lesssim L^{-3} \prod_{l=1}^3  \| v_l \|_{H^s}.  \label{cnl4} 
\end{equation}
\end{lem}

%%%%%%%%%%%%%%%%%%%%%%%%%%%%%%%%%%%%%%%%%%%%
\begin{proof}
First, we prove \eqref{cnl1}. It suffices to show 
\begin{equation} \label{cnl11}
\Big\|  \sum_{k=k_{1,2,3}} \langle k_{2,3} \rangle^{-1} \langle  k_{\max} \rangle^2 \chi_{H2,1}^{(3)} \prod_{l=1}^3 |\ha{v}_l (k_l)|  \Big\|_{l^2} \lesssim \prod_{l=1}^3 \|  v_l \|_{H^s}
\end{equation}
for any $\{  v_l \}_{l=1}^3 \subset H^s (\T)$. 
Since $(k_1, k_2, k_3) \in \supp \chi_{H2,1}^{(3)}$ implies $|k_{1,2,3}| \sim |k_{2,3}|$ and $|k_2| \sim |k_3| \sim k_{\max}$, 
we have 
\begin{equation*}
\langle k_{1,2,3} \rangle^s \langle  k_{2,3} \rangle^{-1} \langle k_{\max}  \rangle^{2} \chi_{H2,1}^{(3)}
\lesssim \langle k_{1,2,3} \rangle^{s-1} \langle k_3 \rangle^2 \chi_{H2,1}^{(3)} 
\lesssim \langle k_2 \rangle^{s-1/4} \langle k_3 \rangle^{5/4}. 
\end{equation*}
Thus, by Lemma~\ref{Le7} with $i=2$, we obtain \eqref{cnl11}. 

Next, we show \eqref{cnl2}. It suffices to show 
\begin{equation} \label{cnl21}
\Big\| \sum_{k=k_{1,2,3}} \langle k_{\max} \rangle^2 \chi_{H2,2}^{(3)} \, \prod_{l=1}^3 |\ha{v}_l(k_l)|  \Big\|_{l^2} 
\lesssim \prod_{l=1}^3 \| v_l  \|_{H^s}
\end{equation}
for any $\{ v_l \}_{l=1}^3 \subset H^s(\T)$. 
For $(k_1, k_2, k_3) \in \chi_{H2,2}^{(3)}$, it follows that $|k_{1,2,3}| \lesssim |k_1|$ and $|k_2| \sim |k_3| \sim k_{\max}$, 
which leads that 
\begin{equation*}
\langle k_{1,2,3} \rangle^{s} \langle k_{\max}  \rangle^{2} \chi_{H2,2}^{(3)} 
\lesssim \langle k_1 \rangle^{s} \langle k_3  \rangle^2 \chi_{H2,2}^{(3)} 
\lesssim \langle k_1 \rangle^{s-1/3} \langle k_2 \rangle^{7/6} \langle k_3  \rangle^{7/6}.
\end{equation*}
Thus, by Lemma~\ref{Le7} with $i=3$, we obtain \eqref{cnl21}. 

Next, we prove \eqref{cnl3}. It suffices to show 
\begin{equation} \label{cnl31}
\Big\| \sum_{k=k_{1,2,3}} \langle k_{1,2,3} \rangle^{s} \La_1^{-1} \langle k_{\max} \rangle^3 \chi_{H3}^{(3)} \, 
\prod_{l=1}^3 |\ha{v}_l (k_l)| \Big\|_{l^2} 
\lesssim \prod_{l=1}^3 \| v_l  \|_{H^s}
\end{equation}
for any $\{ v_l \}_{l=1}^3 \subset H^s(\T)$. 
By symmetry, we may assume that $|k_{1,2}| \le |k_{1,3}| \le |k_{2,3}|$ holds. 
Then, by $\La_1^{-1} \lesssim \langle k_{1,2} \rangle^{-2}$, we have 
\begin{equation*} 
\langle k_{1,2,3} \rangle^{s} \La_1^{-1} \langle k_{\max} \rangle^2 \chi_{H3}^{(3)} 
\lesssim \langle k_{1,2} \rangle^{-2} \prod_{l=1}^3 \langle k_l \rangle^s \langle k_{\max} \rangle^{-2s+3}
\lesssim \langle k_{1,2} \rangle^{-2} \prod_{l=1}^3 \langle k_l  \rangle^s. 
\end{equation*}
Thus, by the Young inequality and the H\"{o}lder inequality,  the left  hand side of  \eqref{cnl31} is bounded by 
\begin{align*} 
& \|\{ \langle \cdot \rangle^{-2} \big(  \langle \cdot \rangle^{s} |\ha{v}_1|* \langle \cdot \rangle^{s} |\ha{v}_2| \big) \} * 
\langle \cdot \rangle^s |\ha{v}_3| \|_{l^2}
\lesssim \| \langle \cdot \rangle^s |\ha{v}_3|  \|_{l^2} 
\| \langle \cdot  \rangle^s |\ha{v}_1| * \langle \cdot \rangle^s |\ha{v}_2|  \|_{l^{\infty}} \\
& \hspace{0.3cm} \lesssim \prod_{l=1}^3 \| v_l \|_{H^s}. 
\end{align*} 
Finally, we show \eqref{cnl4}. Since  
\begin{equation*}
\langle k_{1,2,3} \rangle^{s} \La^{-1} \chi_{H3}^{(3)} \chi_{>L}^{(3)} 
\lesssim L^{-3} \langle k_{1,2,3} \rangle^{s} \La^{-1} \langle k_{\max} \rangle^3 \chi_{H3}^{(3)}, 
\end{equation*}
by \eqref{cnl31}, we get \eqref{cnl4}. 
\end{proof}

%%%%%%%%%%%%%%%%%%%%%%%%%%%%%%%%%%%%%%%%%%%%%%%%%%%%%
\begin{lem} \label{lem_nl5}
Let $s \ge 3/2$. Then, for any $\{  v_l \}_{l=1}^3 \subset H^s(\T)$, it follows that 
\begin{equation} \label{cnl5}
\Big\| \sum_{k=k_{1,2,3}} \langle k_{\min} \rangle \langle k_{\max} \rangle^2 
\big([3 \chi_{H2,2}^{(3)}]_{sym}^{(3)} + \chi_{H3}^{(3)} \big)  \prod_{l=1}^3 |\ha{v}_l (k_l)| \Big\|_{l_{s-1}^2} \lesssim \prod_{l=1}^3 \| v_l \|_{H^s}
\end{equation}
where $k_{\min}:=\min \{ |k_1|, |k_2|, |k_3| \}$. 
\end{lem}

%%%%%%%%%%%%%%%%%%%%%%%%%%%%%%%%%%%%%%%%%%%%%%%%%%%%%%
\begin{proof}
It suffices to show 
\begin{equation} \label{cnl51}
\Big\| \sum_{k=k_{1,2,3}} \langle k_{1,2,3} \rangle^{s-1}  \langle k_{\min} \rangle \langle k_{\max} \rangle^2 
(\chi_{H2,2}^{(3)} + \chi_{H3}^{(3)})  \prod_{l=1}^3 |\ha{v}_l (k_l)| \Big\|_{l^2} \lesssim \prod_{l=1}^3 \| v_l \|_{H^s}.
\end{equation}
For $(k_1, k_2, k_3) \in \supp  \chi_{H2,2}^{(3)} \cup \supp \chi_{H3}^{(3)}$, 
it follows that $|k_{1,2,3}| \lesssim |k_1| \sim  k_{\min} $ and $|k_2| \sim |k_3| \sim k_{\max}$, which leads that 
\begin{align*}
\langle k_{1,2,3} \rangle^{s-1} \langle k_{\min} \rangle \langle k_{\max} \rangle^2 (\chi_{H2,2}^{(3)}+ \chi_{H3}^{(3)})
% \lesssim \langle k_1 \rangle^s \langle k_3 \rangle^2 (\chi_{H2,2}^{(3)} + \chi_{H3}^{(3)})
\lesssim \langle k_1 \rangle^{s-1/3} \langle k_2 \rangle^{7/6} \langle k_3 \rangle^{7/6}.
\end{align*}
Thus, by Lemma~\ref{Le7} with $i=3$, we obtain \eqref{cnl51}. 
\end{proof}

%%%%%%%%%%%%%%%%%%%%%%%%%%%%%%%%%%%%%%%%%%%%%%%%%%%%%%%
\begin{lem} \label{lem_stres}
Let $s \ge 3/2$. Then, for any $\{ v_l \}_{l=1}^3 \subset H^s(\T)$, we have 
\begin{equation} \label{stres}
\big\| \sum_{k=k_{1,2,3}} \big| Q^{(3)} [3 \chi_{R3}^{(3)}]_{sym}^{(3)} \big| \, \prod_{l=1}^3 |\ha{v}_l(k_l)|  \big\|_{l_s^2} \lesssim \prod_{l=1}^3 \| v_l  \|_{H^s}.  
\end{equation}
\end{lem}
%%%%%%%%%%%%%%%%%%%%%%%%%%%%%%%%%%%%
\begin{proof}
By the H\"{o}lder inequality and the continuous embedding $l^2 \hookrightarrow l^6$, the left hand  side of \eqref{stres} is  bounded by 
\begin{equation*}
\big\| \langle \cdot \rangle^{s+3} \prod_{l=1}^3 |\ha{v}_l|  \big\|_{l^2} \le 
\prod_{l=1}^3 \| \langle \cdot  \rangle^{s/3+1}|\ha{v}_l| \|_{l^6} \lesssim \prod_{l=1}^3 \|  v_l \|_{H^{s/3+1}} \le \prod_{l=1}^3 \| v_l \|_{H^s}.  
\end{equation*}
Here we used $s \ge 3/2$ in the last inequality. 
\end{proof}

%%%%%%%%%%%%%%%%%%%%%%%%%%%%%%%%%%%%%%%%%%%%%
%%%%%%%%%%%%%%%%%%%%%%%%%%%%%%%%%%%%%%%%%%%%%
%%%%%%%%%%%%% pointwise upper bounds %%%%%%%%%%%%
%%%%%%%%%%%%%%%%%%%%%%%%%%%%%%%%%%%%%%%%%%%%%
%%%%%%%%%%%%%%%%%%%%%%%%%%%%%%%%%%%%%%%%%%%%%

\section{pointwise upper bounds}

%%%%%%%%%%%%%%%%%%%%%%%%%%%%%%%%%%%%%%%%%%%%%%%%
In this section, we present pointwise upper bounds of some multipliers $L_{j, \vp}^{(N)}$ and $M_{j, \vp}^{(N)}$ defined in Proposition~\ref{prop_NF2}. 
We now put 
\begin{align*}
& J_1=\{(5,1), (5,7), (5,8), (7,1), (7,2)\}, \hspace{0.5cm} J_2=\{ (5,3), (5,4), (5,5), (5,6)  \}, \\
& J_3=\{ (3,2), (3,3), (3,4)   \}, \hspace{0.5cm} J_4=\{ (3,1), (5,2)  \}.
\end{align*}
%%%%%%%%%%%%%%%%%%%%%%%%%%%%%%%%%%%%%%%
%%%%%%%%%%%%%%%%%%%%%%%%%%%%%%%%%%%%%%%%
\begin{lem} \label{lem_pwb1}
Let $f \in L^2(\T)$ 
and $L_{j, f}^{(N)}$ with $(N, j) \in J_1 \cup J_2 \cup J_3 \cup J_4$ be as in Proposition~\ref{prop_NF2}. 
Then, the following hold for $L \gg \max \{1, | \ga| E_1(f) \}$: \\
(I) When $(N, j) \in J_4$, we have
\begin{align} 
|L_{j,f}^{(N)} \chi_{>L}^{(N)} | \lesssim \langle k_{\max}  \rangle^{-1} \chi_{>L}^{(N)}. \label{pwb3} 
\end{align}
(II) When $(N,j) \in J_1$, we have
\begin{align} \label{pwb1}
|L_{j,f}^{(N)} \chi_{>L}^{(N)} | \lesssim \langle k_{1, \dots ,N}  \rangle^{-1} \langle k_{\max} \rangle^{-2} \chi_{>L}^{(N)}. 
\end{align}
(III) When $(N, j) \in J_2$, we have 
\begin{equation} \label{pwb0}
|L_{j, \vp}^{(N)} \chi_{>L}^{(N)} | \lesssim 
\max \big\{ \frac{ \langle k_1 k_2 \rangle }{ \langle k_{1,2}  \rangle},  \frac{ \langle k_{3} k_4 \rangle}{ \langle k_{3,4} \rangle } \big\} 
\langle k_5 \rangle^{-3} \, \chi_{H1}^{(5)} \chi_{>L}^{(5)}. 
\end{equation}
(IV) It follows that 
\begin{align} 
& |L_{2, f}^{(3)} \chi_{>L}^{(3)} | \lesssim \langle k_{2,3} \rangle^{-1} \langle k_{\max} \rangle^{-1} \chi_{H2,1}^{(3)} \chi_{>L}^{(3)}, 
\label{pwb01}  \\
& |L_{3, f}^{(3)} \chi_{>L}^{(3)} | \lesssim \langle k_{2,3} \rangle^{-1} \langle k_{\max} \rangle^{-1} \chi_{H2,2}^{(3)} \chi_{>L}^{(3)}, 
\label{pwb02}  \\
& |L_{4, f}^{(3)} \chi_{>L}^{(3)} | \lesssim \La_1^{-1} \chi_{H3}^{(3)}  \chi_{>L}^{(3)} \label{pwb03} 
\end{align}
where $\La_1$ is as in Lemma~\ref{lem_nes01}. 
\end{lem}

%%%%%%%%%%%%%%%%%%%%%%%%%%%%%%%%%
%%%%%%%%%%%%%%%%%%%%%%%%%%%%%%%%%
\begin{rem} \label{rem_pwb11}
By Lemma~\ref{lem_pwb1} and the definition of the multipliers, we can easily check that 
\begin{equation*} 
|L_{j, f}^{(N)} \chi_{>L}^{(N)} | \lesssim 1, \hspace{0.5cm} |L_{j, f}^{(N)} \Phi_f^{(N)} \chi_{>L}^{(N)}| \lesssim \langle k_{\max} \rangle^3 
\end{equation*}
for $(N,j) \in J_1 \cup J_2 \cup J_3 \cup J_4$. 
Thus, we can apply Proposition~\ref{prop_NF11} with $\ti{m}^{(N)}=\ti{L}_{j, f}^{(N)} \chi_{>L}^{(N)}$. 
%%%%%%%%%%%%%%%%%%%%%%%%%%%%%%%%%%%%%%%%%%%%%%%%%%%%%%%
\end{rem}

%%%%%%%%%%%%%%%%%%%%%%%%%%%%%%%%%%%%%%%%%%%%%
\begin{proof}
(I) First, we prove (\ref{pwb3}) for $(N, j) \in J_4$. \\ 
(Ia) Estimate of $L_{1, f}^{(3)} \chi_{>L}^{(3)}$: By Lemma~\ref{Le1}, we have
$|L_{1, f}^{(3)} \chi_{>L}^{(3)} | \lesssim \langle k_{\max} \rangle^{-1}  \chi_{>L}^{(3)}$. \\
(Ib) Estimate of $L_{2,f}^{(5)} \chi_{>L}^{(5)}$: By \eqref{def_q2}, it follows that
\begin{align} \label{es_q2}
& \sum_{l=1}^7 |p_l(k_1, k_2, k_3, k_4) k_5^{7-l} | \, \chi_{H1}^{(5)} \notag \\
& \hspace{0.3cm} 
\lesssim \big( |k_{1,2,3,4}| |k_5|^6 + |k_{1,2} k_{3,4} | |k_5|^5 + \max \{ |k_1 k_2 k_{1,2}|, |k_3 k_4 k_{3,4}| \} |k_5|^4 \big) \chi_{H1}^{(5)}.
\end{align}
Thus, by Lemma~\ref{lem_key}, we have 
\begin{align*}
|L_{2, f}^{(5)} \Phi_f^{(5)}|& =|q_2^{(5)} \chi_{H1}^{(5)} \chi_{NR2}^{(5)} (1-\chi_{R4}^{(5)}) | \notag \\
& \lesssim \Big( \frac{|k_{1,2,3,4}| }{ |k_{1,2}| |k_{3,4}| } |k_5|^2 +|k_5|+ \max \Big\{ \frac{|k_1 k_2|}{|k_{3,4}|}, \frac{ |k_3 k_4| }{|k_{1,2}| }  \Big\} \Big) \chi_{H1}^{(5)} \chi_{NR2}^{(5)} (1-\chi_{R4}^{(5)}).
\end{align*}
By Lemma~\ref{Le2}, it follows that 
\begin{equation*}
|\Phi_f^{(5)} \chi_{H1}^{(5)} \chi_{NR2}^{(5)} (1- \chi_{R4}^{(5)}) \chi_{>L}^{(5)} | \gtrsim |k_{1,2,3,4}| |k_5|^4 \chi_{H1}^{(5)} \chi_{NR2}^{(5)} (1- \chi_{R4}^{(5)}) \chi_{>L}^{(5)}. 
\end{equation*}
Hence, we have $ |L_{2, f}^{(5)} \chi_{>L}^{(5)}| \lesssim \langle k_{\max}  \rangle^{-2} \chi_{>L}^{(5)}$. \\
%%%%%%%%%%%%%%%%%%%%%%%%%%%%%%%%%%%%%%%%%%%%%%%%%%%
\noindent
(II) Next, we prove (\ref{pwb1}) for $(N,j) \in J_1$. \\
(IIa) Estimate of $L_{1,f}^{(5)} \chi_{>L}^{(5)}$: For $(k_1, k_2, k_3, k_4, k_5) \in \supp \chi_{H1}^{(5)} (1-\chi_{R1}^{(5)}) (1-\chi_{R5}^{(5)}) $, 
either \eqref{rel3} or \eqref{rel30} holds. Thus, by Lemma~\ref{Le2}, it follows that 
\begin{equation*}
|\Phi_{f}^{(5)} \chi_{H1}^{(5)} (1-\chi_{R1}^{(5)}) (1- \chi_{R5}^{(5)}) \chi_{>L}^{(5)} | \gtrsim  
|k_{1,2,3,4} | |k_5|^4  \chi_{H1}^{(5)} (1-\chi_{R1}^{(5)}) (1-\chi_{R5}^{(5)}) \chi_{>L}^{(5)}, 
\end{equation*}
which leads $|L_{1,f}^{(5)} \chi_{>L}^{(5)}| \lesssim \langle k_{\max} \rangle^{-3} \chi_{>L}^{(5)}$. \\
%%%%%%%%%%%%%%%%%%%%%%%%%%%%%%%%%%%%%
(IIb) Estimate of $L_{7, f}^{(5)} \chi_{>L}^{(5)}$: 
For $(k_1, k_2, k_3, k_4, k_5) \in \supp \chi_{NR(1,1)}^{(5)} (1-\chi_{H1}^{(5)}) \chi_{A1}^{(5)}  $, it follows that 
$ 96 \, \text{sec}_{1 \le j \le 4} \{ |k_j| \} < 6 \max_{1 \le j \le 4} \{ |k_j| \} < |k_5|=k_{\max} \lesssim \max_{1 \le j \le 4} \{ |k_j| \} $. 
Thus, by Lemma~\ref{Le2} (i),  it follows that 
\begin{align*}
| \Phi_{f}^{(5)} \chi_{NR1}^{(5)} \chi_{NR(1,1)}^{(5)} (1- \chi_{H1}^{(5)}) \chi_{A1}^{(5)}  \chi_{>L}^{(5)} | 
& \gtrsim \max_{1\le j \le 4} \{ |k_j| \} |k_5|^4 \chi_{NR1}^{(5)} \chi_{NR(1,1)}^{(5)} (1-\chi_{H1}^{(5)}) \chi_{A1}^{(5)} \chi_{>L}^{(5)} \\
& \gtrsim |k_5|^5 \chi_{NR1}^{(5)} \chi_{NR(1,1)}^{(5)} (1-\chi_{H1}^{(5)}) \chi_{A1}^{(5)} \chi_{>L}^{(5)}, 
\end{align*}
which implies that $|L_{7, f}^{(5)} \chi_{>L}^{(5)}| \lesssim \langle k_{\max} \rangle^{-3} \chi_{>L}^{(5)} $ holds. \\
%%%%%%%%%%%%%%%%%%%%%%%%%%%%%%%%%%%%%%%%%%%%%%%%%%%%%%%%%%%%%%%%%%%
(IIc) Estimate of $L_{8, f}^{(5)} \chi_{>L}^{(5)}$: By Lemma~\ref{Le4}, we have  
\begin{equation*}
|\Phi_f^{(5)} \chi_{NR1}^{(5)} \chi_{NR(2,1)}^{(5)} \chi_{A3}^{(5)} \chi_{>L}^{(5)} | 
\gtrsim |k_{1,2,3,4,5}| |k_5|^4 \chi_{NR1}^{(5)} \chi_{NR(2,1)}^{(5)} \chi_{A3}^{(5)} \chi_{>L}^{(5)}, 
\end{equation*}
which leads that $|L_{8, f}^{(5)} \chi_{>L}^{(5)}| \lesssim \langle k_{1,2,3,4,5} \rangle^{-1} \langle k_{\max} \rangle^{-2} \chi_{>L}^{(5)}$ holds. \\
%%%%%%%%%%%%%%%%%%%%%%%%%%%%%%%%%%%%%%%%%%%%%%%%%%%%%%%%%%%%%%%%%
(IId) Estimate of $L_{1, f}^{(7)} \chi_{>L}^{(7)}$: By \eqref{pwes22} in Lemma~\ref{lem_pwb2}, we have 
\begin{equation*}
| L_{1, f}^{(7)} \Phi_f^{(7)}  | \lesssim \langle k_7 \rangle \, \chi_{H1}^{(7)} (1-\chi_{R1}^{(7)})(1- \chi_{R5}^{(7)}).
\end{equation*}
By Lemma~\ref{Le3}, it follows that 
\begin{equation*}
|\Phi_f^{(7)} \chi_{H1}^{(7)} (1-\chi_{R1}^{(7)}) (1-\chi_{R5}^{(7)}) \chi_{>L}^{(7)} | 
\gtrsim |k_{1,2,3,4,5,6}| |k_7|^4 \chi_{H1}^{(7)} (1-\chi_{R1}^{(7)}) (1-\chi_{R5}^{(7)}) \chi_{>L}^{(7)}.
\end{equation*}
Thus, we have $|L_{1, f}^{(7)} \chi_{>L}^{(7)} | \lesssim \langle k_{\max} \rangle^{-3} \chi_{>L}^{(7)} $. \\
%%%%%%%%%%%%%%%%%%%%%%%%%%%%%%%%%%%%%%%%%%%%%%%%%%%
(IIe) Estimate of $L_{2, f}^{(7)} \chi_{>L}^{(7)}$: By \eqref{pwes22} in Lemma~\ref{lem_pwb2}, we have 
\begin{equation*}
|L_{2, f}^{(7)} \Phi_{f}^{(7)} | \lesssim \max_{1 \le  j\le 6} \{ |k_j| \} \, \chi_{H1}^{(7)} \chi_{A1}^{(7)}. 
\end{equation*}
Since $(k_1, k_2,  k_3, k_4, k_5, k_6, k_7) \in \supp \chi_{H1}^{(7)} \chi_{A1}^{(7)}$ leads that 
\begin{equation*}
|k_7|> 8^5 \max_{1 \le j \le 6} \{ |k_j| \}> 16 \cdot 8^5 \sec_{1 \le j \le 6} \{ |k_j| \}, 
\end{equation*}
by Lemma~\ref{Le3}, it follows that 
\begin{equation*}
|\Phi_{f}^{(7)} \chi_{H1}^{(7)} \chi_{A1}^{(7)} \chi_{>L}^{(7)} | \gtrsim |k_{1,2,3,4,5,6}| |k_5|^4 \chi_{H1}^{(7)} \chi_{A1}^{(7)} \chi_{>L}^{(7)}
\sim \max_{1 \le j \le 6} \{ |k_j| \} |k_5|^4 \chi_{H1}^{(7)} \chi_{A1}^{(7)} \chi_{>L}^{(7)}, 
\end{equation*}
Thus, we get $|L_{2, f}^{(7)} \chi_{>L}^{(7)} |  \lesssim \langle k_{\max}  \rangle^{-4} \chi_{>L}^{(7)}$. 

Therefore, we obtain \eqref{pwb1} for $(N, j) \in J_1$. \\
%%%%%%%%%%%%%%%%%%%%%%%%%%%%%%%%%%%%%%%%%%%%%%%%%%%%%%%%%
\noindent
(III) Next, we prove \eqref{pwb0} for $(N,j) \in J_2$. \\
%%%%%%%%%%%%%%%%%%%%%%%
(IIIa) Estimate of $L_{3, f}^{(5)} \chi_{>L}^{(5)}$ and $L_{4, f}^{(5)} \chi_{>L}^{(5)}$: By Lemma~\ref{Le1}, it follows that 
\begin{align*}
& |L_{3, f}^{(5)} \chi_{>L}^{(5)} | \lesssim \Big( \frac{ |k_3 k_4| }{\langle k_{1,2} \rangle} + \frac{| k_{3,4} |}{ \langle k_{1,2} \rangle } |k_5| \Big) \, \chi_{H1}^{(5)} (1- \chi_{R1}^{(5)}) (1-\chi_{R4}^{(5)}), \\
& |L_{4, f}^{(5)} \chi_{>L}^{(5)} | \lesssim \Big( \frac{ |k_1 k_2| }{ \langle k_{1,2} \rangle} + |k_5|   \Big) \,
\chi_{H1}^{(5)} (1-\chi_{R1}^{(5)}) (1- \chi_{R4}^{(5)}). 
\end{align*}
By Lemma~\ref{Le2}, we have
\begin{equation*}
|\Phi_f^{(5)} \chi_{H1}^{(5)} (1-\chi_{R1}^{(5)})  (1 - \chi_{R4}^{(5)})  \chi_{>L}^{(5)} | 
\gtrsim |k_{1,2,3,4}| |k_5|^4 \chi_{H1}^{(5)} (1-\chi_{R1}^{(5)}) (1-\chi_{R4}^{(5)}) \chi_{>L}^{(5)}. 
\end{equation*}
Thus, by $ \langle k_{3,4} \rangle \lesssim \langle k_{1,2,3,4} \rangle \langle k_{1,2} \rangle$, we get
\begin{align*}
|L_{3,f}^{(5)} \chi_{>L}^{(5)}| \lesssim \frac{\langle k_3 k_4 \rangle}{ \langle k_{3,4} \rangle } \langle k_5 \rangle^{-3} \chi_{H1}^{(5)} \chi_{>L}^{(5)}, 
\hspace{0.5cm}
|L_{4,f}^{(5)} \chi_{>L}^{(5)}| \lesssim \frac{\langle k_1 k_2 \rangle}{ \langle k_{1,2} \rangle } \langle k_5 \rangle^{-3} \chi_{H1}^{(5)} 
\chi_{>L}^{(5)}. 
\end{align*}
%%%%%%%%%%%%%%%%%%%%%%%%%%%%%%%%%%%%%%
(IIIb) Estimate of $L_{5, f}^{(5)} \chi_{>L}^{(5)}$: 
For $(k_1, k_2, k_3, k_4, k_5) \in \chi_{H1}^{(5)} \chi_{A1}^{(5)}$, it follows that 
\begin{equation*}
|k_5|> 8^3 \max_{1 \le j \le 4} \{ |k_j| \} > 16 \cdot 8^{3} \text{sec}_{1\le j  \le 4} \{ |k_j| \}.
\end{equation*}
Thus, by Lemma~\ref{Le2} (i), we have 
\begin{equation*}
|\Phi_f^{(5)} \chi_{H1}^{(5)} \chi_{A1}^{(5)} \chi_{>L}^{(5)}| 
\gtrsim \max_{1 \le j \le 4} \{ |k_j| \} |k_5|^4 \chi_{H1}^{(5)} \chi_{A1}^{(5)} \chi_{>L}^{(5)},
\end{equation*}
which leads that $|L_{5, f}^{(5)} \chi_{>L}^{(5)}| \lesssim \langle k_5 \rangle^{-3} \chi_{H1}^{(5)} \chi_{>L}^{(5)}$ holds. \\
%%%%%%%%%%%%%%%%%%%%%%%%%%%%%%%%%%%%%%%%%
(IIIc) Estimate of $L_{6, f}^{(5)} \chi_{>L}^{(5)}$: 
By Lemma~\ref{Le2} (iii), it follows that 
\begin{equation*}
|\Phi_f^{(5)} \chi_{H1}^{(5)} \chi_{A2}^{(5)} \chi_{>L}^{(5)} | \gtrsim |k_{3,4}| |k_5|^4 \chi_{H1}^{(5)} \chi_{A2}^{(5)} \chi_{>L}^{(5)}, 
\end{equation*}
which leads that 
$|L_{6, f}^{(5)} \chi_{>L}^{(5)} | \lesssim \langle k_3 k_4 \rangle \langle k_{3,4} \rangle^{-1} \langle k_5 \rangle^{-4} \chi_{H1}^{(5)} \chi_{>L}^{(5)}$ holds. 

Therefore, we obtain \eqref{pwb0} for $(N,j) \in J_2$. \\
%%%%%%%%%%%%%%%%%%%%%%%%%%%%%%%%%%%%%%%%%%%%%%%%%%%%%%%%
\noindent
(IV) By Lemma~\ref{Le1}, we immediately obtain \eqref{pwb01}--\eqref{pwb03}.
\end{proof}

%%%%%%%%%%%%%%%%%%%%%%%%%%%%%%%%%%%%%%%%%%%%%%%%%%%%
%%%%%%%%%%%%%%%%%%%%%%%%%%%%%%%%%%%%%%%%%%%%%%%%%%%%
%%%%%%%%%%%%%%%%%%%%%%%%%%%%%%%%%%%%%%%%%%%%%%%%%%%%
Next, we give pointwise upper bounds of some multipliers $M_{j, \vp}^{(N)}$ defined in Proposition~\ref{prop_NF2}. 

%%%%%%%%%%%%%%%%%%%%%%%%%%%%%%%%%%%%%%%%%%%%%%%%%%%%%%%%%%%%%%
%%%%%%%%%%%%%%%%%%%Pointwise Estimates Part2 %%%%%%%%%%%%%%%%%%%%%%
%%%%%%%%%%%%%%%%%%%%%%%%%%%%%%%%%%%%%%%%%%%%%%%%%%%%%%%%%%%%%%
\begin{lem} \label{lem_pwb2}
Let $f \in L^2(\T)$ and $M_{j, f}^{(N)}$ be as in Proposition~\ref{prop_NF2}. 
Then, the following hold for $L \gg \max\{ 1, |\ga| E_1(f) \}$: \\
(I) When $(N,j) \in \{(5,4), (5,12), (5,14), (5,15), (5,16), (5,17) \}$, it follows that  
\begin{equation} \label{pwb11}
\langle k_{1,\dots ,5} \rangle^s  |M_{j,f}^{(N)}| \lesssim \big(  \max \Big\{ \frac{ |k_{1,2}| }{ \langle k_{1,2} \rangle}, \, \frac{|k_3 k_4|}{ \langle k_{3,4} \rangle  } \big\}
+ \big\langle \max_{1\le j \le 4} \{ |k_{j}| \}  \big\rangle^{5/8}  \big\langle \text{{\upshape sec}}_{1 \le j \le 4} \{ |k_j| \} \big\rangle^{5/8}  \big)  \langle k_5 \rangle^s
\end{equation} %% error/ overfull
and 
\begin{align} \label{pwb12}
\langle k_{1, \dots ,5}  \rangle^s  |M_{10,f}^{(5)} | \lesssim \min\{ \langle k_{1,2} \rangle^{-1}, \langle k_{3,4}  \rangle^{-1} \}
 \big\langle \max_{1 \le j \le 4} \{ |k_j|  \} \big\rangle^{5/4} \big\langle \text{{\upshape sec}}_{1 \le j \le 4} \{ |k_j| \}  \big\rangle^{5/4} 
\, \langle k_5  \rangle^s.
%& + \big\langle \max_{1\le j \le 4} \{ |k_{j}| \}  \big\rangle^{5/8}  \big\langle \text{{\upshape sec}}_{1 \le j \le 4} \{ |k_j| \} \big\rangle^{5/8}  \Big)\chi_{R4}^{(5)}  \chi_{H1}^{(5)}. 
\end{align}
(II) When $(N, j) \in  \{ (7,7), (7, 8), (7,9) \}$, it follows that 
\begin{equation} \label{pwb13}
\langle k_{1, \dots, 7}  \rangle^s |M_{j, f}^{(N)}| \lesssim \big\langle \max_{1 \le j \le 6} \{ |k_j| \}  \big\rangle^{5/6} 
\big\langle \text{{\upshape sec}}_{1 \le j \le 6} \{ |k_j| \}  \big\rangle^{5/6} \langle k_7 \rangle^s. 
\end{equation}
(III) It follows that 
\begin{align}
& \langle k_{1, \dots, 5}  \rangle^s |M_{18, f}^{(5)}| \lesssim \big\langle \max_{1 \le j \le 4} \{ |k_j| \}  \big\rangle^{7/6} 
\big\langle \text{{\upshape sec}}_{1 \le j \le 4} \{ |k_j| \}  \big\rangle^{7/6} \langle k_5 \rangle^{s-1/3}, \label{pwb41} \\
& \langle k_{1, \dots, 5}  \rangle^s |M_{22, f}^{(5)}| \lesssim \big\langle \max\{ |k_1|, |k_2| \}  \big\rangle^{s-1/3} 
\langle k_4 \rangle^{7/6} \langle k_5 \rangle^{7/6}, \label{pwb42} \\ 
& \langle k_{1, \dots, 7}  \rangle^s |M_{10, f}^{(7)}| \lesssim \big\langle \max_{1 \le j \le 6} \{ |k_j| \}  \big\rangle^{5/4}  
\langle k_7 \rangle^{s-1/4}, \label{pwb43} \\
& \langle k_{1, \dots, 7}  \rangle^s |M_{14, f}^{(7)}| \lesssim \langle k_6  \rangle^{5/4} \langle k_7  \rangle^{s-1/4}. \label{pwb44}
\end{align}
(IV) When $(N,j) \in \{ (5,19), (5,20), (7,11), (7, 12) \}$, it follows that 
\begin{equation} \label{pwb14}
| M_{j, f}^{(N)} | \lesssim L^2. 
\end{equation}
\end{lem}

%%%%%%%%%%%%%%%%%%%%%%%%%%%%%%%%%%%%%%%%
\begin{proof} 
Put $n_1= \max_{1 \le j \le 2N} \{ |k_j| \} $ and $n_2 = \text{sec}_{1 \le j \le 2N} \{ |k_j| \}  $.  \\
(I) Firstly, we show \eqref{pwb11} and \eqref{pwb12}. \\
%%%%%%%%%%%%%%%%%%%%%%%%%%%%%%%%%%%%%%%%%
(Ia) Estimate of $M_{4, f}^{(5)}$: Since $(k_1, k_2, k_3, k_4, k_5) \in \supp \chi_{H1}^{(5)} \chi_{R5}^{(5)}$ implies that \\
$k_{\max}= |k_5| \lesssim n_1^{5/4}$ and $n_1 \sim n_2$, we have 
$ \langle k_{1,2,3,4,5}  \rangle^s |M_{4, f}^{(5)} | \lesssim n_1^{5/8} n_2^{5/8} \langle k_5 \rangle^s $. \\
%%%%%%%%%%%%%%%%%%%%%%%%%%%%%%%%%%%%%%%%%%%%%%%%%%%%%%%%
(Ib) Estimate of $M_{12,f}^{(5)}$ and $M_{14,f}^{(4)}$: Since $(k_1, k_2, k_3, k_4, k_5) \in \supp \chi_{H1}^{(5)} \chi_{R4}^{(5)}$ leads that  
$|k_{1,2}| \sim |k_{3,4}| $, $|k_{3,4,5}| \sim |k_5| =k_{\max} \lesssim n_1^{5/4}$ and $n_1 \sim n_2$, by Lemma~\ref{Le1}, it follows that 
\begin{align*}
& \langle k_{1,2,3,4,5}  \rangle^s |M_{12,f}^{(5)}| 
 \lesssim \Big( \frac{|k_3 k_4|}{  \langle k_{1,2}  \rangle }+ \frac{|k_{3,4}|}{ \langle k_{1,2} \rangle } |k_5| \Big) 
 \langle k_5 \rangle^s \chi_{H1}^{(5)} \chi_{R4}^{(5)} 
\lesssim \Big( \frac{ |k_3 k_4| }{ \langle k_{3,4} \rangle } +n_1^{5/8} n_2^{5/8}  \Big) \langle k_5 \rangle^s ,  \\
& \langle k_{1,2,3,4,5}  \rangle^s  |M_{14, f}^{(5)} | \lesssim \Big( \frac{ |k_1 k_2|  }{ \langle k_{1,2} \rangle } + |k_5|  \Big) 
\langle k_5 \rangle^s \chi_{H1}^{(5)} \chi_{R4}^{(5)}
\lesssim \Big( \frac{ |k_1 k_2| }{ \langle k_{1,2} \rangle } +n_1^{5/8} n_2^{5/8}  \Big) \langle k_5 \rangle^s,  
\end{align*}
which implies that \eqref{pwb11} holds for $(N.j) \in \{ (5,12), (5,14) \}$. \\
%%%%%%%%%%%%%%%%%%%%%%%%%%%%%%%%%%%%%%%%%%%%%%%%%%%%%%%%
(Ic) Estimate of $M_{15, f}^{(5)}$: Since $(k_1, k_2, k_3, k_4, k_5) \in \supp \chi_{H1}^{(5)} (1- \chi_{A1}^{(5)}) $ implies that 
$|k_{3,4,5}| \sim |k_5| = k_{\max}$ and $n_1 \sim n_2$, by Lemma~\ref{Le1}, we have 
\begin{equation*}
\langle k_{1,2,3,4,5} \rangle^s |M_{15, f}^{(5)}| 
\lesssim \Big( \frac{ |k_1 k_2| }{ \langle k_{1,2} \rangle} + n_1 \Big) \langle k_5  \rangle^s \chi_{H1}^{(5)} (1- \chi_{A1}^{(5)}) 
\lesssim \Big( \frac{ |k_1 k_2| }{ \langle k_{1,2} \rangle } + n_1^{1/2} n_2^{1/2} \Big) \langle k_5 \rangle^s. 
\end{equation*}
%%%%%%%%%%%%%%%%%%%%%%%%%%%%%%%%%%%%%%%%%%%%%%%%%%%
We easily check that \eqref{pwb11} holds for $(N, j) \in \{ (5,16), (5,17) \}$. 
Therefore, we obtain \eqref{pwb11} for $(N,j) \in \{ (5,4), (5,12), (5, 14), (5,15), (5,16), (5,17) \}$.

%%%%%%%%%%%%%%%%%%%%%%%%%%%%%%%%%%%%%%%%%%%%%%%%%%%%%%%%%%%%%%%%%
Next, we prove \eqref{pwb12}. By Lemma~\ref{lem_key},  it follows that 
\begin{equation*}
M_{10, f}^{(5)}= - \frac{2}{5} i \ga^2 
\frac{ \chi_{NR2}^{(5)} \chi_{H1}^{(5)} \chi_{R4}^{(5)}  }{ k_{1,2} k_{3,4} k_{1,2,3,5} k_{1,2,4,5} k_{1,3,4,5} k_{2,3,4,5} } \,
\sum_{l=1}^{7} p_l(k_1, k_2, k_3, k_4) k_5^{7-l}. 
\end{equation*}
Thus,  by \eqref{es_q2}, we have 
\begin{equation} \label{pwes11}
|M_{10, f}^{(5)}| \lesssim \Big( \frac{ |k_{1,2,3,4}| }{ |k_{1,2}| |k_{3,4}| } |k_5|^2 + |k_5|+ 
\max \Big\{ \frac{|k_1 k_2|}{ |k_{1,2}| }, \frac{ |k_3 k_4| }{ |k_{3,4}| } \Big\}   \Big) \, \chi_{NR2}^{(5)} \chi_{H1}^{(5)} \chi_{R4}^{(5)}.
\end{equation}
Since $ (k_1, k_2, k_3, k_4, k_5) \in \supp \chi_{H1}^{(5)} \chi_{R4}^{(5)} $ implies that 
$|k_{1,2}| \sim |k_{3,4}|$, $k_{\max} =|k_5| \lesssim n_1^{5/4}$ and $n_1 \sim n_2$, we have 
\begin{align}
& \frac{ |k_{1,2,3,4}|}{ |k_{1,2}| |k_{3,4}|  } |k_5|^2 \, \chi_{NR2}^{(5)} \chi_{H1}^{(5)} \chi_{R4}^{(5)} 
\lesssim \min\{ \langle k_{1,2}  \rangle^{-1}, \langle k_{3,4} \rangle^{-1} \} n_1^{5/4} n_2^{5/4}  \chi_{H1}^{(5)} \chi_{R4}^{(5)}, \label{pwes12} \\ 
& |k_5| \chi_{NR2}^{(5)} \, \chi_{H1}^{(5)} \chi_{R4}^{(5)} \lesssim n_1^{5/8} n_2^{5/8} \, \chi_{H1}^{(5)} \chi_{R4}^{(5)}. \label{pwes13}
\end{align}
Collecting \eqref{pwes11}--\eqref{pwes13}, we obtain \eqref{pwb12}. 
%% $\langle k_{1,2,3,4,5} \rangle^s |M_{10, f}^{(5)}| \lesssim \min\{ \langle k_{1,2} \rangle^{-1}, \langle k_{3,4} \rangle^{-1} \} n_1^{5/4} n_2^{5/4} \langle k_5  \rangle^s$

\vspace{0.5em}

%%%%%%%%%%%%%%%%%%%%%%%%%%%%%%%%%%%%%%%%%%%%%%%%%%%%%%%%%%%%%
\noindent
(II) Secondly, we prove \eqref{pwb13}. By Lemma~\ref{Le2}, the definition of $Q_2^{(5)}$ and Lemma~\ref{lem_key}, we have 
\begin{equation} \label{pwes21}
\Big| \frac{ Q_2^{(5)} }{ \Phi_f^{(5)} } \chi_{>L}^{(5)} \chi_{H1}^{(5)} \Big| \lesssim \langle k_5 \rangle^{-2} \chi_{H1}^{(5)}, 
\hspace{0.5cm}  \Big| \frac{ Q_2^{(5)} }{ \Phi_0^{(5)} } \chi_{H1}^{(5)} \Big| \lesssim \langle k_5 \rangle^{-2} \chi_{H1}^{(5)}. 
\end{equation}
Let $\chi^{(7)}$be a characteristic function satisfying $\supp \chi^{(7)}  \subset \supp [\chi_{H1}^{(5)}]_{ext1}^{(7)} $. 
By \eqref{pwes21}, $\chi^{(7)} =[ \chi_{H1}^{(5)} ]_{ext1}^{(7)} \chi^{(7)} $ and Remark~\ref{rem_sym}, we have 
\begin{align}
& \Big| \Big[ \frac{Q_2^{(5)}}{ \Phi_f^{(5)} } \chi_{>L}^{(5)} \Big]_{ext1}^{(7)} \chi^{(7)} \Big| 
= \Big[  \Big| \frac{Q_2^{(5)}}{ \Phi_f^{(5)} } \chi_{>L}^{(5)} \chi_{H1}^{(5)}  \Big| \Big]_{ext1}^{(7)} \chi^{(7)} 
\lesssim [ \langle k_5 \rangle^{-2} \chi_{H1}^{(5)} ]_{ext1}^{(7)} \chi^{(7)}
\lesssim \langle k_{5,6,7} \rangle^{-2} \chi^{(7)}, \label{pwes221} \\
& \Big| \Big[ \frac{Q_2^{(5)}}{ \Phi_0^{(5)} } \Big]_{ext1}^{(7)} \chi^{(7)} \Big| 
\lesssim \langle k_{5,6,7} \rangle^{-2} \chi^{(7)}. \label{pwes22}
\end{align}
We notice that 
\begin{equation*}
\supp \chi_{H1}^{(7)} \subset \supp \chi_{NR(1,1)}^{(7)} \subset \supp [\chi_{H1}^{(5)}]_{ext1}^{(7)}, 
\hspace{0.5cm} \supp \chi_{NR(2,1)}^{(7)} \subset \supp [\chi_{H1}^{(5)}]_{ext1}^{(7)}. 
\end{equation*} 
%%%%%%%%%%%%%%%%%%%%%%%%%%%%%%%%%%%%%%%%%%%%%%%%%%%%%%%%%%
(IIa) Estimate of $M_{7, f}^{(7)}$: For $(k_1, \dots, k_7) \in \supp \chi_{H1}^{(7)} \chi_{R1}^{(7)} (1- \chi_{A4}^{(7)}) $, 
it follows that $|k_{1, \dots, 7}| \sim  |k_{5,6,7}| \sim |k_7|=k_{\max} \lesssim n_1^{5/3}$ and $n_1 \sim n_2$. 
Thus, by \eqref{pwes22}, we have 
\begin{equation*}
\langle k_{1, \dots, 7}  \rangle^{s} |M_{7, f}^{(7)}| \lesssim \langle k_7  \rangle^{s+1} \, \chi_{H1}^{(7)} \chi_{R1}^{(7)} (1- \chi_{A4}^{(7)}) \lesssim n_1^{5/6} n_2^{5/6} \langle k_7  \rangle^s.  
\end{equation*}
%%%%%%%%%%%%%%%%%%%%%%%%%%%%%%%%%%%%%%%%%%%%%%%%
(IIb) Estimate of $M_{8, f}^{(7)}$ and $M_{9, f}^{(7)}$: 
Since $(k_1 , \dots , k_7 ) \in \supp \chi_{H1}^{(7)} \chi_{R5}^{(7)}$ implies that 
$| k_{1, \dots, 7} | \sim |k_{5,6,7}| \sim |k_7| =k_{\max} \lesssim n_1^{5/4} $ and $n_1 \sim n_2$, by \eqref{pwes22}, it follows that 
\begin{equation*}
\langle k_{1, \dots, 7}  \rangle^s |M_{8, f}^{(7)}| \lesssim \langle k_7 \rangle^{s+1} \, \chi_{H1}^{(7)} \chi_{R5}^{(7)} \lesssim n_1^{5/8} n_2^{5/8} \langle k_7 \rangle^s. 
\end{equation*}
By \eqref{pwes22}, we have
\begin{equation*}
\langle k_{1, \dots, 7} \rangle^s |M_{9, f}^{(7)}| \lesssim n_1 \langle k_7  \rangle^s \chi_{H1}^{(7)} (1-\chi_{A1}^{(7)})
 \lesssim n_1^{1/2} n_2^{1/2} \langle k_7 \rangle^s.  
\end{equation*}
Therefore, we obtain \eqref{pwb13}. 

\vspace{0.5em}

%%%%%%%%%%%%%%%%%%%%%%%%%%%%%%%%%%%%%%%%%%%%%%%%
\noindent
(III) Thirdly, we prove \eqref{pwb41}--\eqref{pwb44}. \\
%%%%%%%%%%%%%%%%%%%%%%%%%%%%%%%%%%%%%%
(IIIa) Estimate of $M_{18,f}^{(5)}$: Since $(k_1, k_2, k_3, k_4, k_5) \in \supp \chi_{NR(1,1)}^{(5)} (1-\chi_{H1}^{(5)}) (1-  \chi_{A1}^{(5)})$ leads 
$|k_{1, 2,3,4,5}| \sim |k_5|=k_{\max} \lesssim n_1$ and $n_1 \sim n_2$, we have 
\begin{equation*}
\langle k_{1,2,3,4,5} \rangle^s |M_{18, f}^{(5)}| \lesssim \langle k_5 \rangle^{s+2} \chi_{NR(1,1)}^{(5)} (1-\chi_{H1}^{(5)}) (1-\chi_{A1}^{(5)}) 
\lesssim n_1^{7/6} n_{2}^{7/6} \langle k_5  \rangle^{s-1/3}.
\end{equation*}
%%%%%%%%%%%%%%%%%%%%%%%%%%%%%%%%%%%%%%%
(IIIb) Estimate of $M_{22, f}^{(5)}$: By symmetry, we may assume that $|k_1| \le |k_2|$ holds. 
Then, for $(k_1, k_2, k_3, k_4, k_5) \in \supp \chi_{NR(2,1)}^{(5)} (1- \chi_{A3}^{(5)})$, it follows that 
$|k_{1,2,3,4,5}| \sim |k_{3,4,5}| \lesssim |k_2|$ and $|k_4| \sim |k_5| \sim k_{\max}$. Thus, we have 
\begin{equation*}
\langle k_{1,2,3,4,5} \rangle^s |M_{22, f}^{(5)}| \lesssim \langle k_2 \rangle^{s} \langle k_5 \rangle^2 \chi_{NR(2,1)}^{(5)} (1- \chi_{A3}^{(5)}) 
\lesssim \langle k_2 \rangle^{s-1/3} \langle k_4 \rangle^{7/6} \langle k_5 \rangle^{7/6}. 
\end{equation*}
%%%%%%%%%%%%%%%%%%%%%%%%%%%%%%%%%%%%%%%%%%%%%%%%%%
(IIIc) Estimate of $M_{10, f}^{(7)}$: Since $(k_1, \dots, k_7) \in \supp \chi_{NR(1,1)}^{(7)} (1-\chi_{H1}^{(7)}) $ implies 
$|k_{1,\dots, 7}| \sim |k_{5,6,7}| \sim |k_7| =k_{\max} \lesssim n_1$, by \eqref{pwes22}, we have 
\begin{equation*}
\langle k_{1, \dots, 7}  \rangle^s |M_{10, f}^{(7)}| \lesssim \langle k_7 \rangle^{s+1} \chi_{NR(1,1)}^{(7)} (1-\chi_{H1}^{(7)}) 
\lesssim n_1^{5/4} \langle k_7 \rangle^{s-1/4}. 
\end{equation*}
%%%%%%%%%%%%%%%%%%%%%%%%%%%%%%%%%%%%%%%%%%%%%
(IIId) Estimate of $M_{14, f}^{(7)}$: Since $(k_1, \dots, k_7) \in \supp \chi_{NR(2,1)}^{(7)} $ implies \\
$|k_{1,\dots, 7}| \sim |k_{5,6,7}| \lesssim |k_6| \sim |k_7| \sim k_{\max}$, by \eqref{pwes221}, we have 
\begin{equation*}
\langle k_{1, \dots, 7}  \rangle^s |M_{14, f}^{(7)}| \lesssim \langle k_{5,6,7} \rangle^{s-1} \langle k_7 \rangle^{2} \chi_{NR(2,1)}^{(7)} 
\lesssim \langle k_6 \rangle^{5/4} \langle k_7 \rangle^{s-1/4}. 
\end{equation*}
%%%%%%%%%%%%%%%%%%%%%%%%%%%%%%%%%%%%%%%%%%%%
\noindent
(IV) Finally, we prove \eqref{pwb14}. \\
(IVa) Estimate of $M_{20,f}^{(5)}$ and $M_{12, f}^{(7)}$: Let $(N, j) \in \{ (5,20), (7,12) \}$. 
$(k_1, \dots, k_N) \in \supp \chi_{NR(1,1)}^{(N)} [\chi_{ \le L}^{(N-2)}]_{ext1}^{(N)}$ leads that 
$k_{\max} \le 4L/3$. Thus, we have $|M_{j, f}^{(N)}| \lesssim L^2$. \\
%%%%%%%%%%%%%%%%%%%%%%%%%%%%%%%%%%%%%%%%%%%%%%%%%%%%%%%%%%
(IVb) Estimate of $M_{19, f}^{(5)}$: 
By Lemma~\ref{Le1}, $\chi_{NR(1,1)}^{(5)}= [\chi_{H1}^{(3)} ]_{ext1}^{(5)} \chi_{NR(1,1)}^{(5)}$ and Remark ~\ref{rem_sym}, 
we have 
\begin{align*}
& \Big| \Big[ \Big( \frac{Q^{(3)} }{ \Phi_f^{(3)} }- \frac{Q^{(3)} }{ \Phi_0^{(3)} } \Big) \, \chi_{NR1}^{(3)} \chi_{>L}^{(3)}  \Big]_{ext1}^{(5)}  \chi_{NR(1,1)}^{(5)} \Big| \\
& \hspace{0.5cm} = \Big[ \Big| \Big( \frac{Q^{(3)} }{ \Phi_f^{(3)} }- \frac{Q^{(3)} }{ \Phi_0^{(3)} } \Big) \, \chi_{NR1}^{(3)} \chi_{>L}^{(3)} \chi_{H1}^{(3)} \Big| \Big]_{ext1}^{(5)} \chi_{NR(1,1)}^{(5)} \\
& \hspace{0.5cm} \lesssim \big[ |\ga| E_1(f) \, \langle k_3 \rangle^{-3} \chi_{H1}^{(3)} \big]_{ext1}^{(5)} \chi_{NR(1,1)}^{(5)} 
= |\ga| E_1(f) \, \langle k_{3,4,5} \rangle^{-3} \chi_{NR(1,1)}^{(5)}. 
\end{align*}
Thus, by $|\ga| E_1(f) \lesssim L$, we get $|M_{19,f}^{(5)}| \lesssim L$. \\
%%%%%%%%%%%%%%%%%%%%%%%%%%%%%%%%%%%%%%%%%%%%%%%%%%%%%%%%%%%%
(IVc) Estimate of $M_{11,f}^{(7)}$: By Lemma~\ref{Le2}, the definition of $Q_2^{(5)}$ and Lemma~\ref{lem_key}, it follows that 
\begin{equation*}
\Big|\Big( \frac{Q_2^{(5)}}{ \Phi_f^{(5)} } - \frac{Q_2^{(5)}}{\Phi_0^{(5)}}  \Big) \chi_{>L}^{(5)} \chi_{H1}^{(5)} \Big|
\lesssim |\ga| E_1(f) \, \langle k_5 \rangle^{-3} \chi_{H1}^{(5)}. 
\end{equation*}
Thus, by  $\chi_{NR(1,1)}^{(7)}= [\chi_{H1}^{(5)} ]_{ext1}^{(7)} \chi_{NR(1,1)}^{(7)} $ and Remark~\ref{rem_sym}, we have 
\begin{equation*}
\Big| \Big[ \Big( \frac{Q_2^{(5)}}{ \Phi_f^{(5)} } - \frac{Q_2^{(5)}}{\Phi_0^{(5)}}  \Big) \chi_{>L}^{(5)} \Big]_{ext1}^{(7)} \chi_{NR(1,1)}^{(7)} \Big|
\lesssim |\ga| E_1(f) \, \langle k_{5,6,7} \rangle^{-3} \chi_{NR(1,1)}^{(7)}. 
\end{equation*}
Hence, by $|\ga| E_1(f) \lesssim L $, we get $|M_{11,f}^{(7)}| \lesssim L$. 

Therefore, we obtain \eqref{pwb14}.
\end{proof}

%%%%%%%%%%%%%%%%%%%%%%%%%%%%%%%%%%%%%%%%%%%%
%%%%%%%%%%%%%%%   Remark %%%%%%%%%%%%%%%%%%%%%
%%%%%%%%%%%%%%%%%%%%%%%%%%%%%%%%%%%%%%%%%%%%
In the following remark, we state a property of multipliers $L_{j, \vp}^{(N)}$ and $M_{j, \vp}^{(N)}$ which is needed to 
prove main estimates stated as in Section 7.
%%%%%%%%%%%%%%%%%%%%%%%%%%%%%%%%%%%%%%%%%
\begin{rem} \label{rem_pwb3}
For $f, g \in L^2(\T)$ and $L \gg \max \{ |\ga| E_1(f), |\ga| E_1(g)  \}$, 
by Lemmas~\ref{Le1}, \ref{Le2}, \ref{Le3} and \ref{Le4} and the definition of $\{  L_{j, f}^{(N)} \}$ and $\{ M_{j, f}^{(N)} \}$ as in Proposition~\ref{prop_NF2}, 
each $L_{j, f}^{(N)}$ satisfies 
\begin{equation} \label{pwb21}
|(L_{j, f}^{(N)}-L_{j, g}^{(N)} ) \chi_{>L}^{(N)} | \lesssim |\ga| |E_1(f) -E_1(g)| (|L_{j, f}^{(N)} \chi_{>L}^{(N)} |+ |L_{j, g}^{(N)} \chi_{>L}^{(N)} | ). 
\end{equation}
and each $M_{j, f}^{(N)}$ satisfies 
\begin{align} \label{pwb22}
|M_{j, f}^{(N)} - M_{j, g}^{(N)}| \lesssim |\ga| |E_1(f)- E_1(g)| |M_{j, f}^{(N)}|. 
\end{align}
In particular, for $(N, j) \in \{ (5,1), (5,9), (5,11), (5,13), (7, 6) \}$, we have $M_{j, f}^{(N)}- M_{j, g}^{(N)}=0$.  
\end{rem}

%%%%%%%%%%%%%%%%%%%%%%%%%%%%%%%%%%%%%%%%%%%%%%%%
%%%%%%%%%%%%%%%%%%%%%%%%%%%%%%%%%%%%%%%%%%%%%%%%
%%%%%%%%%%%%%%%% Main Estimates %%%%%%%%%%%%%%%%%%%
%%%%%%%%%%%%%%%%%%%%%%%%%%%%%%%%%%%%%%%%%%%%%%%%
%%%%%%%%%%%%%%%%%%%%%%%%%%%%%%%%%%%%%%%%%%%%%%%%
\section{main estimates}

The main estimates of the proof of Theorem~\ref{thm_LWP} are as below.
 
%%%%%%%%%%%%%%%%%%%%%%%%%%%%%%%%%%%%%%%%%%%
%%%%%%%%%%%%%%%%%%%%%%%%%%%%%%%%%%%%%%%%%%%
\begin{prop} \label{prop_main1}
Let $s \ge 3/2$, $f, g \in L^2(\T)$ and $F_{f, L}$ and $G_{f, L}$ be as in Proposition~\ref{prop_NF2}. 
Then, for $v, w \in C([-T, T]: H^s(\T))$ and $L \gg \max\{ 1, |\ga| E_1(f), | \ga| E_1(g) \}$, we have 
\begin{align}
& \big\| F_{f, L} (\ha{v}) - F_{f, L} (\ha{w}) \big\|_{L_T^{\infty} l_s^2} 
\lesssim  L^{-1} (1+ \| v \|_{L_T^{\infty} H^s}+ \| w \|_{L_T^{\infty} H^s} )^6 \| v- w \|_{L_T^{\infty} H^s}, \label{mes11} \\
& \big\| G_{f, L} (\ha{v}) - G_{f, L} (\ha{w}) \|_{L_T^{\infty} l_s^2} 
\lesssim L^3 (1+ \| v \|_{L_T^{\infty} H^s}+ \| w \|_{L_T^{\infty} H^s} )^{10} \| v- w \|_{L_T^{\infty} H^s}. \label{mes12} 
\end{align}
Moreover, it follows that 
\begin{align}
& \big\| F_{f, L} (\ha{v}) - F_{g, L} (\ha{v}) \big\|_{L_T^{\infty} l_s^2} \le C_*,  \label{mes13} \\
& \big\| G_{f, L} (\ha{v})- G_{g, L} (\ha{v}) \big\|_{L_T^{\infty} l_s^2} \le C_* \label{mes14}
\end{align}
where $C_*=C_*(v ,s,|E_1(f)-E_1(g)|,|\ga|,T,L) \ge 0$, $C_*\to 0$ when $|E_1(f)-E_1(g)|\to 0$ and 
$C_*=0$ when $E_1(f)=E_1(g)$. 
\end{prop}
%%%%%%%%%%%%%%%%%%%%%%%%%%%%%%%%%%%%%%%%%%%%%%%%%
As a corollary of Proposition~\ref{prop_main1}, we obtain the following estimates. 
%%%%%%%%%%%%%%%%%%%%%%%%%%%%%%%%%%%%%%%%%%%
\begin{cor} \label{cor_mainest}
Let $s\ge 3/2$. Then, there exists a constant $C \ge 1$ such that the following estimates hold 
for $T>0$, $\vp_1 , \vp_2 \in H^s(\T)$, $L \gg \max\{ 1, | \ga| E_1(\vp_1), |\ga| E_1(\vp_2) \}$ 
and any solution $u_1\in C([-T,T]:H^s(\T))$ (resp. $u_2\in C([-T,T]:H^s(\T))$) to 
\eqref{5mKdV3} with initial data $\vp_1$ (resp. $\vp_2 $):
\EQS{
\| u_1 \|_{L_T^{\infty} H^s}   &\leq  \| \vp_1 \|_{H^s} + C L^{-1} (1+\|  u_1 \|_{L_T^{\infty} H^s} )^6 \| u_1 \|_{L_T^{\infty} H^s} \notag\\
& + C T L^3 (1+\| u_1 \|_{L_T^{\infty} H^s })^{10}  \| u_1 \|_{L_T^{\infty} H^s },  \label{mes31} \\
\| u_1- u_2 \|_{L_T^{\infty} H^s }  &\le \| \vp_1-\vp_2 \|_{H^s} + (1+T) C_{*} \notag\\
& + CL^{-1} (1+\| u_1 \|_{L_T^{\infty} H^s } + \| u_2 \|_{L_T^{\infty} H^s } )^{6} (\| u_1 - u_2\|_{L_T^{\infty} H^s }+C_{*}) \notag \\
& + CT L^{3} (1+\| u_1 \|_{L_T^{\infty} H^s } + \| u_2 \|_{L_T^{\infty} H^s } )^{10}( \| u_1 - u_2\|_{L_T^{\infty} H^s }+C_{*}) \label{mes32}
}
where a constant $C_{*}$ is as in Proposition~\ref{prop_main1}. 
% where $C_{*}=C_{*}(u_1, s , |E_1(\vp_1) - E_1(\vp_2)|, |\ga| , T, L) \ge 0$, $C_{*} \to 0$ when $|E_1(\vp_1)-E_1(\vp_2)| \to 0$ and 
% $C_{*}=0$ when $E_1(\vp_1)  =E_1(\vp_2)$. 
\end{cor}
%%%%%%%%%%%%%%%%%%%%%%%%%%%%%%%%%%%%%%%%%%%%%
\begin{proof}[Proof of Corollary \ref{cor_mainest}]
By proposition~\ref{prop_NF2}, 
$\ha{v}_j(t,k)=e^{-t \phi_{\vp_j}(k)} \ha{u}_j (t,k)$ satisfies (\ref{NF21}) with initial data $\ha{\vp}_j$ 
for each $k \in \Z$ and any $t \in [-T, T]$. Thus, it follows that 
\EQQ{
& \Big[ \ha{v}_1(t',k )+ F_{\vp_1,L} (\ha{v}_1) (t',k)  \Big]_0^t= \int_0^t G_{\vp_1, L} (\ha{v}_1) (t',k) \, dt', \\
& \Big[(\ha{v}_1-\ha{v}_2)(t',k)+ \big(F_{\vp_1,L}(\ha{v}_1 ) - F_{\vp_2,L}(\ha{v}_2  ) \big) (t',k)  \Big]_0^t \\
& \hspace{2.5cm} =\int_0^t \big( G_{\vp_1,L}(\ha{v}_1)-G_{\vp_2,L}(\ha{v}_2)  \big) (t',k)\, dt'
}
for each $k \in \Z$ and any $t \in [-T, T]$, which leads that 
\EQS{
& \| v_1 \|_{L^{\infty}_T H^s} \leq \| \vp_1 \|_{H^s} + 2\| F_{\vp_1, L} (\ha{v}_1) \|_{L_T^{\infty} l_s^2}
+ T \| G_{\vp_1,L}(\ha{v}_1) \|_{L_T^{\infty} l_s^2 },  \label{6es1}\\
& \|v_1-v_2\|_{L^\infty_T H^s} \leq \|\vp_1-\vp_2\|_{H^s} \notag \\ 
& \hspace{0.3cm} +2 \big( \|F_{\vp_1,L}(\ha{v}_1)-F_{\vp_2,L}( \ha{v}_1)\|_{L_T^{\infty} l_s^2} 
+ \| F_{\vp_2, L} (\ha{v}_1) - F_{\vp_2, L} (\ha{v}_2) \|_{L_T^{\infty} l_s^2} \big) \notag \\ 
& \hspace{0.3cm} +T \big( \|G_{\vp_1,L} (\ha{v}_1)-G_{\vp_2,L}(\ha{v}_1) \|_{L_T^{\infty} l_s^2} 
+ \|  G_{\vp_2, L} (\ha{v}_1)- G_{\vp_2, L} (\ha{v}_2)  \|_{L_T^{\infty} l_s^2}  \big). \label{6es2}
}
Applying \eqref{mes11}--\eqref{mes12} to \eqref{6es1} and \eqref{mes11}--\eqref{mes14} to \eqref{6es2},  
we find a constant $C \ge 1$ 
such that 
\EQS{
& \| v_1 \|_{L^\infty_T H^s}  \le  \| \vp_1 \|_{H^s}
+ C L^{-1} (1 +\| v_1\|_{L_T^{\infty} H^s } )^6 \| v_1\|_{L_T^{\infty} H^s } \notag \\
& \hspace{2cm} +C T L^3(1+\| v_1\|_{L_T^{\infty }H^s } )^{10}\| v_1 \|_{L_T^{\infty} H^s}, \label{6es3} \\
& \| v_1- v_2 \|_{L_T^{\infty} H^s}   \le \| \vp_1- \vp_2 \|_{H^s} +(1+T) C_{*} \notag \\
& \hspace{0.5cm} 
+ CL^{-1} (1+ \| v_1\|_{L_T^{\infty} H^s} + \| v_2\|_{L_T^{\infty} H^s})^6 \| v_1 -v_2 \|_{L_T^{\infty} H^s} \notag \\
& \hspace{0.5cm}
+ CT L^3  (1+ \| v_1\|_{L_T^{\infty} H^s} + \| v_2\|_{L_T^{\infty} H^s})^{10} \| v_1 -v_2 \|_{L_T^{\infty} H^s} . \label{6es4}
}
By $\| u_1 \|_{L_T^{\infty} H^s}= \|  v_1 \|_{L_T^{\infty} H^s }$ and \eqref{6es3}, we have \eqref{mes31}. 
We notice that it follows that 
\begin{equation} \label{6es5}
\| u_1-u_2 \|_{L_T^{\infty} H^s} \le \| v_1 -v_2 \|_{L_T^{\infty} H^s}+C_{*}, \hspace{0.5cm} 
\| v_1-v_2 \|_{L_T^{\infty}H^s} \le \| u_1- u_2  \|_{L_T^{\infty}H^s} + C_{*} 
\end{equation}
by uniform $l_s^2$-continuity of $\ha{u}_1$ and $\ha{v}_1$ and the dominated convergence theorem. 
By \eqref{6es4}, \eqref{6es5} and $\| u_j \|_{L_T^{\infty} H^s } = \|  v_j \|_{L_T^{\infty} H^s }$ with $j=1,2$, 
we obtain \eqref{mes32}. 
\end{proof}

%%%%%%%%%%%%%%%%%%%%%%%%%%%%%%%%%%%%%%%%%
%%%%%%%%%%%%%%%%%%%%%%%%%%%%%%%%%%%%%%%%%
Before we prove Proposition~\ref{prop_main1}, we prepare some lemmas. 
The proofs of them are based on the nonlinear estimates and the pointwise upper bounds of multipliers 
stated in Section 5 and 6. 

%%%%%%%%%%%%%%%%%%%%%%%%%%%%%%%%%%%%%%%%%%%%%%%%
%%%%%%%%%%%%%%%%%%%%%%%%%%%%%%%%%%%%%%%%%%%%%%%%
\begin{lem} \label{lem_nes0} 
Let $s \ge 3/2$, $f, g \in L^2(\T)$ and $L_{j, f}^{(N)}$ with $(N,j) \in J_1 \cup J_2 \cup J_3 \cup J_4$ be as in Proposition~\ref{prop_NF2}. 
Then, for any $\{ v_l \}_{l=1}^N \subset C([-T, T]: H^s(\T))$ and $L \gg \max \{1, |\ga| E_{1} (f) , | \ga| E_1(g) \}$, we have
\begin{align}
& \Big\| \sum_{k=k_{1, \dots, N}} |\ti{L}_{j, f}^{(N)} \chi_{>L}^{(N)} | \, \prod_{l=1}^N |\ha{v}_l(t, k_l)| \Big\|_{L_T^{\infty} l_2^s} 
\lesssim L^{-1} \prod_{l=1}^N \| v_l \|_{L_T^{\infty} H^s}, \label{ne01} \\
& \big\|  \La_f^{(N)} ( \ti{L}_{j, f}^{(N)} \chi_{>L}^{(N)}, \ha{v}_1) - \La_{g}^{(N)} ( \ti{L}_{j, g}^{(N)} \chi_{>L}^{(N)}, \ha{v}_1 )  \big\|_{L_T^{\infty}l_s^2} 
\le C_* \label{ne03}
\end{align}
where $C_*=C_*(v_1,s,|E_1(f)-E_1(g)|,|\ga|,T) \ge 0$, $C_*\to 0$ when $|E_1(f)-E_1(g)|\to 0$ and 
$C_{*} =0$ when $E_1(f)=E_1(g)$. 
\end{lem}
%%%%%%%%%%%%%%%%%%%%%%%%%%%%%%%%%%%%%%%%x
\begin{proof}
First, we prove \eqref{ne01}. It suffices to show 
\begin{equation} \label{ne02}
\Big\| \sum_{k=k_{1, \dots, N}} |L_{j, f}^{(N)} \chi_{>L}^{(N)}| \, \prod_{l=1}^N |\ha{v}_l(t, k_l)| \Big\|_{L_T^{\infty} l_s^2} \lesssim L^{-1} \prod_{l=1}^N \| v_l \|_{L_T^{\infty} H^s}. 
\end{equation} 
By Lemma~\ref{lem_pwb1}, it follows that 
\begin{equation*}  
| L_{j, f}^{(N)} \chi_{>L}^{(N)} | \le \langle k_{\max} \rangle^{-1} \chi_{>L}^{(N)} \lesssim L^{-1}
\end{equation*}
for $(N, j) \in J_1 \cup J_2 \cup J_3 \cup J_4 \setminus \{  (3,4) \}$. 
By \eqref{cnl4} and \eqref{pwb03}, we have % \eqref{ne02} with $(N, j)=(3,4)$. 
\begin{equation*}
 \Big\| \sum_{k=k_{1,2,3}} |L_{4, f}^{(3)} \chi_{>L}^{(3)} | \, \prod_{l=1}^3 |\ha{v}_l(t, k_l)| \Big\|_{L_T^{\infty} l_s^2} \lesssim L^{-3} \prod_{l=1}^3 \| v_l \|_{L_T^{\infty} H^s}.  
 \end{equation*}
Thus, we obtain \eqref{ne02}. 
Next, we prove (\ref{ne03}). A direct computation yields that  
\begin{align*}
& [ \La_{f}^{(N)} (\ti{L}_{j, f}^{(N)} \chi_{>L}^{(N)}, \ha{v}_1)- \La_{g}^{(N)} (\ti{L}_{j, g}^{(N)} \chi_{>L}^{(N)}, \ha{v}_1)] (t,k) \\
& = \La_{f}^{(N)} ( ( \ti{L}_{j, f}^{(N)}-\ti{L}_{j,g}^{(N)} ) \chi_{>L}^{(N)}, \ha{v}_1) (t,k)
+ [ \La_f^{(N)} (\ti{L}_{j,g}^{(N)} \chi_{>L}^{(N)}, \ha{v}_1) - \La_{g}^{(N)} (\ti{L}_{j, g}^{(N)} \chi_{>L}^{(N)}, \ha{v}_1)] (t,k) \\
& =: J_{1,1}^{(N)}(\ha{v}_1) (t,k) + J_{1,2}^{(N)} (\ha{v}_1) (t,k).
\end{align*}
By Lemma~\ref{lem_pwb1} and \eqref{pwb21}, it follows that 
$| ( \ti{L}_{j, f}^{(N)} - \ti{L}_{j,g}^{(N)}) \chi_{>L}^{(N)} | \lesssim |\ga| |E_1(f)-E_1(g)|$, 
which leads that 
\begin{align*}
\| J_{1,1}^{(N)} (\ha{v}_1) \|_{L_T^{\infty} l_s^2} \lesssim |\ga| |E_1(f) - E_1(g)| \| v_1 \|_{L_T^\infty H^s}^N \le C_*.
\end{align*}
By \eqref{ne01} and Lemma~\ref{lem_go}, we have $\| J_{1,2}^{(N)} (\ha{v}_1)  \|_{L_T^{\infty} l_s^2} \le C_*$.
Therefore, we obtain (\ref{ne03}). 
\end{proof}

%%%%%%%%%%%%%%%%%%%%%%%%%%%%%%%%%%%%%%%%%%%%%%%%%%%%%%%%%%%%%%%%%%%%%%
Next, we give multilinear estimates for all multipliers $M_{j, \vp}^{(N)}$ defined in 
Proposition~\ref{prop_NF2}. Now we put 
\begin{align*}
K_1=& \{(5,4), (5,12), (5,14), (5,15), (5, 16), (5, 17), (5,19), (5,20), \\ 
& \hspace{0.3cm} (7,6), (7,7), (7,8), (7,9), (7,11), (7,12) \}, \\
K_2=& \{ (5,2), (5,3), (5,5), (5,18), (5,22), (7,10), (7,14) \}, \\
K_3=& \{ (3,1), (5,10), (5,21), (5,23), (7,13), (7,15)  \}. \\
K_4=& \{ (5,7), (7,3), (7,4), (9,3) \}, \hspace{0.5cm} K_5=\{ (5,8), (7,5) \}, \\
K_6=&\{ (5,6), (7,2), (9,2) \}, \hspace{0.5cm} K_7=\{ (7,1), (9,1), (11,1)   \}.
\end{align*}
%%%%%%%%%%%%%%%%%%%%%%%%%%%%%%%%%%%%%%%%%%%%%%%%
\begin{lem} \label{lem_mle}
Let $s \ge 3/2$, $f, g \in L^2(\T)$ and $M_{j, \vp}^{(N)}$ with $(N, j) \in K_1 \cup K_2 \cup K_3$ as be in 
Proposition~\ref{prop_NF2}. Then, for $\{ v_l \}_{l=1}^N  \subset C([-T, T]: H^s(\T))$ and $L \gg \max\{1, |\ga| E_1(f), |\ga| E_1(g) \}$, 
we have  
\begin{align}
& \Big\| \sum_{k=k_{1, \cdots, N}} |\ti{M}_{j, f}^{(N)}| \, \prod_{l=1}^N |\ha{v}_l(t, k_l)| \Big\|_{L_T^{\infty} l_s^2 } 
\lesssim L^2 \prod_{l=1}^N \| v_l \|_{L_T^{\infty} H^s}, \label{mle1} \\
& \big\| \La_f^{(N)} (\ti{M}_{j, f}^{(N)}, \ha{v}_1) - \La_g^{(N)} (\ti{M}_{j, g}^{(N)}, \ha{v}_1) \big\|_{L_T^{\infty} l_s^2} \le C_{*} \label{mle2}
\end{align}
where $C_{*}=C_{*}(v_1, s, |E_1(f)-E_1(g) |, |\ga|, L, T) \ge 0$, $C_{*} \to 0$ when $|E_1(f) -E_1(g)| \to 0$ 
and $C_{*} =0$ when $E_1(f)=E_1(g)$. 
Moreover, it follows that 
\begin{align}
& \Big\| \sum_{k=k_{1, 2,3,4,5}} |\ti{M}_{1, f}^{5)} + \ti{M}_{9, f}^{(5)} | \, \prod_{l=1}^5 |\ha{v}_l(t, k_l)| \Big\|_{L_T^{\infty} l_s^2 } 
\lesssim \prod_{l=1}^5 \|v_l \|_{L_T^{\infty} H^s}, \label{mle3} \\
& \Big\| \sum_{k=k_{1, 2,3,4,5}} |\ti{M}_{11, f}^{(5)} +\ti{M}_{13, f}^{(5)} | \, \prod_{l=1}^5 |\ha{v}_l(t, k_l)| \Big\|_{L_T^{\infty} l_s^2 } 
\lesssim \prod_{l=1}^5 \| v_l \|_{L_T^{\infty} H^s} \label{mle4} 
\end{align}
and
\begin{align}
& \big\| \La_f^{(5)} (\ti{M}_{1,f}^{(5)}+ \ti{M}_{9, f}^{(5)}, \ha{v}_1)- \La_g^{(5)} (\ti{M}_{1,g}^{(5)}+ \ti{M}_{9, g}^{(5)}, \ha{v}_1 ) 
\big\|_{L_T^{\infty} l_s^2} \le C_{*}, \label{mle5} \\
& \big\| \La_f^{(5)} (\ti{M}_{11,f}^{(5)}+ \ti{M}_{13, g}^{(5)} , \ha{v}_1)- \La_g^{(5)} (\ti{M}_{11,g}^{(5)}+ \ti{M}_{13, g}^{(5)}, \ha{v}_1 ) 
\big\|_{L_T^{\infty} l_s^2} \le C_{*} \label{mle6} 
\end{align}
where a constant $C_{*}$ is defined above. 
% $C_{*}= C_{*} (v_1, s, |E_1(f)-E_1(g)|, |\ga|, T) \ge 0$, $C_{*} \to 0$ when $|E_1(f) -E_1(g)| \to 0$ and  
% $C_{*} =0$ when $E_1(f)=E_1(g)$. 
\end{lem}

%%%%%%%%%%%%%%%%%%%%%%%%%%%%%%%%%%%%%%%%%%%%%%%%%%%%%%%%%%%
\begin{proof}
Firstly, we show \eqref{mle1}. 
By Proposition~\ref{prop_res2}, we immediately have \eqref{mle1} for $(N, j)=(7,6)$. 
Thus, we only need to prove 
\begin{equation} \label{mle7}
\Big\| \sum_{k=k_{1, \dots, N}} |M_{j, f}^{(N)}| \, \prod_{l=1}^N |\ha{v}_l (t, k_l) | \Big\|_{L_T^{\infty} l_s^2 } 
\lesssim L^2 \prod_{l=1}^N \| v_l \|_{L_T^{\infty} H^s}
\end{equation}
for $(N, j) \in K_1 \cup K_2 \cup K_3 \setminus \{ (7,6) \}$. \\
%%%%%%%%%%%%%%%%%%%%%%%%%%%%%%%%%%%%%%%%%%
(I) By \eqref{pwb11}, \eqref{pwb13} and \eqref{pwb14}, we have \eqref{mle7} for $(N,j) \in K_1 \setminus \{ (7,6) \}$ and $s > 4/3$. \\
%%%%%%%%%%%%%%%%%%%%%%%%%%%%%%%%%%%%%%%%%
(II) By the definition of multipliers and Lemma~\ref{Le7}, we have \eqref{mle7} for $(N,j) \in \{(5,2) , (5,3), (5,5) \}$. 
Combining \eqref{pwb41}--\eqref{pwb44} and Lemma~\ref{Le7}, we get \eqref{mle7} for $(N, j) \in \{ (5,18), (5,22), (7,10), (7,14) \}$. 
Therefore, we obtain \eqref{mle7} for $(N, j) \in K_2$. \\
%%%%%%%%%%%%%%%%%%%%%%%%%%%%%%%%%%%%%%%%%%%%%%%%%%%%%%
(III) Next, we prove \eqref{mle7} for $(N, j) \in K_3$. 
Put $n_1= \max_{1 \le j \le N-1} \{ |k_j| \} $ and $n_2=\text{sec}_{1 \le j \le N-1} \{ |k_j| \}$. \\
%%%%%%%%%%%%%%%%%%%%%%%%%%%%%%%%%%%%%%%%%
(IIIa) Estimate of $M_{1,f}^{(3)}$:  By Lemma~\ref{lem_stres}, 
we have \eqref{mle7} with $(N,j)=(3,1)$. \\
%%%%%%%%%%%%%%%%%%%%%%%%%%%%%%%%%%%%%%%%%%%%%%
(IIIb) Estimate of $M_{10, f}^{(5)}$: By symmetry, we may assume that $|k_1| \le |k_2|$ and $|k_3| \le |k_4|$ hold. 
Then, we need to deal with the following cases. 
% By \eqref{pwb12}, we have % By \eqref{pwb12} in Lemma~\ref{lem_pwb2}
% \begin{equation} \label{mle11}
% |M_{10, f}^{(5)}| \lesssim n_1^{5/4} n_2^{5/4} \langle k_{1,2} \rangle^{-1} \, \chi_{H1}^{(5)} \chi_{R4}^{(5)}.  
% \end{equation}
% Since $(k_1, k_2, k_3, k_4, k_5) \in \supp \chi_{H1}^{(5)} \chi_{R4}^{(5)}$ leads that 
% $|k_{1,2}| \sim |k_{3,4}|$, $k_{\max} =|k_5|$ and $n_1 \sim n_2$, we need to deal with the following cases. 

First, we suppose that $n_1=|k_4|$ and $n_2=|k_2|$ hold. Then, by \eqref{pwb12}, it follows that 
\begin{equation*}
\langle k_{1,2,3,4,5} \rangle^{s} |M_{10, f}^{(5)}| \lesssim 
\langle k_{1,2} \rangle^{-1/2} \langle k_2 \rangle^{5/4} \langle k_{3,4} \rangle^{-1/2} \langle k_4 \rangle^{5/4} \langle k_5 \rangle^s.
\end{equation*}
Thus, by the Young inequality and the H\"{o}lder inequality, the left hand side of \eqref{mle7} is bounded by 
\begin{align*}
& \| \langle \cdot \rangle^{-1/2} (|\ha{v}_1|* \langle \cdot \rangle^{5/4} |\ha{v}_2| ) \|_{l^1} 
\|  \langle \cdot \rangle^{-1/2} (|\ha{v}_3| * \langle \cdot \rangle^{5/4} |\ha{v}_4|) \|_{l^1} \| \langle \cdot  \rangle^s |\ha{v}_5| \|_{l^2} \\
& \lesssim \| v_1 \|_{H^{1/2+}} \| v_2 \|_{H^{5/4+}} \| v_3 \|_{H^{1/2+}} \| v_4 \|_{H^{5/4+}} \| v_5 \|_{H^s},
\end{align*}
which implies that \eqref{mle7} holds for $s >5/4$. 
In a same manner as above, we have \eqref{mle7} when $n_1 =|k_2|$ and $n_2=|k_4|$. 

Next, we suppose that $n_1=|k_4|$ and $n_2=|k_3|$ hold. 
Then, by \eqref{pwb12}, we have 
\begin{align*}
\langle k_{1,2,3,4,5} \rangle^s |M_{10, f}^{(5)}| \lesssim 
\langle k_{3,4}  \rangle^{-1} \langle k_3 \rangle^{5/4} \langle k_4 \rangle^{5/4} \langle k_5 \rangle^s.
\end{align*}
Thus, by the Young inequality and the H\"{o}lder inequality, the left hand side of \eqref{mle7} is bounded by 
\begin{align*}
& \| \ha{v}_1| \|_{l^1} \| \ha{v}_2 \|_{l^1}  
\|  \langle \cdot \rangle^{-1} ( \langle \cdot \rangle^{5/4} |\ha{v}_3| * \langle \cdot \rangle^{5/4} |\ha{v}_4|) \|_{l^1} 
\| \langle \cdot  \rangle^s |\ha{v}_5| \|_{l^2} \\
& \lesssim \| v_1 \|_{H^{1/2+}} \| v_2 \|_{H^{1/2+}} \| v_3 \|_{H^{5/4+}} \| v_4 \|_{H^{5/4+}} \| v_5 \|_{H^s},
\end{align*}
which implies that \eqref{mle7} holds for $s>5/4$. 
In a similar manner as above, we have \eqref{mle7} when $n_1=|k_2|$ and $n_2=|k_1|$. \\
%%%%%%%%%%%%%%%%%%%%%%%%%%%%%%%%%%%%%%%%%%%%%%%%%%%%%%%%%%%%%%%%%%%%%%%%
(IIIc) Estimate of $M_{21, f}^{(5)}$: Since $(k_1, k_2, k_3, k_4, k_5) \in \supp \chi_{NR(1,2)}^{(5)}$ leads that $|k_{3,4,5}| \sim |k_5|$ and  
$6 \max_{j=2,3,4,5}\{ |k_j| \} <|k_1|=k_{\max}$, by Lemma~\ref{Le1}, we have 
\begin{equation} \label{mle12}
\langle k_{1,2,3,4,5} \rangle^s |M_{21,f}^{(5)}| 
\lesssim \langle k_1 \rangle^{s-1} \langle k_{2,3,4,5} \rangle^{-1} \langle k_{3,4,5} \rangle \langle k_5 \rangle^2  \chi_{NR(1,2)}^{(5)}.  
\end{equation}
First, we suppose that $16|k_2| \le  |k_{3,4,5}|$ holds. 
% Then, $(k_1, k_2, k_3, k_4, k_5) \in \supp \chi_{NR(1,2)}^{(5)}$ implies $|k_{2,3,4,5}| \sim |k_{3,4,5}| \sim |k_5|$. 
Then, by \eqref{mle12}, we have 
\begin{equation*}
\langle k_{1,2,3,4,5} \rangle^s | M_{21, f}^{(5)} | \lesssim \langle k_1 \rangle^{s-1} \langle k_5 \rangle^2 \chi_{NR(1,2)}^{(7)} 
\lesssim \langle k_1 \rangle^{s-1/4} \langle k_5 \rangle^{5/4}.
\end{equation*}
Thus, by  Lemma~\ref{Le7} with $i=2$, we get \eqref{mle7}. \\ 
Next, we suppose that $ |k_{3,4,5}| \le 16 |k_2| $ holds. 
%Then, $(k_1,k_2, k_3, k_4, k_5) \in \supp \chi_{NR(1,2)}^{(7)}$ implies that $|k_{5}| \sim |k_{3,4,5}| \lesssim |k_2|$. 
Then, by \eqref{mle12}, we have 
\begin{equation*}
\langle k_{1,2,3,4,5} \rangle^s |M_{21,f}^{(5)}| \lesssim 
\langle k_{1} \rangle^{s-1} \langle k_2 \rangle \langle k_5 \rangle^2 \chi_{NR(1,2)}^{(5)} 
\lesssim \langle k_1 \rangle^{s-1/3} \langle k_2 \rangle^{7/6} \langle k_5 \rangle^{7/6}.
\end{equation*}
Thus, by Lemma~\ref{Le7} with $i=3$, we get \eqref{mle1}. \\
%%%%%%%%%%%%%%%%%%%%%%%%%%%%%%%%%%%%%%%%%%%%%%%%%%%%%%%%%%
(IIId) Estimate of $M_{23,f}^{(5)}$: Since $(k_1, k_2, k_3, k_4,  k_5) \in \supp \chi_{NR(2,2)}^{(5)}$ leads that 
$|k_4| \sim |k_5|$ and $8 \max\{ |k_2|, |k_{3,4,5}| \} < |k_1|$, by Lemma~\ref{Le1}, we have 
\begin{equation*} 
\langle k_{1,2,3,4,5} \rangle^s  |M_{23,f}^{(5)}| 
\lesssim \langle k_1 \rangle^{s-1} \langle k_{2,3,4,5} \rangle^{-1} \langle k_{3,4,5} \rangle \langle k_5 \rangle^2 \chi_{NR(2,2)}^{(5)}
\lesssim \langle k_1 \rangle^s \langle k_{2,3,4,5}  \rangle^{-1} \langle k_4 \rangle \langle k_5 \rangle.
\end{equation*}
Thus, by the H\"{o}lder and Young inequalities, we have \eqref{mle7} for $s>1$. \\
% the left \eqref{mle7} is bounded by $\| v_1 \|_{H^s} \prod_{l=2}^5 \| v_l \|_{H^{1+}}$. \\ 
%%%%%%%%%%%%%%%%%%%%%%%%%%%%%%%%%%%%%%%%%%%%%%%%%%%%%%%%%%
(IIIe) Estimate of $M_{13,f}^{(7)}$ and $M_{15, f}^{(7)}$: Since $Q_2^{(5)} \chi_{>L}^{(5)} / \Phi_f^{(5)} $ is symmetric with \\
$(k_1, k_2, k_3, k_4)$ and 
$\chi_{NR(1,2)}^{(7)}= \chi_{H1}^{(5)} (k_{5,6,7}, k_2, k_3, k_4, k_1) \chi_{NR(j, 2)}^{(7)}$ for $j=1,2$, it follows that 
\begin{align*}
& \Big[ \Big[ 5 \, \frac{Q_2^{(5)}}{ \Phi_f^{(5)}} \chi_{>L}^{(5)} \chi_{H1}^{(5)} \Big]_{sym}^{(5)}  \Big]_{ext1}^{(7)} \chi_{NR(j ,2)}^{(7)} \notag \\
& = \Big[ \Big[ 5 \, \frac{Q_2^{(5)}}{ \Phi_f^{(5)}} \chi_{>L}^{(5)} \chi_{H1}^{(5)} \Big]_{sym}^{(5)}  \Big]_{ext1}^{(7)} 
\chi_{H1}^{(5)} (k_{5,6,7}, k_2, k_3, k_4, k_1) \chi_{NR(j,2)}^{(7)} \notag \\
& =
\Big( \frac{Q_2^{(5)}}{ \Phi_f^{(5)}} \chi_{>L}^{(5)} \chi_{H1}^{(5)} \Big) (k_{5,6,7}, k_2, k_3, k_4, k_1) \chi_{NR(j,2)}^{(7)}, 
\end{align*}
which leads that 
\begin{align*}
& M_{13,f}^{(7)}=- 12\Big( \frac{Q_2^{(5)}}{ \Phi_f^{(5)} } \chi_{>L}^{(5)} \chi_{H1}^{(5)} \Big)(k_{5,6,7}, k_2, k_3, k_4, k_1) 
(Q^{(3)} \chi_{NR1}^{(3)}) (k_5, k_6, k_7) \chi_{NR(1,2)}^{(7)}, \\
& M_{15,f}^{(7)}=- 12\Big( \frac{Q_2^{(5)}}{ \Phi_f^{(5)} } \chi_{>L}^{(5)} \chi_{H1}^{(5)} \Big)(k_{5,6,7}, k_2, k_3, k_4, k_1) 
(Q^{(3)} \chi_{NR1}^{(3)}) (k_5, k_6, k_7) \chi_{NR(2,2)}^{(7)}. 
\end{align*}
By Lemma~\ref{Le2}, the definition of $Q_2^{(5)}$ and Lemma~\ref{lem_key}, 
\begin{equation} \label{mle15}
\Big| \Big( \frac{Q_2^{(5)}}{ \Phi_{f}^{(5)} } \chi_{>L}^{(5)} \chi_{H1}^{(5)} \Big) (k_{5,6,7}, k_2, k_3, k_4, k_1) \chi_{NR(j ,2)}^{(7)}  \Big| 
\lesssim \langle k_1 \rangle^{-2} \chi_{NR(j,2)}^{(7)}
\end{equation}
for $j=1,2$. 
Since $(k_1, \dots, k_7) \in \supp \chi_{NR(1,2)}^{(7)}$ implies that $|k_{5,6,7}| \sim |k_7|$ and $8^3 |k_{5,6,7}| <|k_1|= k_{\max} $, 
by \eqref{mle15}, we have 
\begin{equation*}
\langle k_{1,2,3,4,5,6,7}  \rangle^{s} |M_{13, f}^{(7)}| \lesssim \langle k_1  \rangle^{s-2} \langle k_7 \rangle^{3} \chi_{NR(1,2)}^{(7)} 
\lesssim \langle k_1 \rangle^{s-1/4} \langle k_7 \rangle^{5/4}.  
\end{equation*}
Thus, by Lemma~\ref{Le7} with $i=2$, we get \eqref{mle1}. 
%%%%%%%%%%%%%%%%%%%%%%%%%%%%%%%%%%%%%%%%%%%%%%%%%%%%%%%%%%%
Moreover, for $(k_1, \dots, k_7) \in \supp \chi_{NR(2,2)}^{(7)}$ it follows that 
$8^3 |k_{5,6,7}| <|k_1|\sim |k_{1,2,3,4,5,6, 7}|$ and $|k_6| \sim |k_7|$. Thus, by \eqref{mle15}, 
\begin{equation*}
\langle k_{1,2,3,4,5,6,7} \rangle^{s} |M_{15,f}^{(5)}|
 \lesssim \langle k_1  \rangle^{s-2} \langle  k_{5,6,7} \rangle \langle k_7 \rangle^{2} \chi_{NR(2,2)}^{(7)} 
 \lesssim \langle k_1 \rangle^{s-1} \langle k_6 \rangle \langle k_7 \rangle,
\end{equation*}
which implies that \eqref{mle7} holds. 

%%%%%%%%%%%%%%%%%%%%%%%%%%%%%%%%%%%%%%%%%%%%%%%%
Therefore, we obtain \eqref{mle7} for $(N,j) \in K_1 \cup K_2 \cup K_3 \setminus \{ (7,6) \}$. 

%%%%%%%%%%%%%%%%%%%%%%%%%%%%%%%%%%%%%%%%%%%%%%%%
Secondly, we prove \eqref{mle2}. By a direct computation, it follows that  
\begin{align*}
& [\La_f^{(N)} (\ti{M}_{j, f}^{(N)}, \ha{v}_1)- \La_g^{(N)} (\ti{M}_{j, g}^{(N)}, \ha{v}_1)](t,k) \\
& = \La_f^{(N)} (\ti{M}_{j, f}^{(N)}- \ti{M}_{j,g}^{(N)}, \ha{v}_1 ) (t, k)
+ [ \La_f^{(N)} (\ti{M}_{j, g}^{(N)}, \ha{v}_1)- \La_g^{(N)} (\ti{M}_{j, g}^{(N)}, \ha{v}_1) ](t,k) \\
& =: J_{2,1}^{(N)} (\ha{v}_1)(t, k) + J_{2,2}^{(N)} (\ha{v}_1)(t,k). 
\end{align*}
By \eqref{pwb22}, \eqref{mle7} and $\ti{M}_{6, f}^{(7)}-\ti{M}_{6,g}^{(7)}=0$, we have
\begin{equation*}
\| J_{2,1}^{(N)} (\ha{v}_1) \|_{L_T^{\infty} l_s^2} \lesssim |\ga| |E_1(f) - E_1(g)| \, \| v_1 \|_{L_T^{\infty} H^s}^N \le C_{*}. 
\end{equation*}
By \eqref{mle1} and Lemma~\ref{lem_go}, we have 
$\| J_{2,2}^{(N)} (\ha{v}_1) \|_{L_T^{\infty} l_s^2} \le C_{*}$. 
Thus, we obtain \eqref{mle2}. 

%%%%%%%%%%%%%%%%%%%%%%%%%%%%%%%%%%%%%%%%%%%%%%%%%%%%%%%%%%%%%%%%%
Finally, we prove \eqref{mle3} and \eqref{mle4}. It suffices to show 
\begin{align}
& \Big\| \sum_{k=k_{1,2,3,4,5}} |M_{1, f}^{(5)} + M_{9, f}^{(5)}| \, \prod_{l=1}^5 |\ha{v}_l(t, k_l)| \Big\|_{L_T^{\infty} l_s^2} 
\lesssim \prod_{l=1}^5 \| v_l \|_{L_T^{\infty} H^s},  \label{mle17} \\
& \Big\| \sum_{k=k_{1,2,3,4,5}} |M_{11, f}^{(5)} + M_{13, f}^{(5)}| \, \prod_{l=1}^5 |\ha{v}_l(t, k_l)| \Big\|_{L_T^{\infty} l_s^2} 
\lesssim \prod_{l=1}^5 \| v_l \|_{L_T^{\infty} H^s}. \label{mle18}
\end{align}
We prove \eqref{mle17}. 
Since $(k_1,  k_2, k_3, k_4, k_5) \in \supp \chi_{H1}^{(5)} \chi_{R1}^{(5)} (1-\chi_{R2}^{(5)})$ leads that 
$|k_{1,2}|=|k_{3,4}| \ge 1$, $n_1 \sim n_2$ and $|k_5|=k_{\max}$, by Proposition~\ref{prop_res1}, we have   
\begin{equation*}
\langle k_{1,2,3,4,5} \rangle^{s} |M_{1,f}^{(5)}+ M_{9, f}^{(5)}| \lesssim 
\Big( \max \Big\{ \frac{|k_1k_2|}{ \langle k_{1,2} \rangle }, \frac{ |k_3 k_4| }{ \langle k_{3,4} \rangle } \Big\} 
+n_1^{1/2} n_2^{1/2}   \Big) \langle k_5 \rangle^s. 
\end{equation*}
Thus, by the Young inequality and the H\"{o}lder inequality, we have \eqref{mle17} for $s>1$. 
Next, we show \eqref{mle18}. 
We notice that $(k_1, k_2, k_3, k_4, k_5) \in \supp \chi_{H1}^{(5)} \chi_{R1}^{(5)} (1- \chi_{R2}^{(5)})$ implies that 
$|k_{1,2} | = |k_{3,4}| \ge 1$, $n_1 \sim n_2$ and  $|k_5|=k_{\max}$. 
Thus, by Proposition~\ref{prop_res3}, we have 
\begin{equation*}
\langle k_{1,2,3,4,5}  \rangle^s | M_{11,f}^{(5)}+ M_{13,f}^{(5)} | \lesssim \min\{ \langle k_{1,2}  \rangle^{-1}, \langle k_{3,4} \rangle^{-1} \} 
n_1 n_2 \langle k_5 \rangle^s.
\end{equation*} 
In a similar manner as the case (IIIb) in Lemma~\ref{lem_mle}, we get \eqref{mle18} for $s >1$. 

In a similar way to the proof of \eqref{mle2}, we obtain \eqref{mle5} and \eqref{mle6}. 
\end{proof}

%%%%%%%%%%%%%%%%%%%%%%%%%%%%%%%%%%%%%%%
%%%%%%%%%%%%%%%%%%%%%%%%%%%%%%%%%%%%%%%
%%%%%%%%%%%%%%%%%%%%%%%%%%%%%%%%%%%%%%%
\begin{lem} \label{lem_nl10}
Let $s \ge 3/2$, $f , g \in L^2(\T)$ and $M_{j, f}^{(N)}$ with 
$(N,j) \in K_4 \cup K_5 \cup K_6 \cup K_7$ be as in Proposition~\ref{prop_NF2}. 
Then, for $\{ v_l \}_{l=1}^N \subset C([-T, T]: H^s(\T)) $ and $L \gg \max\{ 1, |\ga| E_1(f), |\ga| E_1(g) \}$, we have
\begin{align}
& \big\| \La_{f}^{(N)} (\ti{M}_{j, f}^{(N)}, \ha{v}_1, \dots, \ha{v}_N) \big\|_{L_T^{\infty} l_s^2} 
\lesssim \prod_{l=1}^N \| v_l \|_{L_T^{\infty} H^s}, \label{nes11} \\
& \big\|  \La_{f}^{(N)} (\ti{M}_{j, f}^{(N)}, \ha{v}_1) - \La_g^{(N)} (\ti{M}_{j, g}^{(N)} , \ha{v}_1) \big\|_{L_T^{\infty} l_s^2} 
\le C_* \label{nes13} 
\end{align}
where $C_*=C_*(v_1, s ,|E_1(f)-E_1(g)|, |\ga|, T) \ge 0$, $C_*\to 0$ when $|E_1(f)-E_1(g)|\to 0$ and 
$C_{*} =0$ when $E_1(f)=E_1(g)$. 
\end{lem}

%%%%%%%%%%%%%%%%%%%%%%%%%%%%%%%%%%%%%%%%%%%%%
%%%%%%%%%%%%%%%%%%%%%%%%%%%%%%%%%%%%%%%%%%%%%
\begin{proof} 
Put 
\begin{align*}
& \mathcal{N}_{1,f}^{(3)} (\ha{w}_1, \ha{w}_2, \ha{w}_3) (t,k):= 
\La_f^{(3)} (-Q^{(3)} \chi_{NR1}^{(3)}, \ha{w}_1, \ha{w}_2, \ha{w}_3 )(t, k), \\
& \mathcal{N}_{2,f}^{(3)} (\ha{w}_1, \ha{w}_2, \ha{w}_3) (t,k):= 
\La_f^{(3)} \big( (- Q^{(3)} \chi_{NR1}^{(3)}) ( [ 3\chi_{H2,2}^{(3)} ]_{sym}^{(3)} + \chi_{H3}^{(3)}) , \ha{w}_1, \ha{w}_2 , \ha{w}_3 \big)(t, k), \\
& \mathcal{N}_{3,f}^{(3)} (\ha{w}_1, \ha{w}_2, \ha{w}_3) (t,k):= 
\La_f^{(3)} \big(Q^{(3)} [3 \chi_{R3}^{(3)}]_{sym}^{(3)}   , \ha{w}_1, \ha{w}_2 , \ha{w}_3 \big)(t, k), \\
& \mathcal{N}_{f}^{(5)} (\ha{w}_1, \ha{w}_2, \ha{w}_3, \ha{w}_4, \ha{w}_5) (t,k):= 
\La_f^{(5)} (-Q_1^{(5)}, \ha{w}_1, \ha{w}_2, \ha{w}_3, \ha{w}_4, \ha{w}_5 )(t, k)
\end{align*}
and
\begin{align*}
& \mathcal{N}_{i, f}^{(3)}(\ha{w}_1) (t, k):= \mathcal{N}_{i, f}^{(3)} (\ha{w}_1, \ha{w}_1, \ha{w}_1) (t, k), \\
& \mathcal{N}_f^{(5)} (\ha{w}_1) (t,k):= \mathcal{N}_f^{(5)} (\ha{w}_1, \ha{w}_1, \ha{w}_1, \ha{w}_1, \ha{w}_1)(t,k)
\end{align*}
for $i=1,2,3$. By $|Q^{(3)} \chi_{NR1}^{(3)} | \lesssim \langle k_{\max} \rangle^3$ and Lemma~\ref{lem_nl2}, it follows that 
\begin{equation} \label{nest11}
\big\| \sum_{k=k_{1,2,3}} |Q^{(3)} \chi_{NR1}^{(3)} | \, \prod_{l=1}^3 |\ha{w}_l(t, k_l)| \big\|_{L_T^{\infty} l_{s-3}^2 } 
\lesssim \prod_{l=1}^3 \| w_l \|_{L_T^{\infty} H^s}
\end{equation}
for $\{ w_l \}_{l=1}^3 \subset C([-T, T]: H^s(\T))$, which implies 
\begin{equation} \label{nest12}
\big\| \mathcal{N}_{1,f}^{(3)} (\ha{w}_1, \ha{w}_2, \ha{w}_3)  \big\|_{L_T^{\infty} l_{s-3}^2 } 
\lesssim \prod_{l=1}^3 \| w_l \|_{L_T^{\infty} H^s}. 
\end{equation}
% Note that $N_{1,f}^{(3)} (\ha{w}_1, \ha{w}_2, \ha{w}_3) \in C([-T, T]: l_{s-3}^2) $. 
% since $\| w_l (\cdot)  \|_l_{s-3}^2} \in C([-T, T])$ with $l=1,2,3$. 
By \eqref{nest11} and Lemma~\ref{lem_go}, we have 
\begin{equation} \label{nest13}
\| \mathcal{N}_{1,f}^{(3)} (\ha{w}_1)- \mathcal{N}_{1,g}^{(3)} (\ha{w}_1)  \|_{L_T^{\infty} l_{s-3}^2} \le C_{*}(w_1, s, |E_1(f) -E_1(g) |, |\ga|, T).
\end{equation} 
By $ |Q^{(3)} \chi_{NR1}^{(3)} | \, ([3 \chi_{H2,2}^{(3)} ]_{sym}^{(3)} + \chi_{H3}^{(3)}) \lesssim  
\langle k_{\min}  \rangle \langle k_{\max} \rangle^2 \, ([3 \chi_{H2,2}^{(3)} ]_{sym}^{(3)} + \chi_{H3}^{(3)})$ and 
Lemma~\ref{lem_nl5}, 
\begin{equation}
\big\| \sum_{k=k_{1,2,3}} |Q^{(3)} \chi_{NR1}^{(3)} | \, ([3 \chi_{H2,2}^{(3)} ]_{sym}^{(3)} + \chi_{H3}^{(3)}) \, 
\prod_{l=1}^3 |\ha{w}_l(t, k_l)| \big\|_{L_T^{\infty} l_{s-1}^2 } 
\lesssim \prod_{l=1}^3 \| w_l \|_{L_T^{\infty} H^s} \label{nest14}
\end{equation}
and by Lemma~\ref{lem_stres}  
\begin{equation}
\big\| \sum_{k=k_{1,2,3}} \big| Q^{(3)} [3 \chi_{R3}^{(3)}]_{sym}^{(3)} \big| \,
\prod_{l=1}^3 |\ha{w}_l(t, k_l)| \big\|_{L_T^{\infty} l_{s}^2 } 
\lesssim \prod_{l=1}^3 \| w_l \|_{L_T^{\infty} H^s} \label{nest15}
\end{equation}
for $\{ w_l \}_{l=1}^{3} \subset C([-T, T]: H^s(\T))$, which leads that  
\begin{align}
& \big\| \mathcal{N}_{2,f}^{(3)} (\ha{w}_1, \ha{w}_2, \ha{w}_3)  \big\|_{L_T^{\infty} l_{s-1}^2} \lesssim \prod_{l=1}^3 \| w_l \|_{L_T^{\infty} H^s}, 
\label{nest16} \\
& \big\| \mathcal{N}_{3,f}^{(3)} (\ha{w}_1, \ha{w}_2, \ha{w}_3)  \big\|_{L_T^{\infty} l_{s}^2} \lesssim \prod_{l=1}^3 \| w_l \|_{L_T^{\infty} H^s}. 
\label{nest17}
\end{align}
By \eqref{nest14}, \eqref{nest15} and Lemma~\ref{lem_go}, we have 
\begin{align}
&\big\| \mathcal{N}_{2,f}^{(3)} (\ha{w}_1)- \mathcal{N}_{2, g}^{(3)} (\ha{w}_1) \big\|_{L_T^{\infty} l_{s-1}^2} 
\le C_{*}(w_1, s, |E_1(f)-E_1(g)|, |\ga|, T) , \label{nest18} \\
&\big\| \mathcal{N}_{3,f}^{(3)} (\ha{w}_1)- \mathcal{N}_{3, g}^{(3)} (\ha{w}_1) \big\|_{L_T^{\infty} l_{s}^2} 
\le C_{*}(w_1, s, |E_1(f)-E_1(g)|, |\ga|, T). \label{nest19}
\end{align}
Furthermore, it follows that 
\begin{align}
\big\| \sum_{k=k_{1, \dots, 5}} |Q_1^{(5)}| \, \prod_{l=1}^5 |\ha{w}_l (t, k_l)|  \big\|_{L_T^{\infty} l_{s-1}^2} 
\lesssim \prod_{l=1}^5 \| w_l \|_{L_T^{\infty} H^s}, \label{nest211} \\ 
\big\| \mathcal{N}_{f}^{(5)}  (\ha{w}_1, \ha{w}_2 , \ha{w}_3, \ha{w}_4, \ha{w}_5)  \big\|_{L_T^{\infty} l_{s-1}^2} 
\lesssim \prod_{l=1}^5 \| w_l \|_{L_T^{\infty} H^s}. \label{nest212} 
\end{align}
for $\{ w_l \}_{l=1}^5 \subset C([-T, T]: H^s(\T))$. Thus, by \eqref{nest211} and Lemma~\ref{lem_go}, we have 
\begin{equation} \label{nest20}
\big\| \mathcal{N}_f^{(5)} (\ha{w}_1)- \mathcal{N}_g^{(5)} (\ha{w}_1)  \big\|_{L_T^{\infty} l_{s-1}^2} \le C_{*}(w_1, s, |E_1(f)-E_1(g)|, |\ga|, T).
\end{equation}
%%%%%%%%%%%%%%%%%%%%%%%%%%%%%%%%%%%%%%%%%%%%%%%
Here we introduce the following notations: 
\begin{align*}
& \mathcal{I}_{i, f}^{(N, j)} (\ha{v}_1) (t,k) 
:= \La_f^{(N)} \big( \ti{L}_{j, f}^{(N)} \chi_{>L}^{(N)}, \ha{v}_1, \dots, \ha{v}_1, \mathcal{N}_{i,f}^{(3)} (\ha{v}_1)  \big) (t,k), \\
& \mathcal{K}_{f}^{(N, j)} (\ha{v}_1) (t,k) := 
\La_f^{(N)} \big( \ti{L}_{j, f}^{(N)} \chi_{>L}^{(N)}, \ha{v}_1, \dots, \ha{v}_{1}, \mathcal{N}_{f}^{(5)} (\ha{v}_1)  \big) (t,k).
\end{align*}
% \begin{align*}
% &\mathcal{I}_{i,f}^{(N,j)} (\ha{v}_1)(t, k): = \mathcal{I}_{i,f}^{(N,j)}(\ha{v}_1, \dots, \ha{v}_1) (t, k), \\
% & \mathcal{K}_{f}^{(N,j)} (\ha{v}_1)(t, k): = \mathcal{K}_{f}^{(N,j)}(\ha{v}_1, \dots, \ha{v}_1) (t, k).
% \end{align*}
%%%%%%%%%%%%%%%%%%%%%%%%%%%%%%%%%%%%%%%%%%%%%%%
(I) Firstly, we prove \eqref{nes11} and \eqref{nes13} for $(N,j) \in K_4$. By the definition, it follows that 
\begin{align*}
& \La_f^{(7)}(\ti{M}_{3,f}^{(7)}, \ha{v}_1)= \sum_{j=1,7,8} 5 \, \mathcal{I}_{1,f}^{(5, j)} (\ha{v}_1), \hspace{0.3cm} 
\La_f^{(9)} (\ti{M}_{3,f}^{(7)}, \ha{v}_1)= \sum_{j=1.2} 7 \, \mathcal{I}_{1,f}^{(7, j)} (\ha{v}_1), \\
& \La_f^{(7)} (\ti{M}_{4,f}^{(7)}, \ha{v}_1)= \sum_{j=3,4,5,6} 5 \, \mathcal{I}_{1,f}^{(5,j)} (\ha{v}_1), \hspace{0.3cm} 
\La_f^{(5)}(\ti{M}_{7,f}^{(5)}, \ha{v}_1)= \sum_{j=2,3,4} 3 \, \mathcal{I}_{1,f}^{(3,j)} (\ha{v}_1). 
\end{align*}
Thus, it suffices to show that for $i=1$, $(N,j) \in J_1\cup J_2 \cup J_3$ and $\{ v_l  \}_{l=1}^{N+2} \subset C([-T, T]: H^s(\T)) $ 
\begin{align}
& \Big\| \sum_{k=k_{1, \dots, N}} |\ti{L}_{j, f}^{(N)} \chi_{>L}^{(N)} | \prod_{l=1}^{N-1} |\ha{v}_l (t, k_l) | 
| \mathcal{N}_{i,f}^{(3)} (\ha{v}_N, \ha{v}_{N+1}, \ha{v}_{N+2}) (t, k_N) |  \Big\|_{L_T^{\infty} l_s^2 } 
\lesssim \prod_{l=1}^{N+2} \| v_l \|_{L_T^{\infty} H^s}, \label{nest21} \\
& \big\| \mathcal{I}_{i, f}^{(N,j)} (\ha{v}_1)- \mathcal{I}_{i, g}^{(N, j)} (\ha{v}_1) \big\|_{L_T^{\infty} l_s^2} \le C_{*}. \label{nest22}
\end{align} 
For $(N,j) \in J_1$, by \eqref{pwb1}, it follows that
\begin{equation} \label{nest23}
| \ti{L}_{j, f}^{(N)} \chi_{>L}^{(N)} | \leq [ |L_{j, f}^{(N)} \chi_{>L}^{(N)} |  ]_{sym}^{(N)} 
\lesssim \langle k_{1, \dots, N}  \rangle^{-1} \langle k_{\max} \rangle^{-2}.
\end{equation}
By \eqref{nl41}, \eqref{nest12} and \eqref{nest23}, the left hand side of \eqref{nest21} is bounded by 
\begin{equation*}
\prod_{l=1}^{N-1} \| v_l \|_{L_T^{\infty} H^s} \, 
\big\| \mathcal{N}_{1,f}^{(3)} (\ha{v}_N, \ha{v}_{N+1}, \ha{v}_{N+2}) \big\|_{L_T^{\infty} l_{s-3}^2 } 
\lesssim \prod_{l=1}^{N+2} \| v_l \|_{L_T^{\infty} H^s}. 
\end{equation*}
Similarly, by \eqref{pwb0}, \eqref{nl42} and \eqref{nest12}, we get \eqref{nest21} for $(N, j) \in J_2$. 
Collecting \eqref{pwb01}--\eqref{pwb03}, \eqref{cnl1}--\eqref{cnl3} and \eqref{nest12}, we have \eqref{nest21} for $(N, j) \in J_3$. 
Therefore, we  obtain \eqref{nest21} for $(N, j) \in J_1 \cup J_2 \cup J_3$. 
Next, we prove \eqref{nest22}. By a direct computation, it follows that 
\begin{align*}
& \big[ \mathcal{I}_{1, f}^{(N, j)} (\ha{v}_1)- \mathcal{I}_{1, g}^{(N,j)} (\ha{v}_1) \big] (t,k) \\
&= \La_f^{(N)} \big( (\ti{L}_{j,f}^{(N)}- \ti{L}_{j,g}^{(N)}) \chi_{>L}^{(N)}, \ha{v}_1, \dots, \ha{v}_1, \mathcal{N}_{1,f}^{(3)} (\ha{v}_1) \big) 
(t,k) \notag \\
& + \big[ \La_f^{(N)} \big(\ti{L}_{j, g}^{(N)} \chi_{>L}^{(N)}, \ha{v}_1, \dots, \ha{v}_1, \mathcal{N}_{1,f}^{(3)} (\ha{v}_1) \big) 
- \La_{g}^{(N)} \big(\ti{L}_{j, g}^{(N)} \chi_{>L}^{(N)}, \ha{v}_1, \dots, \ha{v}_1, \mathcal{N}_{1,f}^{(3)} (\ha{v}_1) \big) \big] (t,k) \notag \\
&+ \La_g^{(N)} \big( \ti{L}_{j, g}^{(N)} \chi_{>L}^{(N)}, \ha{v}_1, \dots, \ha{v}_1,  
\,\mathcal{N}_{1,f}^{(3)} (\ha{v}_1) - \mathcal{N}_{1,g}^{(3)} (\ha{v}_1) \big) (t,k) \notag \\
&= : J_{3,1}^{(N+2)} (\ha{v}_1) (t, k)+ J_{3,2}^{(N+2)} (\ha{v}_1)(t, k)+ J_{3,3}^{(N+2)} (\ha{v}_1) (t,k).
\end{align*}
By \eqref{pwb21} and a slight modification of \eqref{nest21}, we have 
\begin{equation*}
\| J_{3,1}^{(N+2)} (\ha{v}_1) \|_{L_T^{\infty} l_s^2 } \lesssim |\ga| |E_1(f) -E_1(g)| \, \| v_1 \|_{L_T^{\infty} H^s }^{N+2} \leq C_{*}.  
\end{equation*}
By \eqref{nest21} and Lemma~\ref{lem_go}, we get  $\|  J_{2,2}^{(N+2)} (\ha{v}_1) \|_{L_T^{\infty} l_s^2} \le C_{*} $. 
By \eqref{pwb0}--\eqref{pwb03}, \eqref{nl41}, \eqref{nl42}, \eqref{cnl1}--\eqref{cnl3}, \eqref{nest13} and \eqref{nest23}, we have 
\begin{equation*}
\| J_{3,3}^{(N+2)} (\ha{v}_1) \|_{L_T^{\infty} l_{s}^2 } \lesssim \| v_1  \|_{L_T^{\infty} H^s}^{N-1}
\big\| \mathcal{N}_{1,f}^{(3)} (\ha{v}_1)- \mathcal{N}_{1,g}^{(3)} (\ha{v}_1) \big\|_{L_T^{\infty} l_{s-3}^2} \le C_{*}. 
\end{equation*}
Therefore, we obtain \eqref{nest22}. \\
%%%%%%%%%%%%%%%%%%%%%%%%%%%%%%%%%%%%%%%%%%%%%%%%%%%%%%%%%%%%%%%%%%%%%%%%
(II) Secondly, we prove \eqref{nes11} and \eqref{nes13} for $(N,j) \in K_5$. Since 
\begin{equation*}
\La_f^{(5)} (\ti{M}_{8,f}^{(7)}, \ha{v}_1)= 3 \,  \mathcal{I}_{2,f}^{(3,1)} (\ha{v}_1), \hspace{0.3cm} 
\La_f^{(7)} (\ti{M}_{5,f}^{(7)}, \ha{v}_1)= 5 \, \mathcal{I}_{2,f}^{(5,2)} (\ha{v}_1), 
\end{equation*}
it suffices to show that \eqref{nest21} and \eqref{nest22} hold for $i=2$, $(N, j) \in J_4$ and $\{ v_l \}_{l=1}^{N+2} \subset C([-T, T]: H^s(\T))$. 
By \eqref{pwb3} and \eqref{pwb21}, it follows that 
\begin{equation}
| \ti{L}_{j, f}^{(N)} \chi_{>L}^{(N)} |  \lesssim \langle k_{\max}  \rangle^{-1}. \label{nest31}
\end{equation}
By \eqref{nest16} and \eqref{nest31}, the left hand side of \eqref{nest21} is bounded by 
\begin{align*}
& \Big\| \sum_{k=k_{1, \dots, N}} \langle k_{\max} \rangle^{-1} \prod_{l=1}^{N-1} |\ha{v}_l(t, k_l)| 
 |\mathcal{N}_{2,f}^{(3)} (\ha{v}_N, \ha{v}_{N+1}, \ha{v}_{N+2}) (t, k_N) |  \Big\|_{L_T^{\infty} l_s^2} \\ 
& \lesssim \prod_{l=1}^{N-1} \| v_l \|_{L_T^{\infty} H^s} 
\big\| \mathcal{N}_{2,f}^{(3)} (\ha{v}_N, \ha{v}_{N+1}, \ha{v}_{N+2}) \big\|_{L_T^{\infty} l_{s-1}^2} \lesssim \prod_{l=1}^{N+2} \| v_l \|_{L_T^{\infty} H^s}. 
\end{align*}
In a similar manner as the case (I), 
by \eqref{pwb21}, \eqref{nest16}, \eqref{nest18}, \eqref{nest31}, Lemma~\ref{lem_go} and a slight modification of \eqref{nest21},  
we get \eqref{nest22}. \\ 
%%%%%%%%%%%%%%%%%%%%%%%%%%%%%%%%%%%%%%%%%%%%%%%%%%%%%%%%%%%%%%%%%%%
(III) Thirdly, we prove \eqref{nes11} and \eqref{nes13} for $(N,j) \in K_6$. 
% Since 
% \begin{align*}
% & \La_f^{(5)} (\ti{M}_{6, f}^{(5)}, \ha{v}_1)= \sum_{j=1}^4 3 \, \mathcal{I}_{3,f}^{(3,j)} (\ha{v}_1), \hspace{0.3cm} 
% \La_f^{(7)} (\ti{M}_{2, f}^{(7)}, \ha{v}_1)= \sum_{j=1}^8 5 \, \mathcal{I}_{3,f}^{(5,j)} (\ha{v}_1), \\ 
% & \La_f^{(9)} (\ti{M}_{2,f}^{(9)}, \ha{v}_1)= \sum_{j=1}^2 7 \, \mathcal{I}_{3,f}^{(7,j)} (\ha{v}_1),
% \end{align*}
By the definition of $M_{j, f}^{(N)}$ with $(N, j) \in K_6$, it suffices to show that \eqref{nest21} and \eqref{nest22} hold 
for $i=3$, $(N, j) \in J_1 \cup J_2 \cup J_3 \cup J_4$ and $\{ v_l \}_{l=1}^{N+2} \subset C([-T, T]: H^s(\T))$. 
By Lemma~\ref{lem_pwb1}, it follows that $|\ti{L}_{j, f}^{(N)} \chi_{>L}^{(N)} | \lesssim 1$. Thus, by \eqref{nest17}, we get \eqref{nest21}. 
In a similar manner as the case (I), by \eqref{pwb21}, \eqref{nest17}, \eqref{nest19}, Lemma~\ref{lem_go}, Lemma~\ref{lem_pwb1} and 
a slight modification of \eqref{nest21}, we obtain \eqref{nest22}. \\
%%%%%%%%%%%%%%%%%%%%%%%%%%%%%%%%%%%%%%%%%%%%%%%%%%%%%%%%%%
(IV) Finally, we prove \eqref{nes11} and \eqref{nes13} for $(N, j) \in K_7$. 
%% Since 
%% \begin{align*}
%% & \La_f^{(7)} (\ti{M}_{1,f}^{(7)}, \ha{v}_1)= \sum_{j=1}^4 3 \, \mathcal{K}_f^{(3, j)} ( \ha{v}_1 ), \hspace{0.3cm} 
%% \La_f^{(9)} ( \ti{M}_{1,f}^{(9)}, \ha{v}_1 )= \sum_{j=1}^8 5 \, \mathcal{K}_f^{(5, j)} (\ha{v}_1), \\
%% & \La_f^{(11)} (\ti{M}_{1,f}^{(11)}, \ha{v}_1)= \sum_{j=1}^2 7 \, \mathcal{K}_f^{(7, j)} (\ha{v}_1),
%% \end{align*}
By the definition of $M_{j, f}^{(N)}$ with $(N, j) \in K_7$, 
it suffices to show that for $(N, j) \in J_1 \cup J_2 \cup J_3 \cup J_4$ and 
$\{ v_l \}_{l=1}^{N+4} \subset C([-T, T]: H^s(\T))$
\begin{align}
& \Big\| \sum_{k=k_{1, \dots, N}} |\ti{L}_{j, f}^{(N)} \chi_{>L}^{(N)} |  \prod_{l=1}^{N-1} |\ha{v}_l(t, k_l)| 
 |\mathcal{N}_{f}^{(5)} (\ha{v}_N, \ha{v}_{N+1}, \ha{v}_{N+2}, \ha{v}_{N+3}, \ha{v}_{N+4}) (t, k_N) |  \Big\|_{L_T^{\infty} l_s^2}  \notag \\
& \hspace{0.5cm} \lesssim \prod_{l=1}^{N+4} \| v_l \|_{L_T^{\infty} H^s}, \label{nest51} \\
& \big\| \mathcal{K}_f^{(N,j)} (\ha{v}_1)- \mathcal{K}_g^{(N,j)} (\ha{v}_1) \big\|_{L_T^{\infty} l_s^2 } \le C_{*}. \label{nest52}
\end{align}
By Lemma~\ref{lem_pwb1}, it follows that 
\begin{equation} 
|\ti{L}_{j,f}^{(N)} \chi_{>L}^{(N)} | \lesssim \langle k_{\max} \rangle^{-1} \label{nest53} 
\end{equation}
for $(N, j) \in  J_1 \cup J_2 \cup J_3 \cup J_4 \setminus \{ (3,4) \} $. 
By \eqref{cnl3}, \eqref{pwb03}, \eqref{nest212} and \eqref{nest53}, the left hand side of \eqref{nest51} is bounded by 
\begin{equation*}
\prod_{l=1}^{N-1} \| v_l \|_{L_T^{\infty} H^s} 
\big\| \mathcal{N}_f^{(5)} (\ha{v}_N, \ha{v}_{N+1}, \ha{v}_{N+2}, \ha{v}_{N+3}, \ha{v}_{N+4} )  \big\|_{L_T^{\infty} l_{s-1}^2 } 
\lesssim \prod_{l=1}^{N+4} \| v_l \|_{L_T^{\infty} H^s }. 
\end{equation*}
In a similar manner as the case (I), 
by \eqref{cnl3}, \eqref{pwb03}, \eqref{pwb21}, \eqref{nest212}, \eqref{nest20}, \eqref{nest53}, Lemma~\ref{lem_go} 
and a slight modification of \eqref{nest51}, we obtain \eqref{nest52}. 
\end{proof}

%%%%%%%%%%%%%%%%%%%%%%%%%%%%%%%%%%%%%%%%%%%%%%%%%%%%
%%%%%%%%%%%%%%%%%%%%%%%%%%%%%%%%%%%%%%%%%%%%%%%%%%%%
%%%%%%%%%%%%%%%%%%%%%%%%%%%%%%%%%%%%%%%%%%%%%%%%%%%
Now, we prove Proposition~\ref{prop_main1}. 
\begin{proof}[Proof of Proposition~\ref{prop_main1}] 
By the definition of $F_{f, L}$ and Lemma~\ref{lem_nes0}, we have (\ref{mes11}) and (\ref{mes13}).  
For $(N,j) \in J_1 \cup J_2 \cup J_3 \cup J_4$, we can easily check
\begin{align} \label{mes20}
|L_{j, f}^{(N)} \Phi_{f}^{(N)} \chi_{ \le L}^{(N)} | \lesssim L^3. 
\end{align} 
By (\ref{mes20}), we have 
\begin{align} \label{mes21}
\Big\| \sum_{k=k_{1, \dots, N}} |\ti{L}_{j, f}^{(N)} \Phi_{f}^{(N)} \chi_{\le L}^{(N)}| \, \prod_{l=1}^N |\ha{v}_l(t, k_l)| \Big\|_{L_T^{\infty} l_s^2 } 
\lesssim L^{3} \prod_{l=1}^N \| v_l \|_{L_T^{\infty} H^s}
\end{align}
for any $\{ v_l \}_{l=1}^N \subset C([-T, T]: H^s(\T))$. 
Thus, by the definition of $G_{f, L}$, Lemmas~\ref{lem_mle}--\ref{lem_nl10}, we have (\ref{mes12}). 
Next, we prove (\ref{mes14}). 
For $(N, j) \in J_1 \cup J_2 \cup J_3 \cup J_4$, by the definition and (\ref{mes20}), it follows that 
\begin{align}
|\ti{L}_{j, f}^{(N)} \Phi_{f}^{(N)} \chi_{\le L}^{(N)}- \ti{L}_{j, g}^{(N)} \Phi_g^{(N)} \chi_{\le L}^{(N)}  | 
%& \lesssim |\ga| |E_1(f)- E_1(g)| [ |L_{j, f}^{(N)} \Phi_f^{(N)} \chi_{\le L}^{(N)}  | ]_{sym}^{(N)} \notag \\
\lesssim  |\ga| |E_1(f)- E_1(g)| L^3 .  \label{mes23}
\end{align}
By (\ref{mes21}), (\ref{mes23}) and Lemma~\ref{lem_go}, we have 
\begin{align*}
\big\|  \La_f^{(N)} ( \ti{L}_{j, f}^{(N)} \Phi_f^{(N)} \chi_{\le L}^{(N)}, \ha{v} )
- \La_g^{(N)} ( \ti{L}_{j, g}^{(N)} \Phi_g^{(N)} \chi_{\le L}^{(N)}, \ha{v} )  \big\|_{L_T^{\infty} l_s^2} \le C_*.
\end{align*}
Hence, by the definition of $G_{f, L}$ and Lemmas~\ref{lem_mle}--\ref{lem_nl10}, we obtain (\ref{mes14}). 
\end{proof}

%%%%%%%%%%%%%%%%%%%%%%%%%%%%%%%%%%%%%%%%%%%%%%%%%%%%%
%%%%%%%%%%%%%%%%%%%%%%%%%%%%%%%%%%%%%%%%%%%%%%%%%%%%%
%%%%%%%%%%%Proof of  the main result %%%%%%%%%%%%%%%%%%%%%%
%%%%%%%%%%%%%%%%%%%%%%%%%%%%%%%%%%%%%%%%%%%%%%%%%%%%%%
%%%%%%%%%%%%%%%%%%%%%%%%%%%%%%%%%%%%%%%%%%%%%%%%%%%%%%
\section{Proof of Theorem~\ref{thm_LWP}}

In this section, we give the proof of Theorem~\ref{thm_LWP}. 
By Lemma~\ref{lem_E1} below, it is verified that the following conservation law holds:
\begin{equation} \label{cle1}
E_1(u)(t)=E_1 (\vp)  \hspace{0,5cm} (t \in [-T, T])
\end{equation}
when $u \in C([-T,T]: H^{3/2}(\T))$ is a solution to \eqref{5mKdV1} or \eqref{5mKdV2} with \eqref{initial} and $\al = \be$. 
Thus, the unconditional local well-posedness for \eqref{5mKdV1} with \eqref{initial} is equivalent to 
that for \eqref{5mKdV2} with \eqref{initial} under the assumption that $s \ge 3/2$ and $\alpha= \beta$ or $\ga=0$. 
%%%%%%%%%%%%%%%%%%%%%%%%%%%%%%%%%%%%%%%%%%%%%%%%%%%%%%%%%%%%%%%%%%%%%%%%%%%%%
\begin{lem} \label{lem_E1}
Let $s \ge 3/2$, $\alpha= \be$, $\vp \in H^s(\T)$ 
and $u \in C([-T,T]: H^{s} (\mathbb{T}))$ be a solution to \eqref{5mKdV1} or \eqref{5mKdV2} with \eqref{initial}. 
Then it follows $\| u (t) \|_{L^2} = \| \varphi \|_{L^2}$ for any $t \in [-T,T]$. 
\end{lem} 
%%%%%%%%%%%%%%%%%%%%%%%%%%%%%%%%%%%%%%%%%%%%%%%%%%%%%%%%%%%%%%%%%%%%%%%%%%%%%%%%
\begin{proof}
First, we suppose that $u$ is a solution to \eqref{5mKdV1} with \eqref{initial}.
For $N \in \mathbb{N}$, the projection operator $P_{\leq N}$ is defined as 
\begin{align*}
P_{\leq N} u(t,x)= \sum_{|k| \leq N} e^{ikx} \ha{u}(t,k). 
\end{align*}
$P_{ \leq N} u$ belongs to $C^{1} ([-T,T]:H^{\infty} (\mathbb{T}) )$ since 
\begin{align*}
\p_t P_{\leq N} u & =  -\p_x^5 P_{\leq N} u- 6 \delta P_{\leq N} \p_x (u^5)
- \alpha P_{\leq N} (\p_x u)^3 - \beta  P_{\leq N} \p_x (u (\p_x u)^2)\\
& - \gamma P_{\leq N} \{\p_x^2 (u \p_x (u^2) )-\p_x (\p_x u \p_x (u^2) )\} 
 \in C([-T,T]:H^{\infty}). 
\end{align*}
Thus we calculate $\frac{d}{dt} \| P_{\leq N} u(t) \|_{L^2}^2$ in the classical sense as follows:
\begin{equation}\label{es74}
\begin{split}
& \frac{d}{dt} \| P_{\leq N} u(t) \|_{L^2}^2= 2 \int_{\mathbb{T}} P_{\leq N} u \p_t P_{\leq N} u \,  dx \\
& = -12 \delta \int_{\mathbb{T}} P_{\leq N} u P_{\leq N} \p_x(u^5) \, dx
- 2 \alpha \int_{\mathbb{T}} P_{\leq N} u P_{\leq N} (\p_x u)^3 \, dx \\
& - 2 \beta \int_{\mathbb{T}} P_{\leq N} u P_{\leq N} \p_x (u (\p_x u)^2 ) \, dx \\
& -2 \gamma \int_{\mathbb{T}} P_{\leq N} u P_{\leq N} \{\p_x^2 (u \p_x (u^2) )-\p_x (\p_x u \p_x (u^2) )\} \, dx \\
&=: I_1+I_2+I_3+I_4. 
\end{split}
\end{equation}
We first estimate $I_1$. A direct calculation yields that 
\begin{align*}
I_1& = 12 \delta \int_{\mathbb{T}} P_{\leq N} \p_x u P_{\leq N} u^5 \, dx= 
12 \delta \int_{\mathbb{T}} u^5 \p_x u \, dx- 12 \delta \int_{\mathbb{T}} u^5 P_{>N} \p_x u \, dx \notag \\
& = -12 \de \int_{\T} u^5  P_{>N} \p_x u \, dx. 
\end{align*}
By the H\"{o}lder inequality and the Sobolev inequality, we have 
\begin{align*}
|I_{1}| \leq \| u \|_{L_{T}^{\infty}  L^{10}}^5 \| P_{> N} \p_x u \|_{L_T^{\infty} L^2} \lesssim 
\| u \|_{L_T^{\infty} H^{2/5}}^5 \| P_{> N} u \|_{ L_T^{\infty} H^1}. 
\end{align*} 
Next we estimate $I_2+I_3+I_4$. A simple computation yields that 
\begin{align*}
 I_2+I_3+I_4 & = 2(\be-\alpha) \int_{\mathbb{T}} u (\p_x u)^3 \, dx 
+2\alpha \int_{\mathbb{T}} (P_{>N}  u ) (\p_x u)^3~dx \notag \\
& -(2\beta- 4 \ga ) \int_{\mathbb{T}} \big( P_{>N} \p_x u \big)  \, u (\p_x u)^2  \, dx 
+4 \gamma \int_{\mathbb{T}} \big( P_{>N} \p_x^2 u \big) u^2 \p_x u \, dx \notag \\
&=: H_1+H_2+H_3+H_4. 
\end{align*}
%%%%%%%%%%%%%%%%%%%%%%%%%%%%%
By the assumption $\alpha= \beta$, $H_1$ is equal to $0$. 
By the H\"{o}lder inequality and the Sobolev inequality, we have
\begin{align*}
& |H_2| \lesssim \| P_{>N}  u \|_{L_T^{\infty} L^{\infty}} \| \p_x u \|_{L_T^{\infty} L^3}^3 
\lesssim  \| u \|_{ L_T^{\infty} H^{7/6}}^3 \| P_{>N} u \|_{L_T^{\infty} H^{7/6}}, \\ 
& |H_3| \lesssim \| u \|_{L_T^{\infty} L^{\infty}} \| P_{>N} \p_x u \|_{L_T^{\infty} L^3} \|  \p_x u  \|_{L_T^{\infty} L^3}^2 
\lesssim  \|  u \|_{L_T^{\infty} H^{7/6}}^3 \| P_{>N} u \|_{ L_T^{\infty} H^{7/6}} , \\
& |H_4| \lesssim \| u \|_{ L_T^{\infty} H^{1/2+}}^2 \| \p_x u \|_{L_T^{\infty} H^{1/2}} \| P_{>N} \p_x^2 u \|_{L_T^{\infty} H^{-1/2}} 
\lesssim \| u \|_{L_T^{\infty} H^{3/2}}^3 \| P_{>N} u \|_{L_T^{\infty} H^{3/2}}.
\end{align*}
Integrating \eqref{es74} over $[0,t]$ for $t\in [-T,T]$, we have 
\begin{equation*}
\bigl| \| P_{\leq N} u (t) \|_{L^2}^2 - \| P_{\leq N} \vp \|_{L^2}^2 \bigr|
\lesssim T \bigl(  \| u \|_{L^\infty_T H^{3/2} }^3 + \| u \|_{L^\infty_T H^{3/2}}^5 \bigr) \| P_{> N} u \|_{L^\infty_T H^{3/2} }. 
\end{equation*}
Take $N\to \infty$.
Then, by the dominated convergence theorem and uniform $H^{3/2}$-continuity on $[-T,T]$, the right-hand side goes to $0$.
Furthermore, $P_{\le N} u(t) \to u(t)$ for $t$ fixed and $P_{\le N} \vp \to \vp$ in $L^2(\T)$.
Thus, we obtain 
$$ \big| \| u (t) \|_{L^2}^2- \| \vp \|_{L^2}^2 \big| =0 $$  
for any $t \in [-T,T]$. 

Next, we suppose that $u$ is a solution to \eqref{5mKdV2} with \eqref{initial}.
We notice that \eqref{5mKdV2} is equivalent to
\EQQ{
&\p_t u +\p_x^5 u +2\ga(E_1(\vp)-E_1(u))\p_x^3 u \\
&+\al(\p_x u)^3+\be\p_x(u(\p_x u)^2)+\ga\p_x(u\p_x^2( u^2))+6\de \p_x (u^5)=0.
}
Since $2\ga(E_1(\vp)-E_1(u))\int_\T P_{\leq N} u \p_x^3 P_{\leq N} u\, dx=0$,
we have the same result as above in the same manner.
\end{proof}
%%%%%%%%%%%%%%%%%%%%%%%%%%%%%%%%%%%%%%%%%%%%%%%%%%%%%%%%%%%%
Here, we define the translation operators $\mathcal{T}_{v}$ and $\mathcal{T}_{v}^{-1}$ by 
\begin{align*}
\mathcal{T}_v u(t, x) := u \Big( t,  x+ \int_0^t [K(v)] (t') \, dt'  \Big), \hspace{0.3cm}
\mathcal{T}_{v}^{-1} u(t,x):=u \Big(t, x- \int_0^t [K(v)] (t') \, dt'  \Big) 
\end{align*}
for $v \in C([-T, T]: H^1(\T))$. 
In a similar manner as Proposition 8.1 in \cite{KTT}, the following lemma holds. 
%%%%%%%%%%%%%%%%%%%%%%%%%%%%%%%%%%%%%%
\begin{lem} \label{lem_homeo}
Let $s \ge 1$. Then, a map $\mathcal{T}: u \mapsto \mathcal{T}_u u$ is a homeomorphism on 
$C([-T,T]: H^s(\T))$.
\end{lem}
%%%%%%%%%%%%%%%%%%%%%%%%%%%%%%%%%%%%%%%%%%%
We notice that $\mathcal{T}^{-1}u = \mathcal{T}_{u}^{-1} u $ holds when $u \in C([-T, T]: H^s(\T))$ for $s \ge 1$. 
Thus, we have the following lemma. 
%%%%%%%%%%%%%%%%%%%%%%%%%%%%%%%%%%%%%%%%%%%%%%%%%%%%%%
\begin{lem} \label{lem_equiv1}
Let $s \ge 3/2$ and $\al = \be$ or $\ga=0$. 

(i) If $u \in C([-T, T]: H^s(\T))$ is a solution to \eqref{5mKdV2} with \eqref{initial}, then $\mathcal{T} u \in C([-T, T]: H^s(\T))$ is 
a solution to \eqref{5mKdV3} with \eqref{initial}.

(ii) If $u \in C([-T, T]: H^s(\T))$ is a solution to \eqref{5mKdV3} with \eqref{initial}, then $\mathcal{T}^{-1} u \in C([-T, T]: H^s(\T))$ is 
a solution to \eqref{5mKdV2} with \eqref{initial}. 
\end{lem} 
%%%%%%%%%%%%%%%%%%%%%%%%%%%%%%%%%%%%%%%%%%%%%%%%%%%%%%%%
By \eqref{cle1} and Lemmas \ref{lem_homeo}--\ref{lem_equiv1}, Theorem ~\ref{thm_LWP} is equivalent to Proposition~\ref{prop_LWP} below 
and we only  need to show it.  
%%%%%%%%%%%%%%%%%%%%%%%%%%%%%%%%%%%%%%%%%%%%%%%%%%
\begin{prop} \label{prop_LWP}
Let $s \geq 3/2$ and $\al=\be$ or $\ga=0$. Then, for any $\varphi \in H^s (\mathbb{T})$, 
there exists $T=T(\| \varphi \|_{H^{s}}) >0$ and a unique solution $u \in C([-T,T]: H^{s} (\T))$ to 
\eqref{5mKdV3} with \eqref{initial}. 
 Moreover the solution map, $H^s(\T) \ni \varphi \mapsto u \in C([-T,T]: H^s (\T))$ is continuous.  
\end{prop}
%%%%%%%%%%%%%%%%%%%%%%%%%%%%%%%%%%%%%%%%%%%%%%%%%%%%%%%%%%%
%%%%%%%%%%%%%%%%%%%%%%%%%%%%%%%%%%%%%%%%%%%%%%%%%%%%%%%%%%%
\begin{rem}
As we proved in Proposition \ref{prop_NF2}, for any solution $u \in C([-T,T]:H^s(\T))$ to \eqref{5mKdV3}, $\ha{v}=e^{-t\phi_\vp(k)}\ha{u}$ satisfies \eqref{NF21}.
By the standard argument with Proposition \ref{prop_main1}, 
we can prove the existence of solutions of \eqref{NF21}.
However, this argument is useless to show the existence of solutions to \eqref{5mKdV3} since we do not know whether $u:=U_\vp(t) v$  satisfy \eqref{5mKdV3} or not when $\ha{v}$ satisfy \eqref{NF21}.
To avoid this difficulty, we use the existence of the solution to \eqref{5mKdV1}--\eqref{initial} for sufficiently smooth initial data.
\end{rem}
In \cite{T},  the second author has proved the local well-posedness of fifth order dispersive equations for sufficiently smooth data by the modified energy method. 
We can easily check that the nonlinear term of \eqref{5mKdV1} is non-parabolic resonant type. 
Therefore, we have the following result by Theorem 1.1 in \cite{T}. 
%%%%%%%%%%%%%%%%%%%%%%%%%%%%%%%%%%%%%%%%%%%%%%%%%%%%%%%%%%%%%%%%%%
\begin{prop}\label{prop_existence}
Let $m\in \N$ be sufficiently large.
Then, \eqref{5mKdV1}--\eqref{initial} is locally well-posed in $H^m(\T)$ on $[-T,T]$ without any condition on $\al, \be, \ga$ and $\de$.
The existence time $T=T(\|\vp\|_{H^m})>0$.
\end{prop}
\begin{rem}\label{rem_time}
Since $u(T), u(-T)\in H^{m}(\T)$, we can extend the solution on $[-T-T',T+T']$
where $T'=T'(\|u(T)\|_{H^{m}},\|u(-T)\|_{H^{m}})>0$.
Iterating this process, the solution can be extended on $(-T^{max},T^{max})$ 
where $T^{max}$ satisfies
$$\liminf_{t\to T^{max}} \|u(t)\|_{H^{m}} =\infty, \ \liminf_{t\to -T^{max}} \|u(t)\|_{H^{m}} =\infty \ \ \text{or} \ \  T^{max}=\infty.$$
\end{rem}
%%%%%%%%%%%%%%%%%%%%%%%%%%%%%%%%%%%%%%%%%%%%%%%
%%%%%%%%%%%%%%%%%%%%%%%%%%%%%%%%%%%%%%%%%%%%%%%

%%%%%%%%%%%%%%%%%%%%%%%%%%%%%%%%%%%%%%%%%%%%%
\begin{proof}[Proof of Proposition~\ref{prop_LWP}]
Firstly, we show the existence of the solution. 
For any $\vp \in H^s(\T)$, there exists $\{ \vp_n  \}_{n=1}^{\infty} \subset H^{\infty} (\T)$ such that 
$\vp_n \to \vp$ in $H^s(\T)$, $\| \vp_n \|_{H^s} \le \| \vp \|_{H^s}$ and $E_1(\vp_n) \le E_1(\vp)$.  
% If $E_1(\vp_n) \neq 0$, we take $\ti{\vp}_n:= \vp_n \, \frac{E_1(\vp)}{ E_1(\vp_n) }$. Then, it follows that $E_1(\ti{\vp}_n)=E_1(\vp)$. 
By Proposition~\ref{prop_existence} and Remark ~\ref{rem_time}, 
we have the existence time $T_n^{max}>0$ and the solution 
$w_n \in C((-T_{n}^{max}, T_n^{max}):H^{\infty}(\T))$ to (\ref{5mKdV1}) with initial data $\vp_n$. 
By the conservation law \eqref{cle1} and (i) of Lemma~\ref{lem_equiv1}, 
$u_n:=\mT w_n$ is in $C((-T_n^{\max}, T_n^{\max}): H^{\infty} (\T))$ and satisfies
\begin{align} \label{5thmKdV}
\p_t u_n +\p_x^5 u_n +2 \ga E_1(\vp_n) \p_x^3 u_n  =J_1(u_n)+ J_2 (u_n)+ J_3(u_n)+J_4(u_n).
\end{align}
Thus, by (\ref{mes31}) with $u_1=u_n$ and $\vp_1= \vp_n$, 
there exists a constant $C \ge 1$ such that the following estimate holds for any 
$0< \delta < T_n^{max}$ and $L \gg \max \{ 1, |\ga|  E_1(\vp) \}$:
\EQS{ \label{es71}
\| u_n \|_{L_{\de}^{\infty}H^s} \leq &  \| \vp_n \|_{H^s}+ C L^{-1}  (1+ \| u_n \|_{L_{\de}^{\infty} H^{s} } )^6 \| u_n \|_{L_{\delta}^{\infty} H^s} \nonumber \\
& + C \delta L^3 (1+ \| u_n  \|_{L_{\delta}^{\infty} H^s })^{10} \| u_n \|_{L_{\delta}^{\infty} H^s}. 
}
Here, we take sufficiently large $L$ and small $T_0>0$ such that
\begin{align} \label{co71}
CL^{-1} (1 + 4\| \vp  \|_{H^s})^6 \le \frac{1}{4}, \ \ C T_0 L^{3} (1+4 \| \vp \|_{H^s})^{10} \le \frac{1}{4}.
\end{align}
If $0 < \delta < T:=\min\{T_n^{\max}, T_0\}$ and $\| u_n \|_{L_{\de}^{\infty} H^{s} } \le 4 \| \vp \|_{H^{s}}$, 
then, by (\ref{es71}) and (\ref{co71}), we have 
\EQ{ \label{es72}
\| u_n \|_{L_{\de}^{\infty} H^{s}} \le 2 \| \vp_n \|_{H^{s}} \le 2 \| \vp \|_{H^{s}}. 
}
Therefore, by continuity argument, 
we obtain (\ref{es72}) for $\de=T$ without the assumption $\| u_n \|_{L_{\de}^{\infty} H^{s}} \leq 4 \| \vp \|_{H^s} $. 
By Remark~\ref{rem_time}, it follows that $T< T_n^{max}$. Thus, $T=T_0$. 
Moreover, by (\ref{mes32}) with $u_1=u_m$, $u_2=u_n$, $\vp_1=\vp_m$ and $\vp_2 = \vp_n$, we have 
\begin{align} \label{es73}
& \| u_m- u_n \|_{L_T^{\infty} H^s}  \le \| \vp_m - \vp_n \|_{ H^s}  + (1+T)C_{*} \notag \\
& \hspace{0.5cm} 
+ CL^{-1} (1+ \| u_m \|_{L_T^{\infty} H^s}+ \|  u_n \|_{L_T^{\infty} H^s })^{6} (\| u_m - u_n \|_{L_T^{\infty} H^s}+C_{*}) \notag \\
& \hspace{0.5cm} + C T L^3  (1+ \| u_m \|_{L_T^{\infty} H^s} + 
\| u_n \|_{L_T^{\infty} H^s})^{10}  (\|  u_m- u_n \|_{L_T^{\infty} H^s}+C_{*}) 
\end{align}
Applying (\ref{co71}) and (\ref{es72}) with $\de=T$ to (\ref{es73}), we have  
\begin{align*}
\| u_m- u_n \|_{L_T^{\infty} H^s} \le 2 \| \vp_m -\vp_n  \|_{H^s} + (3+2T) C_{*}  \to 0
\end{align*}
as $m,n \to \infty$. 
Since $u_n$ satisfies \eqref{5thmKdV} and \eqref{es72} with $\de=T$, we also have
\begin{align*}
& \|\p_t (u_m-  u_n)\|_{L_T^{\infty} H^{s-5}} \\
&\lesssim  (1+\| \vp \|_{H^s}^2 + \| \vp \|_{H^s}^4)  \|u_m-u_n\|_{L_{T}^{\infty} H^{s}}
 + \| \vp \|_{H^s} |E_1(\vp_m) -E_1(\vp_n) | \to 0
\end{align*}
as $m, n \to \infty$.
Therefore, there exists $u \in C([-T, T]):H^s(\T)) \cap C^1([-T, T]:H^{s-5}(\T))$ such that $u_n \to u$ strongly in $C([-T, T]:H^s(\T))\cap C^1([-T, T]:H^{s-5}(\T))$.
Then $u$ satisfies \eqref{5mKdV3} in the sense of $C([-T, T]:H^{s-5}(\T))$. 

%%%%%%%%%%%%%%%%%%%%%%%%%%%%%%%%%%%%%%%%%%%%%%%%%%%%%%%%%%%%%%%%
Next, we prove the uniqueness.
Assume that $u_1, u_2 \in C([-T, T] :H^s(\T))$ satisfy \eqref{5mKdV3} with \eqref{initial} for the same initial data $\vp\in H^s(\T)$.
Then, by (\ref{mes32}) with $\vp_1=\vp_2=\vp$, we have 
\begin{align*}
&\| u_1- u_2 \|_{L_\de^{\infty} H^s} \lec \| u_1- u_2 \|_{L_\de^{\infty} H^s}\\
& \times \left(L^{-1} (1+ \| u_1 \|_{L_\de^{\infty} H^s} + \| u_2 \|_{L_\de^{\infty} H^s})^{6} +
\de L^3  (1+ \| u_1 \|_{L_\de^{\infty} H^s} + \| u_2 \|_{L_\de^{\infty} H^s})^{10} \right)
\end{align*}
for any $0 < \de \le T$. 
Therefore, we obtain $u_1(t)=u_2(t)$ on $[-\de,\de]$ by taking sufficiently small $\de>0$ and large $L$.
Note that $\de$ and $L$ depend only on $\| u_j \|_{L_T^{\infty} H^s}$ for $j=1,2$ because
$\| u_j \|_{L_\de^{\infty} H^s} \le \| u_j \|_{L_T^{\infty} H^s}$ for $j=1,2$.
Iterating this argument, we obtain $u_1(t)=u_2(t)$ on $[-T,T]$. 

%%%%%%%%%%%%%%%%%%%%%%%%%%%%%%%%%%%%%%%%%%%%%%%%%%%%%%%%%%%%%%%%%%%%%%%%%
Finally, we can easily show the proof of the continuity of the solution map by Corollary \ref{cor_mainest}.
\end{proof}

%%%%%%%%%%%%%%%%%%%%%%%%%%%%%%%%%%%%%%%
%%%%%%%%%%%%%%%%%%%%%%%%%%%%%%%%%%%%%%%
%%%%%%%%%%%%   References %%%%%%%%%%%%%%%%%%%%
%%%%%%%%%%%%%%%%%%%%%%%%%%%%%%%%%%%%%%%
%%%%%%%%%%%%%%%%%%%%%%%%%%%%%%%%%%%%%%%


\begin{thebibliography}{99}
\bibitem{BIT} A. Babin, A. Ilyin and E. Titi, {\itshape  On the regularization mechanism for the periodic Korteweg-de Vries equation}, 
Comm. Pure Appl. Math. {\bfseries 64} (2011), no. 5, 591--648.  
\bibitem{Bo} J. Bourgain, {\itshape Fourier transform restriction phenomena for certain lattice 
subsets and applications to nonlinear evolution equations. II. The KdV-equation}, 
Geom. Funct. Anal. {\bfseries 3} (1993), no. 3, 209--262. 
\bibitem{BKV} B. Bringmann, R. Killip and M. Visan, {\itshape Global well-posedness for the Fifth-Order KdV Equation in $H^{-1}(\R)$}, 
Ann. of PDE {\bfseries 7} (2021), no. 2, Paper No.21, 46pp. % inverse scattering, 5th KdV on $H^{-1}(\R)$
\bibitem{GKK} Z. Guo, C. Kwak and S. Kwon, {\itshape Rough solutions of the fifth-order KdV equations}, J. Funct. Anal. {\bf 265} (2013), no. 11, 2791--2829.
\bibitem{GKO} Z. Guo, S. Kwon and T. Oh, {\itshape Poincar\'{e}-Dulac normal form reduction for unconditional well-posedness of the periodic cubic NLS}, Comm. Math. Phys. {\bfseries 322} (2013), no. 1, 19--48.  
%\bibitem{HKV} B. Harrop-Griffiths, R. Killip and M. Visan, 
%{\itshape Sharp well-posedness for the cubic NLS and mKdV in $H^s(\R)$}, Forum Math. Pi {\bfseries 12} (2024), Paper No. e6, 86pp. 
\bibitem{KM} T. Kappeler and J.-C. Molnar, {\itshape On the well-posedness for the KdV/KdV2  equations and their frequency maps}, 
 Ann. Inst. H. Poincar\'{e} C Anal. Non Lin\'{e}aire {\bfseries 35} (2018), no. 1, 101--160.  % inverse scattering, 5th KdV on $L^2(\T)$ 
%\bibitem{KT} T. Kappeler and T. Topalov, {\itshape Global well-posedness of KdV in $H^{-1} (\T, \R)$},
% Duke Math. J. {\bfseries 135} (2006), no. 2, 327--360. % inverse scattering, KdV on $H^{-1} (\T)$ 
%\bibitem{KT2} T. Kappeler and P. Topalov, {\itshape Global well-posedness of mKdV in $L^2(\T, \R)$}, 
%Comm. PDE {\bfseries 30} (2005), no. 1--3, 435--449. 
%\bibitem{TK} T. Kato, {\itshape On nonlinear Sch\"{o}dinger equations, II, $H^s$-solutions and unconditional well-posedness}, J. Anal. Math. 
%{\bfseries 67} (1995), 281--306. 
\bibitem{TKK} T.K. Kato, {\itshape Unconditional well-posedness of fifth order KdV type equations with periodic boundary condition}, 
Harmonic analysis and nonlinear partial differential equations, 105--129, RIMS K\^{o}ky\^{u}roku Bessatsu {\bfseries B70} 
Res. Int. Math. Sci. (RIMS), Kyoto, 2018. 
\bibitem{KTT} T. K. Kato and K. Tsugawa, {\itshape Cancellation properties and unconditional well-posedness for the fifth order KdV type equations with periodic boundary condition}, Partial Differ. Equ. Appl. {\bfseries 5} (2024), no. 3, Paper No. 18, 55 pp.
\bibitem{KP} C. E. Kenig and D. Pilod, {\itshape Well-posedness for the fifth-order KdV equation in the energy space},  Trans. Amer. Math. Soc., {\bf 367} (2015), no.4, 2551--2612.
%\bibitem{KV} R. Killip and M. Visan, {\itshape KdV is well-posed in $H^{-1}$}, Ann. of Math. {\bfseries 190} 
%(2019), no. 1, 249--305. % inverse scattering, KdV on $H^{-1} (\R)$
\bibitem{Kn1} N. Kishimoto, {\itshape Unconditional local well-posedness for periodic NLS}, J. Differential Equations 
{\bfseries 274} (2021), 766--787. 
\bibitem{Kn2} N. Kishimoto, {\itshape Unconditional uniqueness of solutions for nonlinear dispersive equations}, 
arXiv:math/1911.04349v4 [math. AP]. 
\bibitem{Kn3} N. Kishimoto, {\itshape Remarks on periodic Zakharov systems}, Electron. J. Differential Equations 
{\bfseries 2022}, Paper No. 20, 19pp. 
\bibitem{Kn4} N. Kishimoto, {\itshape Unconditional uniqueness for the periodic modified Benjamin-Ono equation by normal form approach}, 
Int. Math. Res. Not. {\bfseries 2022} no. 16, 12180--12219. 
\bibitem{Kn5} N. Kishimoto, {\itshape Unconditional uniqueness for the periodic Benjamin-Ono equation by normal form approach}, 
J. Math. Anal. Appl. {\bfseries 514} (2022), no. 2, Paper No. 126309, 22pp.
\bibitem{Kwa} C. Kwak, {\itshape Local well-posedness for the fifth-order KdV equations on $\mathbb{T}$}, 
J. Differential Equations {\bfseries 260} (2016), no. 10, 7683--7737. 
\bibitem{Kwa2} C. Kwak, {\itshape Low regularity Cauchy problem for the fifth-order modified KdV equations on $\mathbb{T}$}, 
J. Hyperbolic Differ. Equ. {\bfseries 15} (2018), no. 3, 463--557. 
\bibitem{Kwa3} C. Kwak and K. Lee, {\itshape Energy solutions for the fifth-order modified Korteweg de-Vries equations}, 
Discrete Contin. Dyn. Syst. {\bfseries 44} (2024), no.11, 3302--3345.
\bibitem{Kwo} S. Kwon, {\itshape On the fifth order KdV equation: local well-posedness and lack of uniform continuity of the solution map}, 
J. Differential Equations {\bfseries 245} (2008), no.9, 2627--2659. 
\bibitem{Kwo2} S. Kwon, {\itshape Well-posedness and ill-posedness of the fifth-order modified KdV equation}, 
Electron J. Differential Equations {\bfseries 2008} no. 1, 15pp. 
\bibitem{KO} S. Kwon and T. Oh, {\itshape On unconditional well-posedness of modified KdV}, Int. Math. Res. Not. 
{\bfseries 2012} no. 15, 3509--3534. 
\bibitem{KOY} S. Kwon, T. Oh and H. Yoon, {\itshape Normal form approach to unconditional well-posedness 
of nonlinear dispersive PDEs on the real line}, Ann. Fac. Sci. Toulouse Math. {\bfseries 29} (2020), no. 3, 649--720.  
\bibitem{Mc} R. McConnell, {\itshape Well-posedness for the non-integrable periodic fifth order KdV in Bourgain spaces}, Discrete Contin. Dyn. Syst. {\bfseries 44} (2024), no. 11, 3507--3552.
\bibitem{MPV} L. Molinet, D. Pilod and S. Vento, {\itshape On unconditional well-posedness for the periodic modified Korteweg-de Vries equation}, 
J. Math. Soc. Japan {\bfseries 71} (2019), no. 1, 147--201. 
\bibitem{MoP} R. Mosincat and D. Pilod, {\itshape Unconditional uniqueness for the Benjamin-Ono equation}, Pure Appl. Anal. {\bfseries 5} (2023), no.2, 285--322. 
\bibitem{MoY} R. Mosincat and H. Yoon, {\itshape Unconditional uniqueness for the derivative nonlinear Schr\"{o}dinger equation on the real line}, 
Discrete Contin. Dyn. Syst. {\bfseries 40} (2020), no.1, 47--80. 
\bibitem{NTT} K. Nakanishi, H. Takaoka and Y. Tsutsumi, {\itshape Local well-posedness in low regularity of the mKdV equation with periodic boundary condition}, Discrete Contin. Dyn. Syst. {\bfseries 28} (2010), no.4, 1635--1654. 
\bibitem{Ponce}G. Ponce, {\itshape Lax pairs and higher order models for water waves}, J. Differential Equations {\bf 102} (1993), no. 2, 360--381.
\bibitem{Sh} J. Shatah, {\itshape Normal forms and quadratic nonlinear Klein-Gordon equations},
Comm. Pure. Appl. Math {\bfseries 38} (1985), no.5, 685--696. 
\bibitem{TaTs} H. Takaoka and Y. Tsustumi, {\itshape Well-posedness of the Cauchy problem for the modified KdV equation with periodic boundary condition}, Int. Math. Res. Not. {\bfseries 2004} no. 56, 3009--3040.
\bibitem{T} K. Tsugawa, {\itshape Parabolic smoothing effect and local well-posedness of fifth order semilinear dispersive equations on the torus}, arXiv:math/1707.09550v1 [math.AP].
\end{thebibliography}
\end{document}